\documentclass[a4paper,12pt]{amsart}

\usepackage{amsmath,amssymb,amsxtra,amscd,mathrsfs}
\usepackage[margin=3cm]{geometry}
\usepackage{color}  
\usepackage{tikz}
\usepackage{hyperref} 
\usepackage[all]{xy}

\numberwithin{equation}{section}
\allowdisplaybreaks[3]

\newtheorem{theorem}{Theorem}[section]
\newtheorem{lemma}[theorem]{Lemma} 
\newtheorem{proposition}[theorem]{Proposition} 
\newtheorem{corollary}[theorem]{Corollary} 
 
\theoremstyle{definition}
\newtheorem{definition}[theorem]{Definition} 
\newtheorem{notation}[theorem]{Notation} 
\newtheorem{remark}[theorem]{Remark} 
\newtheorem{example}[theorem]{Example}

\newcommand{\C}{\mathbb{C}} 
\newcommand{\Z}{\mathbb{Z}}
 
\newcommand{\N}{\mathbb{N}}

\newcommand{\J}{\mathbb{J}}
\newcommand{\K}{\mathbb{K}}
\newcommand{\LL}{\mathbb{L}}
\newcommand{\PP}{\mathbb{P}} 
\newcommand{\Q}{\mathbb{Q}} 
\newcommand{\R}{\mathbb{R}}
\newcommand{\SSS}{\mathbb{S}}
\newcommand{\T}{\mathbb{T}}
\newcommand{\X}{\mathbb{X}}
\newcommand{\LLeff}{\LL_{\rm eff}}

\newcommand{\bc}{\mathbf{c}} 
 
\newcommand{\f}{\mathbf{f}} 
\newcommand{\bg}{\mathbf{g}} 
\newcommand{\bone}{\mathbf{1}} 
\newcommand{\bp}{\mathbf{p}} 
\newcommand{\bq}{\mathbf{q}} 
 
\newcommand{\bt}{\mathbf{t}}
\newcommand{\btau}{\boldsymbol{\tau}}
\newcommand{\bL}{\mathbf{L}}

\newcommand{\sfL}{\mathsf{L}}	

\newcommand{\tbtau}{\tilde{\btau}}
\newcommand{\te}{\tilde{e}} 
\newcommand{\tf}{\tilde{f}}
\newcommand{\tq}{\tilde{q}}

\newcommand{\vecalpha}{{\vec{\alpha}}} 
\newcommand{\vecbc}{{\vec{\bc}}} 

\newcommand{\veccD}{{\vec{\cD}}}
\newcommand{\veck}{{\vec{k}}} 

\newcommand{\vecr}{{\vec{r}}} 
\newcommand{\vecs}{{\vec{s}}} 
\newcommand{\vecE}{{\vec{E}}}
\newcommand{\vecG}{{\vec{\Gamma}}}
\newcommand{\vecL}{{\vec{L}}}

\newcommand{\vecO}{{\vec{\cO}}}
\newcommand{\vecV}{{\vec{V}}} 
\newcommand{\vecW}{{\vec{W}}}

\newcommand{\cA}{\mathcal{A}}
\newcommand{\cB}{\mathcal{B}}
\newcommand{\cC}{\mathcal{C}}
\newcommand{\cD}{\mathcal{D}}

\newcommand{\cF}{\mathcal{F}} 

\newcommand{\cH}{\mathcal{H}}
\newcommand{\cI}{\mathcal{I}}
\newcommand{\cJ}{\mathcal{J}}
\newcommand{\cL}{\mathcal{L}}
\newcommand{\cM}{\mathcal{M}}

\newcommand{\cO}{\mathcal{O}} 

\newcommand{\cQ}{\mathcal{Q}}
\newcommand{\cS}{\mathcal{S}}
\newcommand{\cT}{\mathcal{T}}
\newcommand{\cU}{\mathcal{U}}
\newcommand{\cW}{\mathcal{W}}
\newcommand{\cX}{\mathcal{X}}

\newcommand{\hcS}{\widehat{\cS}}

\newcommand{\hbS}{\widehat{\SSS}}
\newcommand{\hS}{\widehat{S}}

\newcommand{\ovbeta}{\overline{\beta}}
\newcommand{\ovcM}{\overline{\cM}}

\newcommand{\ab}{{\alpha,\beta}}
\newcommand{\acb}{{\alpha\cup\beta}}
\newcommand{\dab}{d_{\alpha\beta}}
\newcommand{\paba}{p_{\alpha\cup\beta,\alpha}}
\newcommand{\pabb}{p_{\alpha\cup\beta,\beta}}
\newcommand{\KL}{\K}
\newcommand{\TL}{{\T}}		
\newcommand{\TLd}{{\T'}}
\newcommand{\XLO}{\X_\sfL(\vecO)}	
\newcommand{\XLV}{\X_\sfL(\vecV)}
\newcommand{\barA}{\bar{A}}

\newcommand{\adj}{\operatorname{adj}} 
\newcommand{\ch}{\operatorname{ch}} 
 
\newcommand{\eff}{{\operatorname{eff}}}
\newcommand{\ev}{\operatorname{ev}} 
\newcommand{\extEff}{\Eff^{\operatorname{ext}}}
\newcommand{\id}{\operatorname{id}} 
\newcommand{\fix}{\operatorname{fix}}

\newcommand{\loc}{\operatorname{loc}} 
\newcommand{\mrk}{\operatorname{mark}}
\newcommand{\mov}{\operatorname{mov}} 
\newcommand{\pol}{\operatorname{pol}}
\newcommand{\pr}{\operatorname{pr}}
\newcommand{\pt}{\operatorname{pt}} 
\newcommand{\rank}{\operatorname{rank}}
\newcommand{\tw}{\operatorname{tw}}

\newcommand{\val}{\operatorname{val}} 

\newcommand{\vir}{\operatorname{vir}} 
\newcommand{\Bl}{\operatorname{Bl}}
\newcommand{\Cont}{\operatorname{Cont}}
\newcommand{\Def}{\operatorname{Def}}
\newcommand{\Ext}{\operatorname{Ext}}

\newcommand{\Res}{\operatorname{Res}}

\newcommand{\Aut}{\operatorname{Aut}} 
 
\newcommand{\DG}{\operatorname{DG}} 
\newcommand{\Eff}{\operatorname{Eff}} 
 
\newcommand{\Frac}{\operatorname{Frac}} 

\newcommand{\Hom}{\operatorname{Hom}}

\newcommand{\Ker}{\operatorname{Ker}}
\newcommand{\Ob}{\operatorname{Ob}}
\newcommand{\Prin}{\operatorname{Prin}}

\def\corr#1{\left\langle#1 \right\rangle}

\begin{document} 
\title{A mirror theorem for non-split toric bundles}

\author{Yuki Koto} 
\email{ykoto@gate.sinica.edu.tw}
\address{Institute of Mathematics, Academia Sinica, Astronomy-Mathematics Building, No.\ 1, Sec.\ 4, Roosevelt Road, Taipei 10617, Taiwan.}

\begin{abstract} 
We construct an $I$-function for toric bundles obtained as a fiberwise GIT quotient of a (not necessarily split) vector bundle.
This is a generalization of Brown's $I$-function for split toric bundles \cite{Brown} and the $I$-function for non-split projective bundles \cite{IK:quantum}.
In order to prove the mirror theorem, we establish a characterization of points on the Givental Lagrangian cones of toric bundles and prove a mirror theorem for the twisted Gromov-Witten theory of a fiber product of projective bundles.
The former result generalizes Brown's characterization for split toric bundles \cite{Brown} to the non-split case. 
\end{abstract} 

\maketitle 
\tableofcontents

\section{Introduction}
\label{sec:intro}

The genus-zero Gromov-Witten theory of a smooth projective variety $X$ plays a significant role in symplectic geometry, algebraic geometry and mirror symmetry.
It can be studied by a \emph{mirror theorem} \cite{Givental:mirror}, that is, by finding a convenient point (called an \emph{I-function}) on the Givental Lagrangian cone $\cL_X$ \cite{Givental:symplectic}.	
The cone $\cL_X$ is a Lagrangian submanifold of an infinite-dimensional symplectic vector space $\cH_X$, called the Givental space, and is defined by genus-zero gravitational Gromov-Witten invariants. 
A mirror theorem for $X$ enables us to compute genus-zero Gromov-Witten invariants of $X$ and study quantum cohomology.

The $I$-function for a smooth (semi-)projective toric variety $X$ \cite{Givental:mirror,Iritani:shift,CCIT:mirror} can be explicitly described as a hypergeometric series associated with the fan defining $X$.
A relative version of the toric mirror theorem has also been extensively studied.
Let $\LL$ be a free abelian group, $D_1,\dots,D_N$ be elements of $\LL^\vee=\Hom(\LL,\Z)$ and let $\omega\in\LL^\vee\otimes\R$.
If the triple $\sfL=(\LL^\vee,D,\omega)$ is smooth (in the sense of Definition \ref{def:toric}), $\sfL$ determines an embedding of tori $\K=\Hom(\LL^\vee,\C^\times)\hookrightarrow\T=(\C^\times)^N$ and a smooth semi-projective toric variety $X_\sfL=\C^N/\!/_\omega\K$ with a $\T$-action.
For vector bundles $V_1,\dots,V_N$ over a smooth projective variety $B$, we write $\XLV$ for the toric bundle constructed as a fiberwise GIT quotient $(\bigoplus_{i=1}^NV_i)/\!/_\omega\K$.
Note that $\XLV$ is endowed with the $\T$-action induced by the diagonal $\T$-action on $\bigoplus_{i=1}^NV_i$. 
When $V_1,\dots,V_N$ are all line bundles, we call $\XLV$ a \emph{split} toric bundle.

Brown \cite{Brown} proved the mirror theorem for split toric bundles, which had been conjectured by Elezi \cite{Elezi} for the projective bundle case.
Iritani and the author \cite{IK:quantum} constructed an $I$-function for non-split projective bundles.
In this paper, we generalize these works to a (possibly non-split) toric bundle $\XLV$.

\begin{theorem}[Theorem \ref{thm:mirror_thm}]
\label{thm:intro_mirror}
Let $\sfL=(\LL^\vee,D,\omega)$ and $\XLV$ be as above.
We assume that the dual of the bundle $V=\bigoplus_{i=1}^NV_i\to B$ is generated by global sections.
Let $\T$ act on $V$ diagonally, and let $-z I_V^\lambda(-z)$ be a point on the $\T$-equivariant Givental cone of $V$ such that $I_V^\lambda$ depends polynomially on the equivariant parameters $\lambda_1,\dots,\lambda_N$ of $\T$.
Define the $H^*_\T(\XLV)$-valued function $(I_V^\lambda)\sphat\ $ by
\[
(I_V^\lambda)\sphat\ (z) = e^{\sum_{i=1}^Nt_iu_i/z} \sum_{\ell\in\LL} \frac{\tq^\ell e^{\sum_{i=1}^ND_i(\ell)\cdot t_i}}{\prod_{i=1}^N\prod_{c=1}^{D_i(\ell)}\prod_{ \substack{\delta\colon\mathrm{Chern\ roots} \\ \mathrm{of\ } V_i} } (u_i + \delta + cz) } \cdot I_V^{u+D(\ell)z}
\]
where $\tq$ denotes the (extended) Novikov variable for the fiber, $u_i$ denotes the $i$-th $\T$-equivariant relative toric divisor and $I_V^{u+D(\ell)z}$ denotes the function $I_V^\lambda$ with replaced $\mu_i$ with $u_i+D_i(\ell)z$ for $1\leq i\leq N$.
Then, $-z(I_V^\lambda)\sphat\ (-z)$ represents a point on the $\T$-equivariant Givental Lagrangian cone $\cL_{\XLV,\T}$.
\end{theorem}

\begin{remark}
The assumption that $V^\vee$ is globally generated guarantees the semi-projectivity of the total space $\XLV$ (Proposition \ref{prop:semi_proj}).
If the fiber toric variety is proper, this assumption is not a restriction on the toric bundle $\XLV$.
It is because for any set of vector bundles $V_1,\dots,V_N$, by tensoring $V_i$'s with suitable line bundles we can make $V_i$'s satisfy the above condition without changing $\XLV$. 
\end{remark}

\begin{remark}
\label{rem:T/K}
Since the $\K$-action on $\XLV$ is trivial, we can also consider the $\T/\K$-equivariant theory for $\XLV$.
In fact, if we replace $u_i$ in $(I_V^\lambda)\sphat\ (z)$ with the $i$-th $\T/\K$-equivariant relative toric divisor for each $i$, $-z(I_V^\lambda)\sphat\ (-z)$ lies in $\cL_{\XLV,\T/\K}$.
To avoid complicated notation, we will not deal with the $\T/\K$-equivariant setting in this paper except in Appendix \ref{app:Fourier}.
\end{remark}

\begin{remark}
\label{rem:Fourier}
The symbol $(I_V^\lambda)\sphat\ $ represents a ($\T$-equivariant version of) discrete Fourier transformation of $I_V^\lambda$ in the sense of \cite{IK:quantum,Iritani:blowups}.
From this point of view, Theorem \ref{thm:intro_mirror} can be considered as stating that the $\T$-equivariant version of \cite[Conjecture 1.7]{Iritani:blowups} holds true for $\XLV$.

We briefly explain a background of a Fourier transform of Givental cones.
Motivated by Teleman's paper \cite{Teleman:gauge}, Iritani \cite{Iritani:blowups} proposes a conjecture that the quantum $D$-module of a symplectic reduction $X/\!/\K$ can be related to the $\K$-equivariant quantum $D$-module of $X$ via a Fourier transform, which intertwines an equivariant parameter with a quantum connection and a shift operator with an additional Novikov variable.
From this conjecture, we can derive a Fourier transform of Givental cones \cite{IK:quantum,Iritani:blowups}.
In Appendix \ref{app:Fourier}, we will introduce its $\T$-equivariant counterpart and check that the $I$-function $(I_V^\lambda)\sphat\ $ coincides with the Fourier transformation of $I_V^\lambda$.
\end{remark}

Thanks to quantum Riemann-Roch theorem \cite{CG:quantum}, our $I$-function $(I_V^\lambda)\sphat\ $ can be described using only "genus-zero Gromov-Witten invariants of $B$" and "the total Chern class $c^\T(V)$".
Therefore, in particular, the following holds.

\begin{corollary}
All genus-zero descendant Gromov-Witten invariants of $\XLV$ can be computed from those of the base $B$ and the total Chern class $c^{\T}(V)$.
\end{corollary}

Using Theorem \ref{thm:intro_mirror}, we can derive $I$-functions that are already known: Brown's $I$-function for split toric bundles \cite{Brown}, the extended $I$-function \cite{CCIT:mirror} (for toric varieties), and the $I$-function for projective bundles \cite{IK:quantum}.

To prove Theorem \ref{thm:intro_mirror}, we establish a Givental and Brown style characterization of points on $\cL_{\XLV,\T}$, that is, we characterize them in terms of the restrictions to the fixed loci.
We write $F_\sfL$ for the set of $\T$-fixed points on $X_\sfL$.
Since the torus $\T$ acts on $\XLV$ fiberwise, for any $\alpha\in F_\sfL$ we can construct the $\T$-fixed locus $\iota_\alpha\colon\XLV_\alpha\hookrightarrow\XLV$ by gathering the points corresponding to $\alpha$ in each fiber, and hence the $\T$-fixed loci of $\XLV$ can be indexed by $F_\sfL$.
It can be easily seen that $\XLV_\alpha$ is a fiber product of projective bundles over $B$, and $\XLV_\alpha \cong B$ for any $\alpha\in F_\sfL$ if $\XLV$ is a split toric bundle.
We can recover a point $\f$ on the Givental space $\cH_{\XLV,\T}$ from its restrictions $\{\iota_\alpha^*\f\}_{\alpha\in F_\sfL}$ by the (classical) localization formula \cite{AB,BV}.
The following theorem provides an equivalent condition for $\f$ to be a point on $\cL_{\XLV,\T}$ in terms of $\{\iota_\alpha^*\f\}_{\alpha\in F_\sfL}$.

\begin{theorem}[Theorem \ref{thm:characterization}]
\label{thm:intro_characterization}
Let $\f$ be a point on the Givental space $\cH_{\XLV,\T}$ for the total space $\XLV$.
The point $\f(z)$ lies in $\cL_{\XLV,\T}$ if and only if $\{\iota_\alpha^*\f\}_{\alpha\in F_\sfL}$ satisfies the three conditions: 
\begin{itemize}
\item[\textbf{(C1)}] for each $\alpha\in F_\sfL$, the set of poles of $\iota_\alpha^*\f(z)$ as a function in $z$ is contained in a specific subset of $H^*_\T(\pt,\Q)$ determined by the triple $\sfL=(\LL^\vee,D,\omega)$;
\item[\textbf{(C2)}] the principle parts of the functions $\iota_\alpha^*\f(z)$ $(\alpha\in F_\sfL)$ satisfy certain recursion formulas;
\item[\textbf{(C3)}] $\iota_\alpha^*\f(z)$ represents a point on the Givental cone of the fixed locus $\XLV_\alpha$ twisted by the normal bundle and the inverse Euler class.
\end{itemize}
\end{theorem}

This is a generalization of Brown's result \cite[Theorem 2]{Brown}, which gives the same characterization for split toric bundles.
There are also similar characterization results for other varieties/stacks; see \cite{CCIT:mirror,JTY,FL}.

Thanks to this characterization theorem, we can prove Theorem \ref{thm:intro_mirror} by checking that the function $-z(I_V^\lambda)\sphat\ (-z)$ satisfies the three conditions.
Note that we can confirm that the function fulfills \textbf{(C1)} and \textbf{(C2)} through a direct calculation.
The verification of \textbf{(C3)} requires another mirror theorem for the twisted Gromov-Witten theory of $\XLV_\alpha$, a fiber product of projective bundles over $B$.
This is a new issue in the non-split case since $\XLV_\alpha\cong B$ if $\XLV$ is a split bundle.

Let $V_1,\dots,V_K,V_{K+1},\dots,V_N$ be vector bundles over a smooth projective variety $B$, and let $E$ be the fiber product of the projective bundles $\PP(V_1),\dots,\PP(V_K)$ over $B$. 
Let $\LL \cong \Z^K$ and let $D_1,\dots,D_N\in\LL^\vee$ be such that $\{D_i\}_{i=1}^K$ forms a basis of $\LL^\vee$.
We set $\cW$ to be the vector bundle $\bigoplus_{i=K+1}^N V_i\otimes\cO_E(D_i)$ over $E$ with $\cO_E(D_i) = \cO_{\PP(V_1)}(a_{i,1})\boxtimes_B\cdots\boxtimes_B\cO_{\PP(V_K)}(a_{i,K})$ where $D_i=\sum_{j=1}^K a_{i,j}D_j$.
Note that $\cW$ is a fiberwise GIT quotient of $V=\bigoplus_{i=1}^NV_i$, and is endowed with the fiberwise $\T$-action induced from that on $V$.
The bundle $\cW$ arises as a normal bundle to the fixed loci of toric bundles.
The following theorem implies that $-z(I_V^\lambda)\sphat\ (-z)$ satisfies \textbf{(C3)}.

\begin{theorem}[Theorem \ref{thm:twist_mirror}]
\label{thm:intro_twist} 
Assume that the dual of the bundle $\bigoplus_{i=1}^NV_i$ is generated by global sections.
Let $I_V^\lambda$ be as in Theorem \ref{thm:intro_mirror}.
Define the $H^*_\T(E)$-valued function $(I_V^\lambda)_{\tw}\sphat$ by
\[
(I_V^\lambda)_{\tw}\sphat(z) = e^{\sum_{i=1}^Nt_iu_i/z} \sum_{\ell\in\LL} \frac{\tq^\ell e^{\sum_{i=1}^ND_i(\ell)\cdot t_i}}{\prod_{i=1}^N\prod_{c=1}^{D_i(\ell)}\prod_{ \substack{\delta\colon\mathrm{Chern\ roots} \\ \mathrm{of\ } V_i} } (u_i + \delta + cz) } \cdot I_V^{u+D(\ell)z}
\]
where $\tq$ denotes the (extended) Novikov variable for the fiber, and $u_i\in H^2_\T(E)$ is defined as 
\[
u_i = 
\begin{cases}
c_1(\cO_{\PP(V_i)}(1))					&1\leq i\leq K,	\\
-\lambda_i+\sum_{j=1}^Ka_{i,j}(u_j+\lambda_j)	&K+1\leq i\leq N.
\end{cases}
\]
Then, $-z(I_V^\lambda)_{\tw}\sphat(-z)$ repressents a point on the $(\cW,e_{\T}^{-1})$-twisted Givental Lagrangian cone of $E$.
\end{theorem}

\begin{remark}
\label{rem:bundle}
In Theorem \ref{thm:twist_mirror} (and Section \ref{sec:twist_mirror}), we will use different notation from here.
There, we write $W_i$ for $V_i$ ($K+1\leq i\leq N$) here in order to distinguish between the vector bundles $V_1,\dots,V_K$ giving a variety $E$ and those $V_{K+1},\dots,V_N$ providing the twist data $(\cW,e_\T^{-1})$.
Additionally, we will denote the torus acting on $V$ as $\T'$ to distinguish it from the torus $\T$ acting on $\XLV$; see Remark \ref{rem:T'}.
\end{remark}

This result is a straightforward generalization of the mirror theorem for non-split projective bundles \cite[Theorem 3.3]{IK:quantum}.
The key ingredient of the proof is the quantum Riemann-Roch theorem \cite[Corollary 4]{CG:quantum} and the well-known fact \cite{KKP:functoriality} that Gromov-Witten invariants of the zero locus of a regular section of a convex vector bundle over a variety $X$ are given by twisted Gromov-Witten invariants of $X$.

The plan of the paper is as follows.
In Section \ref{sec:GW}, we recall the definition of Gromov-Witten invariants, and introduce the non-equivariant/equivariant/twisted Givental cones and quantum Riemann-Roch theorem. 
In Section \ref{sec:toric_bundle}, we introduce the notion of split/non-split toric bundles, and summarize the structure of cohomology and the semigroups generated by effective curve classes, which will be needed in the subsequent sections.
In Section \ref{sec:characterization}, we establish a characterization theorem (Theorem \ref{thm:characterization}) for points on the Lagrangian cone of a toric bundle.
In Section \ref{sec:twist_mirror}, we prove a mirror theorem for twisted Gromov-Witten theory of a fiber product of projective bundles over $B$.
In Section \ref{sec:main_result}, we prove the main result (Theorem \ref{thm:mirror_thm}) of this paper, that is, a mirror theorem for (possibly non-split) toric bundles.

\phantom{A}

\noindent
\textbf{Acknowledgements.}
The author is deeply grateful to Hiroshi Iritani for his guidance and enthusiastic support during the writing of this paper.
He also would like to thank Yuan-Pin Lee and Fumihiko Sanda for very helpful discussions.
This work was supported by JSPS KAKENHI Grant Number 22KJ1717.

\section{Genus-zero Gromov-Witten theory}
\label{sec:GW}
In this section, we briefly recall the (torus-equivariant/twisted) genus-zero Gromov-Witten theory. 
We will introduce Gromov-Witten invariants, Givental Lagrangian cones and the quantum Riemann-Roch theorem.

\subsection{Gromov-Witten invariant and its variants}
We recall the definition of Gromov-Witten invariant.
We also introduce a torus-equivariant version and a twisted version of it.

\subsubsection{Gromov-Witten invariants}
Let $X$ be a smooth projective variety over $\C$.
For $d\in H_2(X,\Z)$ and a non-negative integer $n$, let $X_{0,n,d}$ be the moduli space of degree-$d$ stable maps to $X$ from genus-zero curves with $n$ marked points. 
It is endowed with the evaluation maps $\ev_i\colon X_{0,n,d}\to X$ $(1 \leq i \leq n)$.
For $\alpha_1,\dots,\alpha_n\in H^*(X)$ and $k_1,\dots,k_n\in\Z_{\geq0}$, we define \emph{genus-zero descendant Gromov-Witten invariants} by
\[
\corr{\alpha_1\psi^{k_1},\dots,\alpha_n\psi^{k_n}}^X_{0,n,d} := \int_{[X_{0,n,d}]^{\vir}}\prod^n_{i=1}\ev^*_i(\alpha_i)\psi_i^{k_i},
\]
where $[X_{0,n,d}]^{\vir}\in H_*(X_{0,n,d},\Q)$ is the virtual fundamental class \cite{Behrend} and $\psi_i$ is the $\psi$-class, which is the first Chern class of the $i$-th universal cotangent line bundle over $X_{0,n,d}$.

\subsubsection{Virtual localization}
\label{subsubsec:vir_loc}
Let $\T$ be an algebraic torus and $X$ be a smooth projective variety endowed with a $\T$-action.
Let $F_1,\dots,F_m$ be the connected components of $X_{0,n,d}^\T$, and let $\iota_{F_i}$ be the embedding $F_i\hookrightarrow X_{0,n,d}$ for $1\leq i\leq m$.
For $\xi\in F$, let $T^1$ be a tangent space and $T^2$ be an obstruction space at $\xi$.
These spaces are naturally equipped with a $\T$-action, and the action induces the decomposition of them into the fixed parts $T^{j,\fix}$ and the moving parts $T^{j,\mov}$.
By collecting the space $T^{j,\mov}$ we have the bundle $\cT^{j,\mov}_{F_i}$ over $F_i$.
We define the \emph{virtual normal bundle of $F_i$} as $N_{F_i}^{\vir} = \cT^{1,\mov}_{F_i} \ominus \cT^{2,\mov}_{F_i}$.

\begin{theorem}[\cite{GP}]
\label{thm:vir_loc}
In the above setting, we have
\[
[X_{0,n,d}]^{\vir} = \sum_{i=1}^m \iota_{F_i*} \left( \frac{[F_i]^{\vir}}{e_\T(N^{\vir}_{F_i})} \right).
\]
In particular, for any $\phi\in H^*_\T(X_{0,n,d},\C)$, we have
\[
\int^\T_{[X_{0,n,d}]^{\vir}}\phi = \sum_{i=1}^m \int^\T_{[F_i]^{\vir}} \frac{\iota_{F_i}^*\phi}{e_\T(N_{F_i}^{\vir})}.
\]
\end{theorem}

\subsubsection{Equivariant Gromov-Witten invariants}
We discuss Gromov-Witten invariants in the equivariant setting.
Let $X$ be a smooth variety, and assume that $X$ is semi-projective, that is, $X$ is projective over an affine variety.
Let $\T$ be an algebraic torus and consider a $\T$-action on $X$ whose fixed point set is projective.
This action naturally induces an $\T$-action on the moduli space $X_{0,n,d}$.

For $\alpha_1,\dots,\alpha_n\in H^*_\T(X)$ and $k_1,\dots,k_n\in\Z_{\geq0}$, we define \emph{$\T$-equivariant genus-zero descendant Gromov-Witten invariants} by
\[
\corr{\alpha_1\psi^{k_1},\dots,\alpha_n\psi^{k_n}}^{X,\T}_{0.n,d} := \int^\T_{[X_{0,n,d}]^{\vir}}\prod^n_{i=1}\ev^*_i(\alpha_i)\psi_i^{k_i}.
\]
Here we define the right-hand side via the virtual localization formula (Corollary \ref{thm:vir_loc}), and hence it belongs to the fraction field $\Frac(H^*_\T(\pt))$.
When $X$ is projective, it can be computed without the localization formula and belongs to $H^*_\T(\pt)$.

\subsubsection{Twisted Gromov-Witten invariants}
\label{subsubsec:twisted_inv}
Let $\T$ be an algebraic torus.
For any $\chi\in H^2_\T(\pt,\Z)\setminus0$, we introduce the following four characteristic classes $e_\chi,e_\chi^{-1},\te_\chi,\te_\chi^{-1}$:  
\begin{align*}
e_\chi(V)
&= \sum_{i=0}^{\rank(V)} \chi^{\rank(V)-i} c_i(V),	\\
e_\chi^{-1}(V)
&= (e_\chi(V))^{-1},	\\
\te_\chi(V)
&= \chi^{-\rank(V)} e_\chi(V) = \sum_{i=0}^{\rank(V)} \chi^{-i} c_i(V),	\\
\te_\chi^{-1}(V)
&= (\te_\chi(V))^{-1}	
\end{align*}
where $V$ is a vector bundle over any topological space.
We note that, if $\T$ acts on $V$ fiberwise via the character $\chi\colon\T\to\C^\times$, the class $e_\chi(V)$ coincides with the $\T$-equivariant Euler class $e_\T(V)$.

Let $X$ be a smooth projective variety which is endowed with a trivial $\T$-action.
Let $W_i$, $i=1,\dots,N$, be vector bundles over $X$ and let  $\bc^i$, $i=1,\dots,N$, be one of the characteristic classes $e_{\chi_i}, e_{\chi_i}^{-1}, \te_{\chi_i}, \te_{\chi_i}^{-1}$ associated with a character $\chi_i\colon \T \to \C^\times$. 
The collection
\[
(W_1,\bc^1),\dots,(W_N,\bc^N)
\]
is referred to as \emph{twist data}.
We sometimes use the vector symbol $(\vecW,\vecbc)$ as an abbreviation for the twist data above.

Let $\pi\colon X_{0,n+1,d}\to X_{0,n,d}$ be the map forgetting the $(n+1)$-st point and $\ev_{n+1}\colon X_{0,n+1,d}\to X$ be the $(n+1)$-st evaluation map:
\[
\xymatrix{
X_{0,n+1,d} \ar[r]^{\ev_{n+1}} \ar[d]_\pi & X \\
X_{0,n,d}
}
\]
Note that these maps give the universal family of stable maps.
For a vector bundle $W$ over $X$, we define
\[
W_{0,n,d} := \R\pi_*\ev_{n+1}^* W \in K^0_\T(X_{0,n,d})
\]
where $\R\pi_*$ denotes the $K$-theoretic pushforward.
For $\alpha_1,\dots,\alpha_n\in H^*_{\T}(X)$ and $k_1,\dots,k_n\in\Z_{\geq0}$, we define \emph{genus-zero descendant Gromov-Witten invariants twisted by $(W_1,\bc^1),\dots,(W_N,\bc^N)$} \cite{CG:quantum} by
\[
\corr{\alpha_1\psi^{k_1},\dots,\alpha_n\psi^{k_n}}^{X,(\vec{W},\vec{\bc})}_{0.n,d} := \int^\T_{[\cX_{0,n,d}]^{\vir}}\left(\prod^n_{i=1}\ev^*_i(\alpha_i)\psi_i^{k_i}\right) \cdot \left(\prod^N_{i=1} \bc^i((W_i)_{0,n,d})\right),
\]
which take values in $\Frac(H^*_\T(\pt))$.
We will discuss in Subsection \ref{subsec:QRR} the relationship between the twisted theory and the untwisted theory in terms of their Lagrangian cones, which are introduced in the next subsection.

Finally, we introduce the notation $(W,e_\T^{\pm1})$.
This is useful when we treat the twist data coming from a $\T$-vector bundle $W$. 

\begin{definition}
Let $W$ be a vector bundle over $X$ with a fiberwise $\T$-action such that there are no $\T$-fixed non-zero vectors.
We write $W = \bigoplus_{i=1}^N W_i$ for the weight decompostion, and let $\chi_i\in H^2_\T(\pt)\setminus0$ be the character associated with $\T$-action on $W_i$.
We define $(W,e_\T^{\pm1})$ to be the twist data $(W_1,e_{\chi_1}^{\pm1}),\dots,(W_N,e_{\chi_N}^{\pm1})$.
\end{definition}

\subsection{Givental Lagrangian cone}
\label{subsec:Givental_cone}
Following \cite{Givental:symplectic} \cite{CG:quantum}, we introduce a symplecto-geometric formulation of genus-zero Gromov-Witten theory.
In this formulation, many properties of Gromov-Witten invariants are translated into geometric properties of a certain Lagrangian cone in an infinite dimensional symplectic vector space.
We will introduce three variations of Lagrangian cones.

\subsubsection{Non-equivariant case}
\label{subsubsec:Lagrangian_cone1}
Let $X$ be a smooth projective variety.
Let $(\cH_X,\Omega)$ be a symplectic vector space:
\begin{align*}
\cH_X :=& H^*(X)[z,z^{-1}][\![\Eff(X)]\!],	\\
\Omega(f,g) :=& -\Res_{z=\infty} \left(\int_X f(-z)\cup g(z)\right) dz \qquad \text{for } f,g\in\cH_X
\end{align*}
where $\Eff(X)\subset H_2(X,\R)$\footnote{$\Eff(X)$ contains the image of the semigroup in $H_2(X,\Z)$ generated by the effective curve classes under the map $H_2(X,\Z)\to H_2(X,\R)$. In general, $\Eff(X)$ does not equal the image.} denotes the intersection of the cone generated by effective curve classes and the image of $H_2(X,\Z)\to H_2(X,\R)$, and $\C[\![\Eff(X)]\!]$ is the completion of $\C[\Eff(X)]$ with respect to the additive valuation $v$ induced by the semigroup homomorphism $\omega\colon\Eff(X)\to\Z$ given by a K\"{a}hler form $\omega\in H^2(X,\Z)$:
\begin{align*}
\omega\colon \Eff(X)\to\Z, \qquad d\mapsto\omega(d):=\int_d\omega,	\\
v\left(\sum_{d\in\Eff(X)}c_dQ^d\right) = \min_{d\colon c_d\neq0} \omega(d)
\end{align*}
where $Q$ denotes a formal variable for the group ring $\C[\Eff(X)]$ called the \emph{Novikov variable}.
There is a standard polarization $\cH_X=\cH_+\oplus\cH_-$, where
\begin{align*}
\cH_+ :=& H^*(X)[z][\![\Eff(X)]\!], \\
\cH_- :=& z^{-1}H^*(X)[z^{-1}][\![\Eff(X)]\!].
\end{align*}
These are $\Omega$-isotropic subspaces and $\cH$ can be identified with the cotangent bundle $T^*\cH_+$.

Let $\bt(z)=\sum_{i\geq0}t_iz^i\in H^*(X)[z]$.
We define the \emph{genus-zero descendant potential} $\cF^0_X$ as
\[
\cF^0_X(\bt) := \sum_{ \substack{n\geq0,d\in\Eff(X) \\ (n,d)\neq(0,0),(1,0),(2,0)} } \frac{Q^d}{n!} \corr{\bt(\psi),\dots,\bt(\psi)}_{0,n,d}^X.
\]
Under a shift $\bq(z)=\bt(z)-\bone z$, which is called the \emph{dilaton shift}, we consider $\cF^0_X$ as a formal function on $\cH_+$ in a formal neighborhood of $-\bone z\in\cH_+$.
The Givental Lagrangian cone $\cL_X$ \cite{Givental:symplectic} is a formal germ of a Lagrangian submanifold of $\cH_X$ defined as the graph of the differential of $\cF^0_X$.

We can give an explicit description for points on $\cL_X$.
Let $\{\phi_i\}_{i\in I}$ be a $\C$-basis of $H^*(X)$, $\{\phi^i\}_{i\in I}$ be the dual basis with respect to the Poincar\'{e} pairing.
We let $S$ be a semigroup such that there exists a semigroup homomorphism $i\colon\Eff(X)\hookrightarrow S$ and a valuation $v_S\colon \C[S]\to\Z_{\geq0}$ which extends the valuation $v$ on $\C[\Eff(X)]$ via the inclusion $i$.
In this situation, for any ring $R$ we can define the completion $R[\![S]\!]$ of $R[S]$ by using the valuation $v_S$.
A $\C[\![S]\!]$-valued point on $\cL_X$ is of the form
\[
-\bone z + \bt(z) + \sum_{ \substack{n\geq0,d\in\Eff(X) \\ (n,d)\neq(0,0),(1,0)} } \sum_{i\in I} \frac{Q^d}{n!} \corr{\frac{\phi_i}{-z-\psi},\bt(\psi),\dots,\bt(\psi)}_{0,n+1,d}^X\phi^i
\]
with $\bt(z)\in\cH_+\otimes_{\C[z][\![\Eff(X)]\!]}\C[z][\![S]\!]$ satisfying $\bt(z)|_{Q=0}=0$.
Here we expand $1/(-z-\psi)$ at $z=\infty$:
\[
\frac{1}{-z-\psi} = \sum_{k\geq0} (-z)^{-k-1} \psi^k.
\] 
Since $\psi_1$ is nilpotent, this form indeed belongs to $\cH_X\otimes_{\C[z][\![\Eff(X)]\!]}\C[z][\![S]\!]$.

\begin{example}
We give a typical example of $S$.
Set $S = \Eff(X) \oplus (\Z_{\geq0})^\N$ and write $\{e_n\}_{n\in\N}$ for the standard basis of $(\Z_{\geq0})^\N$.
Define $\omega_S\colon S\to\Z_{\geq0}$ as a semigroup homomorphism satisfying that $\omega_S|_{\Eff(X)}$ coincides with the homomorphism $\omega\colon\Eff(X)\to\Z_{\geq0}$ and $\omega_S(e_i)=i$.
In this case, we have
\[
\C[\![S]\!] = \C[\![\Eff(X)]\!][\![t_1,t_2,\cdots]\!]
\]
where $t_i$ denotes a formal variable for $\C[S]$ associated with $e_i\in S$.
\end{example}

\begin{remark}
The Lagrangian cone $\cL_X$ (and its variants introduced in the following subsections) can be formulated as a formal scheme \cite{CCIT:computing}.
However, in this thesis, it is sufficient to consider a set of valued points on $\cL_X$.
\end{remark}

\subsubsection{$\T$-equivariant case}
\label{subsubsec:equiv_cone}
We let $\T$ be an algebraic torus, and let $X$ be a smooth semi-projective variety endowed with a $\T$-action whose fixed point set is projective.
We will define the Lagrangian cone in a similar way. 
We let $\lambda_1,\dots,\lambda_N$ denote the equivariant parameters and write $\C[\lambda]$ and $\C(\lambda)$ for $H^*_\T(\pt)$ and $\Frac(H^*_\T(\pt))$ respectively.
We set
\begin{align*}
\cH_{X,\T}		:=&\ H^*_{\T}(X)_{\loc}(\!(z^{-1})\!)[\![\Eff(X)]\!],	\\
\cH_+			:=&\ H^*_{\T}(X)_{\loc}[z][\![\Eff(X)]\!],	\\
\cH_-			:=&\ z^{-1}H^*_{\T}(X)_{\loc}[\![z^{-1}]\!][\![\Eff(X)]\!],	\\
\Omega(f,g)		:=&\ -\Res_{z=\infty} \left(\int_X^\T f(-z)\cup g(z)\right) dz \qquad \text{for } f,g\in\cH_{X,\T}
\end{align*}
where $H^*_{\T}(X)_{\loc}:=H^*_{\T}(X)\otimes_{\C[\lambda]}\C(\lambda)$.
There is a standard polarization $\cH_{X,\T}=\cH_+\oplus\cH_-$, and we identify $\cH_{X,\T}$ with $T^*\cH_+$.

Let $\bt(z)=\sum_{i\geq0}t_iz^i\in H^*_{\T}(X)[z]$.
We define \emph{the $\T$-equivariant genus-zero descendant potential} $\cF^0_{\cX,\T}$ as
\[
\cF^0_{X,\T}(-\bone z + \bt(z)) := \sum_{ \substack{n\geq0,d\in\Eff(X) \\ (n,d)\neq(0,0),(1,0),(2,0)} } \frac{Q^d}{n!} \corr{\bt(\psi),\dots,\bt(\psi)}_{0,n,d}^{X,\T},
\]
which is defined over a formal neighborhood of $-\bone z\in\cH_+$.
The equivariant Givental Lagrangian cone $\cL_{X,\T}$ is defined as the graph of the differential of $\cF^0_{X,\T}$.
Let $\{\phi_i\}_{i\in I}$ be a basis of $H^*_{\T}(X)_{\loc}$ over $\Frac(H^*_\T(\pt))$, $\{\phi^i\}$ be the dual basis with respect to the Poincar\'{e} pairing, and let $S$ and $\C[\![S]\!]$ be as in the previous subsection.
A $\C(\lambda)[\![S]\!]$-valued point on $\cL_X$ is of the form
\begin{equation}
\label{eqn:equiv_cone}
-\bone z + \bt(z) + \sum_{ \substack{n\geq0,d\in\Eff(X) \\ (n,d)\neq(0,0),(1,0)} } \sum_{i\in I} \frac{Q^d}{n!} \corr{\frac{\phi_i}{-z-\psi},\bt(\psi),\dots,\bt(\psi)}_{0,n+1,d}^{X,\T}\phi^i
\end{equation}
with $\bt(z)\in\cH_+\otimes_{\C(\lambda)[z][\![\Eff(X)]\!]}\C(\lambda)[z][\![S]\!]$ satisfying $\bt(z)|_{Q=0}=0$.
We note that since $\psi_1$ is not nilpotent in general, this form may not be an element of $H^*_\T(X)_{\loc}[z,z^{-1}][\![S]\!]$. 
The semi-projectivity of $X$ implies that $\ev_1$ is proper, and it holds that
\[
\sum_{i\in I} \corr{\frac{\phi_i}{-z-\psi},\bt(\psi),\dots,\bt(\psi)}_{0,n+1,d}^{X,\T}\phi^i = {\ev_1}_* \left[ \frac{\prod_{i=2}^{n+1}\ev_i^*\bt(\psi_i)}{-z-\psi_1} \cap \left[ X_{0,n+1,d} \right]^{\vir} \right].
\]
From this equation we can see that the form \eqref{eqn:equiv_cone} lies on $H^*_\T(X)(\!(z^{-1})\!)[\![S]\!]$ if $\bt(z)\in H^*_\T(X)[z][\![S]\!]$. 

\begin{remark}
By interpreting $z$ as an equivariant parameter of $\C^\times$ and considering $\T\times\C^\times$-equivariant Gromov-Witten theory of $X$ (where the second factor $\C^\times$ acts trivially on $X$), we can interprete the form \eqref{eqn:equiv_cone} as an element of $H^*_{\T\times\C^\times}(X)_{\loc}[\![S]\!]$.
If we take the Laurent expansion at $z=\infty$, we can recover \eqref{eqn:equiv_cone}.
In Section \ref{sec:characterization}, we will consider the Laurent expansion of the form at $z=0$.
\end{remark}

\subsubsection{Twisted case}
We take $X$, $\T$ and the data $(\vecW,\vecbc)$ as in Subsection \ref{subsubsec:twisted_inv}, and consider the $(\vecW,\vecbc)$-twisted theory.
The construction is almost the same as the previous ones.
We set
\begin{align*}
\cH_{X,(\vecW,\vecbc)}					:=&\ H^*_\T(X)_{\loc} (\!(z)\!) [\![\Eff(X)]\!],	\\
\cH_+								:=&\ H^*_\T(X)_{\loc} [\![z]\!] [\![\Eff(X)]\!],	\\
\cH_-								:=&\ z^{-1}H^*_\T(X)_{\loc} [z^{-1}] [\![\Eff(X)]\!],	\\
\Omega(f,g)							:=&\ -\Res_{z=\infty} \left(\int_X^\T f(-z)\cup g(z)\cup \prod_{i=1}^N \bc^i(W_i)\right) dz,	\\
\cF^0_{X,(\vecW,\vecbc)}(-\bone z + \bt(z))	:=&\ \sum_{ \substack{n\geq0,d\in\Eff(X) \\ (n,d)\neq(0,0),(1,0),(2,0)} } \frac{Q^d}{n!} \corr{\bt(\psi),\dots,\bt(\psi)}_{0,n,d}^{X,(\vecW,\vecbc)}.
\end{align*}
Then there is a canonical polarization $\cH_{X,(\vecW,\vecbc)} = \cH_+\oplus\cH_-$, and we can obtain the Lagrangian cone $\cL_{X,(\vecW,\vecbc)}$ whose $\C(\lambda)[\![S]\!]$-valued point is of the form
\[
-\bone z + \bt(z) + \sum_{ \substack{n\geq0,d\in\Eff(X) \\ (n,d)\neq(0,0),(1,0)} } \sum_{i\in I} \frac{Q^d}{n!} \corr{\frac{\phi_i}{-z-\psi},\bt(\psi),\dots,\bt(\psi)}_{0,n+1,d}^{X,(\vecW,\vecbc)}\cdot \frac{\phi^i}{\prod_{i=1}^N \bc^i(W_i)}
\]
with $\bt(z)\in\cH_+\otimes_{\C(\lambda)[\![z]\!][\![\Eff(X)]\!]}\C(\lambda)[\![z]\!][\![S]\!]$ satisfying $\bt(z)|_{Q=0}=0$.
Here we use the notations in the previous subsection.
Since $\T$ acts trivially on $X$, $\psi_1$ is nilpotent and hence this function belongs to $\cH_{X,(\vecW,\vecbc)}\otimes_{\C(\lambda)(\!(z)\!)[\![\Eff(X)]\!]}\C(\lambda)(\!(z)\!)[\![S]\!]$.
As in the untwisted case, the above function equals
\begin{multline}
\label{eqn:twist_push}
-\bone z + \bt(z) + \sum_{ \substack{n\geq0,d\in\Eff(X) \\ (n,d)\neq(0,0),(1,0)} } \frac{Q^d}{n!} \cdot \prod_{i=1}^N\bc^i(W_i)^{-1} \\
\cdot {\ev_1}_* \left[ \frac{\prod_{i=2}^{n+1}\ev_i^*\bt(\psi_i)}{-z-\psi_1} \cdot \prod_{i=1}^N \bc^i((W_i)_{0,n+1,d}) \cap \left[ X_{0,n+1,d} \right]^{\vir} \right].
\end{multline}

\subsection{Quantum Riemann-Roch theorem and twisted theory}
\label{subsec:QRR}
We introduce quantum Riemann-Roch theorem \cite[Corollary 4]{CG:quantum}, which relates twisted Givental cones via some transcendental operators. 
We also explain relationships between the Gromow-Witten theory of a vector bundle (resp. a subvariety) and that of a base space (resp. an ambient space) in terms of twisted theories.
Note that we will use the material in this subsection only in Section \ref{sec:twist_mirror}.

\subsubsection{Quantum Riemann-Roch operator}
\label{subsubsec:QRR}
We let $X$ be a smooth projective variety and $\T$ be a complex torus acting on $X$ trivially. 
For $\chi\in H^2_\T(\pt)\setminus0$, we set
\[
s_k^\pm(\chi) =
\begin{cases}
0					&\text{for } k=0,	\\
\pm(-1)^{k-1}(k-1)!\chi^{-k}	&\text{for } k>0.
\end{cases}
\]
It is easy to see that 
\[
\te_\chi^\pm(\cdot) = \exp\left(\sum_{k=0}^\infty s_k^\pm(\chi) \cdot \ch_k(\cdot)\right).
\]

We take a vector bundle $V$ over $X$ and let $\chi$ be a non-zero element of $H^2_\T(\pt)$.
For the case $\bc=\te_\chi$ or $\bc=\te_\chi^{-1}$, we define the \emph{quantum Riemann-Roch operator} $\Delta_{(V,\bc)}(-z)$ as follows:
\[
\Delta_{(V,\bc)}(-z) := \exp\left[ \sum_{l,m\geq0} s^\pm_{l+m-1}(\chi) \cdot \frac{B_m}{m!} \cdot \ch_l(V) \cdot z^{m-1} \right]
\]
where we set $s_{-1}=0$ and $B_m$ is the Bernoulli number defined by $\sum_{m=0}^\infty (B_m/m!) x^m = x/(e^x-1)$.
Since $\ch_l(V)$ is nilpotent for $l>0$ and $\ch_l(V)=0$ for $l>\dim X$, the operator $\Delta_{(V,\bc)}(\lambda,z)$ is well-defined and belongs to $H^*(X)[\chi^{-1}](\!(z)\!)\cap H^*(X)[z,z^{-1}][\![\chi^{-1}]\!]$.\footnote{For any $\C$-algebra $R$ and a specific non-zero element $\chi\in H^2_\T(\pt)$, we set $R[\chi^{-1}]$ to be a subring of $R(\lambda)$ consisting of elements of finite sums $\sum_{i=0}^k r_i \chi^{-i}$ where $k$ is a non-negative integer and $r_i\in R$ for any $i$, and $R[\![\chi^{-1}]\!]$ to be the canonical completion of $R[\chi^{-1}]$.}
For vector bundles $V_1,\dots,V_N$ and characteristic classes $\bc^1,\dots,\bc^N$ with each class $\bc^i$ being $\te_{\chi_i}$ or $\te_{\chi_i}^{-1}$ for some $\chi_i\in H^2_\T(\pt)\setminus0$, we define the operator $\Delta_{(\vecV,\vecbc)}(z)$ as a product $\prod_{i=1}^N\Delta_{(V_i,\bc^i)}(z)$.

\begin{theorem}[{\cite[Corollary 4]{CG:quantum}},{\cite[Theorem 1.1]{Tonita}}]
Let $(\vecV,\vecbc)$ be twist data with the characteristic class $\bc^i$ $(1\leq i\leq N)$ being of the form $\te_{\chi_i}$ or $\te_{\chi_i}^{-1}$ for some non-zero element $\chi_i\in H^2_\T(\pt)$.
Then we have 
\[
\Delta_{(\vecV,\vecbc)}(-z)\cL_{X,(\cO_X,1)} = \cL_{X,(\vecV,\vecbc)}.
\]
In particular, for any $\C(\lambda)[\![\Eff(X)]\!][\![t]\!]$-valued point $\f$ on $\cL_{X,(\cO_X,1)}$, the function $\Delta_{(\vecV,\vecbc)}(-z)\cdot\f$ is a $\C(\lambda)[\![\Eff(X)]\!][\![t]\!]$-valued point on $\cL_{X,(\vecV,\vecbc)}$.
\end{theorem}

\begin{remark}
The twist data $(\cO_X,1)$ is trivial, that is, the $(\cO_X,1)$-twisted Gromov-Witten invariants are by definition the untwisted ones.
However, there is a subtle difference between $\cL_{X,\T}$ and $\cL_{X,(\cO_X,1)}$.
In fact, $\cL_{X,\T}$ is defined in $\cH_{X,\T}=H^*_\T(X)_{\loc}(\!(z^{-1})\!)[\![\Eff(X)]\!]$, while $\cL_{X,(\cO_X,1)}$ is defined in $\cH_{X,(\cO_X,1)}=H^*_\T(X)_{\loc}(\!(z)\!)[\![\Eff(X)]\!]$.
\end{remark}

We now consider the inverse Euler twist.
For any $\C(\lambda)[\![\Eff(X)]\!][\![t]\!]$-valued point $\f$ on $\cL_{X,(W,\te_{\chi}^{-1})}$, $\f|_{Q^d\to Q^d\chi^{-d\cdot c_1(W)}}$ is a $\C(\lambda)[\![\Eff(X)]\!][\![t]\!]$-valued point on $\cL_{X,(W,e_\chi^{-1})}$: this follows from the definition of twisted cones and the fact that $\rank((W)_{0,n,d}) = \rank(W) + \int_dc_1(W)$.
We define the operator $\Delta_W^\chi$ on $H^*_\T(X)_{\loc} (\!(z)\!) [\![\Eff(X)]\!][\![t]\!]$ as follows:
\begin{equation}
\label{eqn:modified_QRR}
\Delta_W^\chi\f := \left.\left(\Delta_{(W,\te^{-1}_\chi)}(z)\f\right)\right|_{Q^d\to Q^d\chi^{-d\cdot c_1(W)}}.
\end{equation}
\begin{corollary}
Let $(\vecV,\vecbc)$ be any twist data.
For any $\C(\lambda)[\![\Eff(X)]\!][\![t]\!]$-valued point $\f(-z)$ on $\cL_{X,(\vecV,\vecbc)}$, the form $(\Delta_W^\chi\f)(-z)$ gives a $\C(\lambda)[\![\Eff(X)]\!][\![t]\!]$-valued point on $\cL_{X,(\vecV,\vecbc),(W,e_\chi^{-1})}$.
\end{corollary}

\subsubsection{Gromov-Witten theory of vector bundles}
\label{subsubsec:GW_vector_bundle}
We let $B$ be a smooth projective variety and $V=\bigoplus_{i=1}^N V_i$ be a direct sum of vector bundles over $B$.
By considering the diagonal action of $\T=(\C^\times)^N$ on $V$, we can study $\T$-equivariant Gromov-Witten theory of $V$.

We assume that its dual $V^\vee$ is globally generated, which implies that $V$ is semi-projective (\cite[Lemma 2.1]{IK:quantum}).
In this case, a $\Frac(H^*_\T(\pt))[\![\Eff(B)]\!][\![t]\!]$-valued point of $\cL_{V,\T}$ is a point on 
\[
\cH_{V,\T}^{\pol} := H^*_\T(V)[z,z^{-1}][\![\Eff(B)]\!]
\]
if its non-negative part (as a $z$-series) belongs to $-z + H^*_\T(V)[z][\![\Eff(B)]\!][\![t]\!]$, i.e., it is polynomial in the equivariant parameters $\lambda$.
This is because $\psi_1$ is nilpotent which follows from the fact that an image of each $\T$-fixed stable map $C\to V$ is in $V^\T=B$.

Applying the virtual localization formula (Corollary \ref{thm:vir_loc}), we have
\[
\corr{\alpha_1\psi^{k_1},\dots,\alpha_n\psi^{k_n}}^{V,\T}_{0.n,d} = \corr{\alpha_1\psi^{k_1},\dots,\alpha_n\psi^{k_n}}^{B,(V,e_\T^{-1})}_{0.n,d}.
\]
Hence Gromov-Witten theory of $V$ is equivalent to $(V,e_\T^{-1})$-twisted Gromov-Witten theory of $B$:
\[
\cL_{V,\T} \cap \cH_{V,\T}^{\pol} = \cL_{B,(V,e_\T^{-1})} \cap \cH_{V,\T}^{\pol}.
\]

For any ring $R$, we define $R[\lambda^{-1}]$ (resp. $R[\![\lambda^{-1}]\!]$) to be a ring $R[\lambda_1^{-1},\dots,\lambda_N^{-1}]$ (resp. $R[\![\lambda_1^{-1},\dots,\lambda_N^{-1}]\!]$) where $\lambda_i$ is a $\T$-equivariant parameter corresponding to the $i$-th projection $\T\to\C^\times$.
In Section \ref{sec:twist_mirror}, we will use the following lemma.

\begin{lemma}
\label{lem:vector}
Let $\f$ be a $\C[\lambda][\![\Eff(B)]\!][\![x]\!]$-valued point on $\cL_{V,\T}$ whose non-negative part (as a $z$-series) lies in $H^*(B)[z][\![\Eff(B)]\!][\![x]\!]$, i.e., it contains no equivariant parameters.
Then $(\Delta_V^\lambda|_{z\to-z})\f$ is a $\C[\![\Eff(B)]\!][\![x,\lambda^{-1}]\!]$-valued point on $\cL_B$.
\end{lemma}

\begin{proof}
It is enough to prove that $\f|_{Q^d\to Q^d\lambda^{d\cdot c_1(V)}}$ belongs to $H^*(B)[\lambda^{-1}][\![\Eff(B)]\!][\![x]\!]$ since $\Delta_{(V,\te_\lambda^{-1})}\in H^*(B)[\lambda^{-1}](\!(z)\!)$. 
By the definition of $\cL_{V,\T}$, the function $\f|_{Q^d\to Q^d\lambda^{d\cdot c_1(V)}}$ can be written as
\[
-\bone z + \bt(z) + \sum_{ \substack{n\geq0,d\in\Eff(B) \\ (n,d)\neq(0,0),(1,0)} } \sum_{i\in I} \frac{Q^d}{n!} \corr{\frac{\phi_i}{-z-\psi},\bt(\psi),\dots,\bt(\psi)}_{0,n+1,d}^{B, (V,\te_\lambda^{-1})} \phi^i \te_\lambda(V)
\]
where $\{\phi_i\}_{i\in I}$ is a basis of $H^*(B)$, $\{\phi^i\}_{i\in I}$ is a dual basis and $\bt(z)\in H^*(B)[z][\![\Eff(B)]\!][\![x]\!]$ with $\bt(z)|_{(Q,x)=0}=0$.
Since $\te_\lambda(V)$ and $\te_\lambda^{-1}(V_{0,n+1,d})$ are power series in $\lambda_1^{-1},\dots,\lambda_N^{-1}$, this form indeed belongs to $H^*(B)[\lambda^{-1}][\![\Eff(B)]\!][\![x]\!]$.  
\end{proof}

\subsubsection{Gromov-Witten theory of subvarieties}
\label{subsubsec:GW_subvariety}
Let $X$ be a smooth projective variety, and let $V$ be a vector bundle over $X$.
We assume that $V$ is globally generated,\footnote{This assumption implies that $V$ is convex, i.e., for any genus zero stable map $f\colon C\to X$, the first cohomology $H^1(C,f^*V)$ vanishes..} which implies that $V_{0,n,d}$ is a vector bundle for any $n\in\Z_{\geq0}$ and $d\in\Eff(X)$ with $(n,d)\neq(0,0),(1,0),(2,0)$.

We take a regular section of $V$ and write its zero-scheme as $\iota\colon Z\to X$.
We let $\T\times\C^\times$ act on $X$ trivially, and write equivariant parameters for $\T$ (resp. the second factor) as $\lambda_1,\dots,\lambda_N$ (resp. $\mu$).
Let $W_i$, $i=1,\dots,N$, be a vector bundle over $X$ and let  $\bc^i$, $i=1,\dots,N$, be the characteristic class $e_{\chi_i}^{-1}$ associated with a character $\chi_i\colon \T\times\C^\times \to \C^\times$ which is non-trivial on $\T\times\{1\}$.
In this situation, we can relate $\cL_{X,(V,e_\mu),(\vecW,\vecbc)}$ and $\cL_{Z,(\iota^*\vecW,\vecbc|_{\mu=0})}$.
\begin{theorem}[see also {\cite[Section 2.1]{Pandharipande:rational}}, {\cite[Proposition 2.4]{Iritani:periods}}]
\label{thm:subvariety}
For any $\C(\lambda,\mu)[\![\Eff(X)]\!][\![x]\!]$-valued point $\f$ on $\cL_{X,(V,e_\mu),(\vecW,\vecbc)}$ whose limit $\lim_{\mu\to0}\f$ exists, $\lim_{\mu\to0}\iota^*\f$ gives a $\C(\lambda)[\![\Eff(X)]\!][\![x]\!]$-valued point on $\cL_{Z,(\iota^*\vecW,\vecbc|_{\mu=0})}$.
\end{theorem}

\begin{proof}
We consider the following diagram:
\[
\xymatrix@C=50pt{
Z_{0,n+2,d}\ar[r]^-{\ev_{Z,n+2}}\ar[d]_-{\pi_Z}\ar[rd]^-{\iota_{n+2}}	&	Z\ar@{^{(}->}[rd]^-\iota					&		\\
Z_{0,n+1,d}\ar[rd]^{\iota_{n+1}}							&	X_{0,n+2,d}\ar[r]_-{\ev_{X,n+2}}\ar[d]^-{\pi_X}	&	X	\\
											&	X_{0,n+1,d}							&
}
\]
where $Z_{0,n,d} := \amalg_{d'\colon\iota_*d'=d} Z_{0,n,d'}$, $\pi_X$ and $\pi_Z$ denote the forgetful map for the last marking, $\ev_{X,i}$ (resp. $\ev_{Z,i}$) denotes the $i$-th evaluation map for $X_{0,n+2,d}$ (resp. $Z_{0,n+2,d}$), and $\iota_n\colon Z_{0,n,d}\to X_{0,n,d}$ denotes the map induced by $\iota$.
We note that this diagram is commutative, and in particular the square located in the lower left is a fiber diagram.
Therefore, we have
\begin{equation}
\label{eqn:twist_subvar}
(\iota^*W)_{0,n+1,d} := {\pi_Z}_* \ev_{Z,n+2}^* \iota^* W = \iota_{n+1}^*W_{0,n+1,d}
\end{equation}
for any vector bundle $W$ over $X$.

Let $s_0^!$ be a refined Gysin map \cite{Fulton,Vistoli} associated to the fiber diagram
\[
\xymatrix{
Z_{0,n+1,d}\ar[r]^-{\iota_{n+1}}\ar[d]_-{\iota_{n+1}}	&	X_{0,n+1,d}\ar[d]^-s	\\
X_{0,n+1,d}\ar[r]_-{s_0}					&	V_{0,n+1,d}
}
\]
where $s_0$ is a map induced by the zero section $X\to V$, and $s$ is a map induced by a regular section of $V$ defining the subvariety $Z\hookrightarrow X$.
From the functoriality of virtual fundamental classes \cite{KKP:functoriality}
\[
s_0^![X_{0,n,d}]^{\vir} = \sum_{d'\colon\iota_*d'=d} [Z_{0,n,d'}]^{\vir}
\]
and \eqref{eqn:twist_subvar}, we have
\begin{equation}
\label{eqn:twist_functorial}
s_0^!\left( \prod_{j=1}^N\bc^j((W_j)_{0,n,d}) \cap [X_{0,n,d}]^{\vir} \right) = \prod_{j=1}^N\bc^j((\iota^*W_j)_{0,n,d}) \cap \left( \sum_{d'\colon\iota_*d'=d} [Z_{0,n,d'}]^{\vir} \right).
\end{equation}

Let $\f$ be a $\C(\lambda,\mu)[\![\Eff(X)]\!][\![x]\!]$-valued point on $\cL_{X,(V,e_\mu),(\vecW,\vecbc)}$ with well-defined limit $\lim_{\mu\to0}\f$.
From \eqref{eqn:twist_push}, $\f$ can be written in the form of
\begin{multline*}
-\bone z + \bt(z) + \sum_{ \substack{n\geq0,d\in\Eff(X) \\ (n,d)\neq(0,0),(1,0)} } \frac{Q^d}{n!} \cdot e_\mu(V)^{-1} \cdot \prod_{j=1}^N\bc^j(W_j)^{-1} \\
\cdot {\ev_{X,1}}_* \left[ \frac{\prod_{i=2}^{n+1}\ev_{X,i}^*\bt(\psi_i)}{-z-\psi_1} \cdot e_\mu(V_{0,n+1,d}) \cdot \prod_{j=1}^N \bc^j((W_j)_{0,n+1,d}) \cap \left[ X_{0,n+1,d} \right]^{\vir} \right]
\end{multline*}
for some $\bt(z)\in H^*(X)\otimes\C(\lambda,\mu)[\![z]\!][\![\Eff(X)]\!][\![x]\!]$.
Arguing as the proof of \cite[Proposition 2.4]{Iritani:periods}, for any $\alpha\in A_*(X_{0,n+1,d})$, we have
\begin{multline*}
\lim_{\mu\to0} \iota^* \left( e_\mu(V)^{-1} \cdot {\ev_{X,1}}_* \left[ \frac{\prod_{i=2}^{n+1}\ev_{X,i}^*\bt(\psi_i)}{-z-\psi_1} \cdot e_\mu((V_{0,n+1,d})) \cap \alpha \right] \right)	\\
= {\ev_{Z,1}}_* \left[ \frac{\prod_{i=2}^{n+1}\ev_{Z,i}^*\iota^*\bt_0(\psi_i)}{-z-\psi_1} \cap s_0^!\alpha \right]
\end{multline*}
where $\bt_0(z):=\bt(z)|_{\mu=0}$.
Note that $\bt_0(z)$ exists as an element of $H^*(X)\otimes\C(\lambda)[\![z]\!][\![\Eff(X)]\!][\![x]\!]$ since $\lim_{\mu\to0}\f$ exists.
Applying this for $\alpha=\prod_{j=1}^N\bc^j((W_j)_{0,n,d}) \cap [X_{0,n,d}]^{\vir}$ gives that
\begin{multline*}
\lim_{\mu\to0}\iota^*\f = -\bone z + \iota^*\bt_0(z) + \sum_{ \substack{n\geq0,d\in\Eff(X) \\ (n,d)\neq(0,0),(1,0)} } \frac{Q^d}{n!} \cdot \prod_{i=1}^N\bc^i(\iota^*W_i)^{-1} \\
\cdot {\ev_{Z,1}}_* \left[ \frac{\prod_{i=2}^{n+1}\ev_{Z,i}^*\bt(\psi_i)}{-z-\psi_1} \cdot \prod_{j=1}^N \bc^j((\iota^*W_j)_{0,n+1,d}) \cap \left[ Z_{0,n+1,d} \right]^{\vir} \right],
\end{multline*} 
which is a $\C(\lambda)[\![\Eff(X)]\!][\![x]\!]$-valued point on $\cL_{Z,(\iota^*\vecW,\vecbc|_{\mu=0})}$.
Here we use \eqref{eqn:twist_functorial}.

\end{proof}

\section{Toric bundles}
\label{sec:toric_bundle}
In this section, we introduce toric bundles.
We first review toric varieties, and then define toric bundles by doing the construction of toric varieties in a relative setting.
Note that they include toric bundles appearing in \cite{Brown} \cite{IK:quantum}.
We then investigate geometric structures of toric bundles: $\T$-equivariant cohomology ring (\ref{subsec:coh}), effective curves (\ref{subsec:effective}), $\T$-fixed loci and one-dimensional orbits (\ref{subsec:fixed_locus}).

\subsection{Construction}
We start with the triple $\sfL=(\LL^\vee,D,\omega)$ where
\begin{itemize}
\item $\LL^\vee$ is a free abelian group of rank $K$;
\item $D\colon\Z^N\to\LL^\vee$ is a map;
\item $\omega$ is a vector in $\LL^\vee\otimes\R$.
\end{itemize}
We denote as $D_i$ the image of the $i$-th standard basis vector of $\Z^N$ under the map $D$.
We identify the datum $D\colon\Z^N\to\LL^\vee$ with the subset $\{D_1,\dots,D_N\}$ of $\LL^\vee$. 
We set
\begin{align*}
\cA_\sfL:=&\left\{I\subset\{1,\dots,N\} : \omega\in\sum_{i\in I}\R_{\geq0}\cdot D_i\right\}, \\
\cU_\sfL:=& \C^N\setminus\bigcup_{I\notin\cA_\sfL}\C^I, \qquad \KL:=\Hom(\LL^\vee,\C^\times), \qquad \TL:=\Hom(\Z^N,\C^\times)
\end{align*}
where $\C^I:=\{(z_1,\dots,z_N)\in\C^N : z_i=0 \text{ for } i \notin I\}$.
We call elements of $\cA_\sfL$ \emph{anti-cones}.
By applying the functor $\Hom(\cdot,\C^\times)$ to the map $D\colon\Z^N\to\LL^\vee$, we get an embedding $\KL\hookrightarrow\TL$ and the torus $\KL$ acts on $\cU_\sfL$ via this embedding.
We note that $\cU_\sfL$ is naturally endowed with $\TL$-action.
\begin{definition}
\label{def:toric}
We call the triple $\sfL=(\LL^\vee,D,\omega)$ a \emph{smooth toric data} if it satisfies the following two conditions:
\begin{itemize} 
\item[(1)] The vector $\omega$ belongs to $\sum^N_{i=1}\R_{\geq0}\cdot D_i$.
\item[(2)] For any $I\in\cA_\sfL$, $\{D_i\}_{i\in I}$ generates $\LL^\vee$. 
\end{itemize}
\end{definition}
The condition (1) ensures that $\cU_\bL$ is nonempty and (2) ensures that the action of $\KL$ on $\cU_\sfL$ is free. 
Under these hypothesis, the quotient space $X_\sfL = \cU_\sfL/\KL$ becomes a smooth semi-projective toric variety. 

We consider a relative version of this construction.
We fix a smooth toric data $\sfL=(\LL^\vee,D\colon\Z^N\to\LL^\vee,\omega)$. 
Let $V_1,\dots,V_N$ be non-zero vector bundles over a smooth projective variety $B$ and set $r_i:=\rank V_i$.
We let $\KL$ act on $V_i$ fiberwise by the character $D_i\colon\KL\to\C^\times$ and define a toric bundle $\XLV$ over $B$ as follows:
\[
\cU_\sfL(\vecV):=\left( \bigoplus^N_{i=1}V_i \setminus \bigcup_{I\notin\cA_\sfL} \bigoplus_{i\in I}V_i \right), \qquad \XLV := \left. \cU_\sfL(\vecV) \right/ \KL.
\]
Note that $\cU_\sfL(V)$ is endowed with $\T$-action coming from the $\T$-actions on $V_1,\dots,V_N$, and it induces a $\T$-action on $\XLV$.
When $B=\pt$, then $\XLV$ is a smooth semi-projective toric variety, which we denote by $\X^\vecr$.
For a general base, $\XLV$ is an $\X^\vecr$-bundle over $B$.

\begin{definition}
Let $E\to B$ be a toric bundle obtained by the above procedure.
We say that $E$ is \emph{of split type}, or $E$ is a \emph{split toric bundle} if there is a smooth toric data $\sfL$ and line bundles $\{L_i\}$ over $B$ such that $E\to B$ is isomorphic to $\X_\sfL(\vecL)\to B$.
We say that $E$ is \emph{of non-split type}, or $E$ is a \emph{non-split toric bundle} if $E$ is not of split type.
\end{definition}

As explained in Section \ref{sec:intro}, mirror theorems for split toric bundles \cite{Brown} and (non-split) projective bundles \cite{IK:quantum} are already known.
We will prove a mirror theorem for (non-split) toric bundles (Theorem \ref{thm:mirror_thm}). 

\begin{proposition}
\label{prop:semi_proj}
\phantom{A}
\begin{itemize}
\item[$(1)$] The bundle $\XLV$ is projective if and only if the toric variety $X_\sfL$ is projective.
\item[$(2)$] The bundle $\XLV$ is semi-projective if the vector bundles $V_1^\vee,\dots,V_N^\vee$ are generated by global sections.
\end{itemize}
\end{proposition}

\begin{proof}
We only prove (2).
From the assumption, we have the exact sequences
\[
0\to V_i\to\cO^{\oplus s_i}\to\cQ_i\to0
\]
for $i=1,\dots,N$.
By taking the direct sum of these sequences, we have the inclusion $\bigoplus^N_{i=1}V_i\hookrightarrow\bigoplus^N_{i=1}\cO^{s_i}$.
This induces a closed embedding $\XLV\hookrightarrow B\times\X^{\vecs}$.
Since $B\times\X^{\vecs}$ is semi-projective, $\XLV$ is also semi-projective.
\end{proof}

\subsection{Cohomology ring}
\label{subsec:coh}
We want to describe the ordinary $\TL$-equivariant cohomology of $\XLV$.
We write the element of $H^2_\TL(\pt) \cong \Hom(\TL,\C^\times)$ corresponding to the $i$-th projection $\TL\to\C^\times$ as $-\lambda_i$. 

For $1\leq i\leq N$, let $L_i$ be a $\TL$-equivariant line bundle over $\XLV$ defined by
\begin{equation}
\label{eqn:L_i}
L_i = \left. \cU_\sfL(\vecV)\times\C\right/ \KL
\end{equation}
where $\K$ acts on the second factor $\C$ by the character $D_i\colon\K\to\C^\times$. 
We let $\TL$ act on $L_i$ as $t\cdot[v,w] = [t\cdot v,t_iw]$.
We write the $\TL$-equivariant first Chern class of $L_i$ as $u_i$.

\begin{proposition}
\label{prop:cohomology}
We have the isomorphisms
\begin{align*}
H^*_\TL(\XLV,\Z) &\cong H^*_\TL(B,\Z)[u_1,\dots,u_N]/(\cI+\cJ),	\\
H^*(\XLV,\Z) &\cong H^*(B,\Z)[u_1,\dots,u_N]/(\cI|_{\lambda=0}+\cJ|_{\lambda=0})
\end{align*}
where the ideal $\cI$ is generated by $\{\prod_{i\notin I} e_\TL(V_i\otimes L_i) \}_{I\notin\cA_\sfL}$, and the ideal $\cJ$ is generated by $\{\sum^N_{i=1}a_i(u_i+\lambda_i)\}_{a\in\Ker(D\colon\Z^N\to\LL^\vee)}$.
Here $\T$ acts on $B$ trivially and $H^*_\TL(B,\Z)\cong H^*(B,\Z)\otimes\Z[\lambda]$.
\end{proposition}

\begin{proof}
We first note that this proposition is known to be true when $B$ is a point; see for example \cite[Proposition 2.11]{HS:toric}.
Moreover, it is known that $H^*_\TL(\X^\vecr,\Z)$ is a free $H^*_\TL(\pt,\Z)$-module of finite rank.
We can choose a $H^*_\TL(\pt)$-basis $\{ \phi_b(u) \}_{b\in\cB}$ of $H^*_\TL(\X^\vecr,\Z)$ where $\phi_b(u)\in\Z[u_1,\dots,u_N]$.

There is a natural ring homomorphism $f\colon H^*_\TL(B,\Z)[u_1,\dots,u_N]\to H^*_\TL(\XLV,\Z)$. 
For any $I\notin\cA_\sfL$, the vector bundle 
\[
\bigoplus_{i\notin I}(V_i\otimes L_i) = \left. \left( \cU_\sfL(\vecV)\times_B\bigoplus_{i\notin I}V_i \right) \right/ \KL
\]
has the nowhere vanishing $\TL$-equivariant section $[v_1,\dots,v_N]\mapsto[v_1,\dots,v_N,\{v_i\}_{i\notin I}]$, which implies $\cI\subset\ker(f)$.
In a similar way, we can see that $\cJ\subset\ker(f)$.
Therefore, we have a ring homomorphism
\[
\tf\colon H^*_\TL(B,\Z)[u_1,\dots,u_N]/(\cI+\cJ)\to H^*_\TL(\XLV,\Z).
\]
It is easy to show that the domain of $\tf$ is a free $H^*_\TL(B,\Z)$-module with basis $\{ \phi_b(u) \}_{b\in\cB}$.

We consider the following fiber bundle:
\[
\xymatrix{
\X^\vecr\ar[r]	&	\left.\left( \XLV\times E\T \right) \right/ \T\ar[d]	\\
			&	B\times B\T
}
\]
where $E\T$ is a contractible space on which $\T$ acts freely, $B\T=E\T/\T$ is a classifying space of $\T$, and the vertical map is induced by the projections $\XLV\to B$ and $E\T\to B\T$.
By the Leray-Hirsch theorem, 
\[
H^*((\XLV\times E\TL)/\TL,\Z) = H^*_\TL(\XLV,\Z)
\]
is a free $H^*_\TL(B,\Z)$-module with basis $\{ \tf(\phi_b(u)) \}_{b\in\cB}$.
This implies that $\tf$ is an isomorphism.
\end{proof}

This proposition provides an explicit description of the second cohomology of $\XLV$.
For $\vecr\in(\Z_{>0})^N$, we set
\[
I^\vecr_\sfL:=\{i\in\{1,\dots,N\}\colon r_i=1\text{ and }\{1,\dots,N\}\setminus\{i\}\notin\cA_\sfL\}.
\]
It follows easily from Proposition \ref{prop:cohomology} that there exists a natural exact sequence
\begin{equation}
\label{eqn:seq_for_H^2}
0	\to	\Z^{\oplus I^\vecr_\sfL}	\to	H^2(B,\Z)\oplus\LL^\vee	\to	H^2(\XLV,\Z)	\to	0
\end{equation}
where the second arrow sends a standard generator $e_i\ (i\in I^\vecr_\sfL)$ to $(c_1(V_i),D_i)$.
Hence we have the isomorphism
\[
H^2(\XLV,\Z) \cong H^2(B,\Z)\oplus\LL^\vee/\langle(c_1(V_i),D_i)\colon i\in I^\vecr_\sfL\rangle.
\]
In particular, we have
\[
H^2(\X^\vecr,\Z) \cong \LL^\vee/\langle D_i \colon i\in I^\vecr_\sfL \rangle.
\]

\subsection{Effective curve classes}
\label{subsec:effective}
In this subsection, we will study the effective curve classes $\Eff(\XLV)$ and introduce \emph{extended effective classes} $\extEff(\XLV)$.

Taking the dual of the sequence \eqref{eqn:seq_for_H^2} we can see that
\[
H_2(\XLV,\Z) \cong \left\{ (D,\ell)\in H_2(B,\Z)\oplus\LL: c_1(V_i)\cdot D + D_i(\ell)=0\ \textrm{ for }\ 1\leq i\leq N \right\}.
\]
We construct a canonical splitting of the second homology of $\XLV$ into those of $B$ and $\X^\vecr$ (which depends on the choice of $\vecV$).
There exists a canonical splitting \cite[Section 3.1.2]{Iritani:integral}:
\begin{align}
\label{eqn:toric_splitting1}
H_2(\X^\vecr,\Z)\oplus\Z^{\oplus I^\vecr_\sfL} &\cong \LL,	\\
\label{eqn:toric_splitting2}
H^2(\X^\vecr,\Z)\oplus\Z^{\oplus I^\vecr_\sfL} &\cong \LL^\vee.
\end{align}
We define $\phi\colon H^2(\X^\vecr)\to H^2(\XLV)$ as follows.
For $\overline{\rho}\in H^2(\X^\vecr)$, we take its lift $\rho\in\LL^\vee$ corresponding to $(\overline{\rho},0)$ under the splitting \eqref{eqn:toric_splitting2}. 
We define $\phi(\overline{\rho})$ to be the first Chern class of the line bundle $\cO(\rho)$ defined as
\[
\left.\cO(\rho) = \cU_\sfL(\vecV)\times\C\right/\KL
\]
where $\KL$ acts on the second factor via the character $\KL\to\C^\times$ obtained by $\rho$.
This map is well-defined and gives the splitting $H^2(\XLV) \cong H^2(B)\oplus H^2(\X^\vecr)$.
By dualizing, we obtain the splitting 
\begin{equation}
\label{eqn:splitting0}
H_2(\XLV,\Z) \cong H_2(B,\Z)\oplus H_2(\X^\vecr,\Z).
\end{equation}
We define $\LL_\eff\subset\LL$ to be a cone which coincides with $\Eff(\X^\vecr)\oplus(\Z_{\geq0})^{\oplus I^\vecr_\sfL}$ via the above splitting.
Note that $\LL_\eff$ is independent of $\vecr$ and can be written as 
\[
\LLeff = \sum_{I\in\cA_\sfL} \{ \ell\in\LL: D_i(\ell)\geq0 \text{\ \ for all \ } i\in I \}.
\]
See \cite[Section 3.1.2]{Iritani:integral} for details.

Combining the two splittings \eqref{eqn:toric_splitting1} and \eqref{eqn:splitting0}, we obtain the isomorphism
\begin{equation}
\label{eqn:splitting}
H_2(\XLV,\Z)\oplus\Z^{I^\vecr_\sfL} \cong H_2(B,\Z) \oplus \LL.
\end{equation}

We henceforce assume that $\bigoplus_{i=1}^NV_i$ is globally generated.
We let $\cD\colon H_2(B,\Z)\oplus\LL\to H_2(\XLV,\Z)$ be the isomorphism \eqref{eqn:splitting} composed with the projection to $H_2(\XLV,\Z)$, and define a \emph{semigroup of extended effective curve classes} as follows:
\[
\extEff(\XLV) = \cD(\Eff(B)\oplus\LL_\eff). 
\]
As we can see from the following lemma, the cone $\extEff(\XLV)$ indeed extends $\Eff(\XLV)$.

\begin{lemma}
\label{lem:extEff}
$\extEff(\XLV) \supset \Eff(\XLV).$
\end{lemma}

\begin{proof}
From the assumption that $\bigoplus_{i=1}^NV_i$ is globally generated, we can construct the embedding $\iota\colon\XLV\hookrightarrow B\times\X^\vecs$ for some $\vecs\in(\Z_{\geq2})^N$; see the proof of Proposition \ref{prop:semi_proj}.
We note that $H_2(\X^\vecs,\Z) \cong \LL$ since $I^\vecs_\sfL=\emptyset$. 
For any $\rho\in\LL^\vee$, the line bundle $\cO(\rho)$ comes from $\cO(\rho)$ over $B\times\X^\vecs$ via the embedding $\iota$.
Therefore, the map 
\[
\iota_*\colon H_2(\XLV,\Z)\to H_2(B\times\X^\vecs,\Z) = H_2(B,\Z)\oplus H_2(\X^\vecs,\Z)
\]
coincides with the inclusion $H_2(\XLV,\Z)\hookrightarrow H_2(B,\Z)\oplus\LL$ induced by the splitting \eqref{eqn:splitting}.
This gives the inclusion 
\[
\Eff(\XLV)\xrightarrow{\iota_*}\Eff(B\times\X^\vecs)=\Eff(B)\oplus\LL_\eff\xrightarrow{\cD}\extEff(\XLV)\hookrightarrow H_2(\XLV,\Z).
\] 
\end{proof}

\begin{remark}
The splittings \eqref{eqn:splitting0} and \eqref{eqn:splitting} depend on the choice of vector bundles $V_1,\dots,V_N$.
When replacing $\vecV$ by $\vecV'=(V_1\otimes M^{\otimes D_1(\ell)},\dots,V_N\otimes M^{\otimes D_N(\ell)})$ for some $\ell\in\LL$ and line bundle $M\to B$, the total space $\XLV$ does not change.
However, for $\rho\in\LL^\vee$, the associated line bundles $\cO_{\XLV}(\rho)$ and $\cO_{\X_\sfL(\vecV')}(\rho)$ may not coincide.
In fact, we have
\[
\cO_{\X_\sfL(\vecV')}(\rho) = \cO_{\XLV}(\rho) \otimes M^{\otimes -\rho(\ell)}
\]
via the canonical identification of $\XLV$ and $\X_\sfL(\vecV')$. 
The splittings \eqref{eqn:splitting0}, \eqref{eqn:splitting} and the semigroup $\extEff(\XLV)$ change according to this substitution. 
\end{remark}

Via the splitting \eqref{eqn:splitting}, the cone $\Eff(B)\oplus\LL_\eff$ is identified with $\extEff(\XLV)\oplus(\Z_{\geq0})^{\oplus I^\vecr_\sfL}$.
We choose a K\"{a}hler form $\omega$ on $B\times\X^\vecs$.
Let $v$ be the additive valuation on $\C[\Eff(B)\oplus\LLeff] = \C[\extEff(\XLV)\oplus(\Z_{\geq0}^{\oplus I^\vecr_\sfL})]$ given by $\omega$, and consider the completion with respect to $v$.
Let $\cQ$, $Q$ and $q$ denote the Novikov variables for $\XLV$, $B$ and $\X^\vecr$ respectively, and $\tq$ (resp. $y$) be a formal variable for $\C[\![\LL_\eff]\!]$ (resp. $\C[\![\Z_{\geq0}^{\oplus I^\vecr_\sfL}]\!]$).
Furthermore, we can identify $\C[\![\Eff(B)\oplus\LL_\eff]\!]$ with $\C[\![\extEff(\XLV)]\!][\![y]\!]$ in the following way: 
\begin{equation}
\label{eqn:extEff}
Q^D \tq^\ell = \cQ^{\cD(D,\ell)} \prod_{i\in I^\vecr_\sfL} y_i^{D_i(\ell)}.
\end{equation}
We call $\tq$ the \emph{extended Novikov variable for $\X^\vecr$}.

\subsection{Fixed loci and one-dimensional orbits}
\label{subsec:fixed_locus}
We describe the $\T$-fixed loci on $\XLV$ and introduce some varieties related to one-dimensional orbits.

Let $\sfL$ be a smooth toric data.
We denote by $F_\sfL$ the set of minimal anti-cones:
\[
F_\sfL = \{ \alpha\in\cA_\sfL : |\alpha| = k \}.
\]
For $[x]\in(\X_\sfL^\vecr)^\T$, we take the set $\{ i : x_i\neq0 \}$, which belongs to $F_\sfL$. 
This gives a one-to-one correspondence between the connected components of $(\X_\sfL^\vecr)^\T$ and the set $F_\sfL$.

In fact, there is also a one-to-one correspondence between the connected components of $\XLV^\T$ and the set $F_\sfL$ obtained in the same manner.
More precisely, we assign to $\alpha\in F_\sfL$ a subvariety $\XLV_\alpha$ which is a bundle over $B$ with fiber being the $\T$-fixed locus of $\X_\sfL^\vecr$ corresponding to $\alpha\in F_\sfL$.
Note that $\XLV_\alpha$ is isomorphic to the fiber product over $B$ of the projective bundles $\{\PP(V_j)\}_{j\in\alpha}$.
We introduce several definitions and notations.

\begin{definition}
\label{def:fixed_locus}
Let $\alpha\in F_\sfL$.
\begin{itemize}
\item[$(1)$]
We write the inclusion $\XLV_\alpha\hookrightarrow\XLV$ as $\iota_\alpha$, and write as $N_\alpha$ the normal bundle to $\XLV_\alpha$ in $\XLV$:
\[
N_\alpha = \bigoplus_{i\notin\alpha} (V_i\otimes L_i).
\]
\item[$(2)$]
For $\alpha\in2^{\{1,\dots,N\}}$, we write a fiber product of $\{\PP(V_j)\to B\}_{j\in\beta}$ over $B$ as $\XLV_\alpha$.
(This convention is consistent with the above notation for the fixed loci.)
Here we define $\XLV_\alpha = B$ if $\alpha=\emptyset$.
If $\alpha\supset\beta$ there is a natural projection $\XLV_\alpha\to\XLV_\beta$ which is denoted by $p_\ab$.
\item[$(3)$]
We denote by $\{D_{\alpha,i}^\vee\}_{i\in\alpha}\subset\LL$ the dual basis of $\{D_i\}_{i\in\alpha}\subset\LL^\vee$.
\item[$(4)$]
We say that \emph{$\beta\in F_\sfL$ is adjacent to $\alpha$} if $\#(\beta\setminus\alpha)=1$.
We write the unique element of $\beta\setminus\alpha$ as $i_{\alpha,\beta}$.
Define
\[
\adj(\alpha) = \{ \beta\in F_\sfL \colon \beta \text{ is adjacent to } \alpha \}.
\]
\item[$(5)$]
Let $\beta\in\adj(\alpha)$.
We denote by $\dab$ the homology class of $\XLV$ given by a one-dimensional orbit joining points on $\XLV_\alpha$ and $\XLV_\beta$.
\item[$(6)$]
For $\beta\in\adj(\alpha)$, we define 
\begin{equation}
\label{eqn:Lab}
L_{\alpha,\beta} := \left. \left( \bigoplus_{i\in\alpha\cup\beta} (V_i\setminus B) \oplus \cO_B \right) \right/ \left( \KL\times\C^\times \right)
\end{equation}
where $\TL$ acts trivially on $\cO_B$ and the second factor $t\in\C^\times$ acts on $(v,s)\in V_{i_{\alpha,\beta}}\oplus\cO_B$ as $t\cdot(v,s)=(tv,t^{-1}s)$ and acts trivially on the other components.
(As we will see, $L_\ab$ is a line bundle over $\XLV_\acb$.)
\item[$(7)$]
For $\beta\in\adj(\alpha)$, we define $\lambda_\ab\in H^2_\TL(\pt)$ to be the image of $c_1^\TL(L_\ab)$ under the canonical projection $H^2_\TL(\XLV_\acb) = H^2(\XLV_\acb)\otimes H^2_\TL(\pt) \to H^2_\TL(\pt)$.
\end{itemize}
\end{definition}

From the description of $H^*_\TL(\XLV)$ in Proposition \ref{prop:cohomology}, we can describe $H^*_\TL(\XLV_\alpha)$ and the map $\iota_\alpha^*\colon H^*_\TL(\XLV)\to H^*_\TL(\XLV_\alpha)$.

\begin{corollary}
\label{cor:coh}
For $\alpha\in2^{\{1,\dots,N\}}$, we have the isomorphism
\[
H^*_\TL(\XLV_\alpha) \cong H^*(B)[\{u_i\}_{i\in\alpha}]/\langle e_\TL(V_i\otimes L_i) \colon i\in\alpha \rangle.
\]
If $\alpha\in F_\sfL$, the map
\begin{multline*}
\iota_\alpha^*\colon H^*_\TL(B)[u_1,\dots,u_N]/(\cI+\cJ)\cong H^*_\TL(\XLV)	\\
\to H^*_\TL(\XLV_\alpha)\cong H^*_\TL(B)[\{u_i\}_{i\in\alpha}]/\langle e_\TL(V_i\otimes L_i) \colon i\in\alpha \rangle
\end{multline*}
is a $H^*_\TL(B)$-module homomorphism which sends $u_i$ to 
\[
\iota_\alpha^*u_i = -\lambda_i + \sum_{j\in\alpha} D_i(D_{\alpha,j}^\vee)\cdot(u_j+\lambda_j).
\] 
In particular, $\iota_\alpha^*u_i=u_i$ if $i\in\alpha$.
\end{corollary}

By abuse of notation, for any $\alpha\in2^{\{1,\dots,N\}}$ and $j\in\alpha$, we denote the equivalence class of $u_j$ in $H^*_\TL(\XLV_\alpha)$ by $u_j$.
We remark that for any $\alpha\in F_\sfL$, $\beta\in\adj(\alpha)$ and $i\in(\alpha\cap\beta)^c$, two pullbacks $\paba^*\iota_\alpha^*u_i$ and $\pabb^*\iota_\beta^*u_i$ do not coincide; see in Lemma \ref{lem:dab_comparison} (3).

Let $\alpha\in F_\sfL$.
For $\beta\in F_\sfL$, the condition that $\beta$ is adjacent to $\alpha$ is equivalent to the existence of a one-dimensional orbit joining points on $\XLV_\alpha$ and $\XLV_\beta$.
If $\beta\in\adj(\alpha)$, there is a canonical isomorphism 
\[
\XLV_{\alpha\cup\beta} \cong \left. \bigoplus_{i\in\alpha\cup\beta} (V_i\setminus B) \right/ (\KL\times\C^\times),
\]
and a projection $L_{\alpha,\beta}\to\XLV_{\alpha\cup\beta}$ which makes $L_{\alpha,\beta}$ a line bundle over $\XLV_{\alpha\cup\beta}$.
We give the description of $\dab$ and $c_1^\T(L_\ab)$.

\begin{lemma}
\label{lem:dab_comparison}
Let $\alpha\in F_\sfL$ and $\beta\in\adj(\alpha)$.
\begin{itemize}
\item[$(1)$]
Under the splitting \eqref{eqn:splitting}, it holds that $\dab = D_{\alpha,i_{\beta,\alpha}}^\vee = D_{\beta,i_\ab}^\vee$.
\item[$(2)$]
It holds that
\begin{align*}
c_1^\TL(L_\ab) 
&= -u_{i_\ab} - \lambda_{i_\ab} + \sum_{j\in\alpha} D_{i_\ab}(D_{\alpha,j}^\vee)\cdot(u_j+\lambda_j)	\\
&= -u_{i_\ab} + \paba^*\iota_\alpha^*u_{i_\ab}.
\end{align*}
\item[$(3)$]
For $1\leq i\leq N$, we have
\[
\pabb^*\iota_\beta^*u_i = \paba^*\iota_\alpha^*u_i - D_i(\dab)\cdot c_1^\TL(L_\ab).
\]
\end{itemize}
\end{lemma}

\begin{proof}
We first prove $(1)$.
Since $\dab$ is represented by a curve $C_\ab$ in the fiber, we can assume that $B=\pt$ and $\XLV=\X^\vecr$ is a toric variety.
Without loss of generality, we can assume that $r_i=1$ for any $i$.
In this case, we can write $X:=\X^\vecr$ and $C_\ab$ as follows:
\begin{align*}
X &= \left\{ [v_1,\dots,v_N]\in\C^N/\K \colon (v_1,\dots,v_N)\notin\bigcup_{I\notin\cA_\sfL}\C^I  \right\},	\\
C_\ab &= \{ [v_1,\dots,v_N]\in X \colon v_i=1 \text{ for } i\in\alpha\cap\beta,\ v_j=0 \text{ for } i\notin\alpha\cup\beta \}.
\end{align*}
By definition, $u_i$ is the $i$-th toric divisor for $X$.
Hence its Poincar\'{e} dual $[Z_i]$ can be taken as 
\[
Z_i = \{ [v_1,\dots,v_N]\in X \colon v_i=0 \}.
\]
Using these descriptions, we can show that 
\[
u_i(\dab) = D_i(\dab) =
\begin{cases}
0	&\text{if } i\in\alpha\cap\beta,		\\
1	&\text{if } i=i_\ab, i_{\beta,\alpha}.	\\
\end{cases}
\]
for $i\in\acb$.
This implies that $\dab = D_{\alpha,i_{\beta,\alpha}}^\vee = D_{\beta,i_\ab}^\vee$.

By \eqref{eqn:Lab}, it follows that 
\[
L_\ab = L_{i_\ab}^\vee \otimes \bigotimes_{j\in\alpha} L_j^{\otimes D_{\alpha,j}^\vee(D_{i_\ab})}
\]
as $\TL$-equivariant line bundles where $L_i\to \XLV$ is the line bundle \eqref{eqn:L_i}.
By taking the first Chern class, we obtain the desired formula.

For $1\leq i\leq N$, we calculate $\paba^*\iota_\alpha^*u_i$ as follows:
\begin{align*}
\paba^*\iota_\alpha^*u_i 
&= \paba^*\left( -\lambda_i + \sum_{j\in\alpha}D_i(D_{\alpha,j}^\vee)\cdot(u_j+\lambda_j) \right)	\\
&= -\lambda_i + \sum_{j\in\alpha}D_i(D_{\alpha,j}^\vee)\cdot(u_j+\lambda_j).
\end{align*}
On the other hand, we have
\begin{align*}
\pabb^*\iota_\beta^*u_i 
&= \pabb^*\left( -\lambda_i + \sum_{j\in\beta}D_i(D_{\beta,j}^\vee)\cdot(u_j+\lambda_j) \right)	\\
&= -\lambda_i + \sum_{j\in\beta}D_i(D_{\beta,j}^\vee)\cdot(u_j+\lambda_j).	
\end{align*}
Since $D_{i_{\beta,\alpha}} = D_{i_\ab} - \sum_{j\in\alpha\cap\beta}D_{i_\ab}(D_{\alpha,j}^\vee)\cdot D_j$, we can see that
\[
D_{\beta,j}^\vee = 
\begin{cases}
D_{\alpha,i_{\beta,\alpha}}^\vee										&\text{if } j=i_\ab,		\\
D_{\alpha,j}^\vee - D_{i_\ab}(D_{\alpha,j}^\vee)\cdot D_{\alpha,i_{\beta,\alpha}}^\vee	&\text{if } j\in\alpha\cap\beta.	
\end{cases} 
\]
Therefore, $\pabb^*\iota_\beta^*u_i $ is equal to
\begin{align*}
&-\lambda_i + D_i(D_{\alpha,i_{\beta,\alpha}}^\vee)\cdot(u_{i_\ab}+\lambda_{i_\ab}) + \sum_{j\in\alpha\cap\beta}\left( D_i(D_{\alpha,j}^\vee) - D_{i_\ab}(D_{\alpha,j}^\vee)\cdot D_i(D_{\alpha,i_{\beta,\alpha}}^\vee) \right)\cdot(u_j+\lambda_j)	\\
=&\ \paba^*\iota_\alpha^*u_i + D_i(\dab) \cdot \left( u_{i_\ab}+\lambda_{i_\ab}-u_{i_{\beta,\alpha}}-\lambda_{i_{\beta,\alpha}}-\sum_{j\in\alpha\cap\beta}D_{i_\ab}(D_{\alpha,j}^\vee)\cdot(u_j+\lambda_j) \right)	\\
=&\ \paba^*\iota_\alpha^*u_i - D_i(\dab)\cdot c_1^\T(L_\ab).
\end{align*}
Here we use $(1)$ and $(2)$.
\end{proof}

For later use, we study moduli spaces parametrizing multiple one-dimensional orbits.

\begin{definition}
Let $\alpha\in F_\sfL$, $\beta\in\adj(\alpha)$ and $k\in\N$.
We define $\ovcM^{\alpha,\beta}_k$ to be the closed substack of $\XLV_{0,2,k\cdot d_{\alpha\beta}}$ consisting of the stable maps which satisfy the following conditions:
\begin{itemize}
\item its domain is non-singular and surjects onto a one-dimensional orbit joining points on $\XLV_\alpha$ and $\XLV_\beta$;
\item the map is $\TL$-invariant;
\item the image of the first marking belongs to $\XLV_\alpha$ and that of the second marking belongs to $\XLV_\beta$.
\end{itemize}
\end{definition}

We can interpret $\XLV_{\alpha\cup\beta}$, $L_{\alpha,\beta}$ and $\ovcM^\ab_k$ as follows.
Firstly, since there is a canonical bijection
\begin{equation}
\label{eqn:bij_orbits}
\XLV_{\alpha\cup\beta} \stackrel{\sim}{\longrightarrow} \{ \text{one-dimensional orbits joining points on $\XLV_\alpha$ and $\XLV_\beta$}  \},
\end{equation}
we can understand $\XLV_{\alpha\cup\beta}$ as the moduli space of one-dimensional orbits in the variety $\XLV_{\alpha,\beta}$ which is the subvariety of $\XLV$ defined as
\[
\XLV_{\alpha,\beta} = \left\{ [v_1,\dots,v_N] \in \XLV \colon v_i = 0 \text{ for any $i\notin\alpha\cup\beta$} \right\} .
\]
In other words, we have $\ovcM^\ab_1 \cong \XLV_\acb$.
Its universal curve $\cC_{\alpha,\beta}$ is obtained by the blow up of $\XLV_{\alpha,\beta}$ along $\XLV_\alpha\amalg\XLV_\beta$ with the canonical maps $\pi\colon\cC_{\alpha,\beta}\to\XLV_{\alpha\cup\beta}$ and $\cC_{\alpha,\beta}\to\XLV_{\alpha,\beta}$.
There is a natural embedding of $L_{\alpha,\beta}$ into $\cC_{\alpha,\beta}$: 
\[
L_\ab \cong \Bl_{\XLV_\alpha} (\XLV) \setminus \XLV_\beta \hookrightarrow\cC_\ab,
\]
which gives the projection $L_{\alpha,\beta}\to\XLV_\acb$ by restricting $\pi$ to $L_\ab$.
Due to this isomorphism, $L_{\alpha,\beta}\to\XLV_{\alpha\cup\beta}$ can be seen as the universal tangent line bundle associated to the section $s_\alpha\colon\XLV_\acb\to\cC_\ab$ defined as follows: for any one-dimensional orbit $C$ passing through $\XLV_\alpha$ and $\XLV_\beta$ (which represents a point $[C]\in\XLV_\acb$ via \eqref{eqn:bij_orbits}), set $s_\alpha([C])\in\pi^{-1}([C]) \cong C$ to be the intersection of $C\subset\XLV_\ab$ and $\XLV_\alpha$.

Moreover, we can see that $\ovcM^\ab_k$ is isomorphic to the $k$-th root stack \cite[Section 2]{Cadman} \cite[Appendix B]{AGV:gw} of the line bundle $L_\ab$, which can be given as the quotient stack $[L_\ab^0/\C^\times]$ where $L_\ab^0=L_\ab\setminus\XLV_\acb$ and $\C^\times$ acts on $L_\ab^0$ by the formula $t\cdot x = t^kx$ for $t\in\C^\times$.
The universal map for $\ovcM^\ab_k$ is given by the diagram
\[
\xymatrix{
L_\ab^0\times\PP^1\ar[r]^F\ar[d]	&	\XLV_\ab	\\
L_\ab^0
}
\]
where $F$ is the morphism sending $([ (v_i)_{i\in\acb},s],[x,y])$ to $[(F_i(v,s,x,y))_{i\in\acb}]$:
\[
F_i(v,s,x,y) =
\begin{cases}
v_i				&i\in\alpha\cap\beta,	\\
x^kv_{i_{\beta,\alpha}}	&i=i_{\beta,\alpha},		\\
sy^kv_{i_\ab}		&i=i_\ab.			\\
\end{cases}
\]
Here we use the identification \eqref{eqn:Lab} and endow $\PP^1$ with the $\T$-action such that $t\cdot[x,y] = [tx,y]$ for $t\in\C^\times$.

\section{Lagrangian cones of toric bundles}
\label{sec:characterization}
Throughout this section, we fix a smooth toric data $\sfL$, a smooth projective variety $B$ and a collection of vector bundles $V_1,\dots,V_N$ whose duals are globally generated.
In this section, we study the $\TL$-equivariant Lagrangian cone of $\XLV$ (Theorem \ref{thm:characterization}) and establish a characterization of points on $\cL_{\XLV}$ in a similar manner to \cite{Brown,CCIT:mirror,JTY,FL}.

\subsection{Characterization theorem}
In this subsection, we state a characterization theorem for points on $\cL_{\XLV,\T}$ (Theorem \ref{thm:characterization}). 
The proof will be given in Section \ref{subsec:proof_characterization}.
We first introduce some notions.

\begin{definition}
Let $X$ be a smooth projective variety with trivial $\TL$-action and let $(S,v_S)$ be a semigroup with a discrete grading structure $v_S\colon S\to\Z_{\geq0}$ which extends that on $\Eff(X)$.
We choose a $\C$-basis $\{\phi_i\}_{i\in I}$ of $H^*(X)$ and let $\chi\in H^*_\T(\pt)$.
We say that a function 
\[
\f = \sum_{s\in S} \sum_{i\in I} Q^s \phi_i \cdot \f_{s,i}  \in H^*_\T(X)\otimes_{H^*_\T(\pt)}\Frac(H^*_\T(\pt)[z])[\![S]\!] = H^*(X)\otimes\C(\lambda,z)[\![S]\!],
\]
where $\f_{s,i}\in\C(\lambda,z)$ for any $(s,i)\in S\times I$, has a \emph{pole at} $z=\chi$ if there exists $(s,i)\in S\times I$ such that $\f_{s,i}$ has a pole along $z-\chi=0$.
Define a \emph{principle part of $\f$ at $z=\chi$} as
\[
\Prin_{z=\chi} \f = \sum_{s\in S} \sum_{i\in I} Q^s \phi_i \Prin_{z=\chi} \f_{s,i}.
\]
\end{definition}

\begin{theorem}
\label{thm:characterization}
Let $\sfL=(\LL^\vee,D\colon\Z^N\to\LL^\vee,\omega)$ be a smooth toric data, $B$ be a smooth projective variety and $V_1,\dots,V_N$ be vector bundles over $B$ whose duals are generated by global sections.
Let $x=(x_1,x_2,\dots)$ be formal vaiables and let $\f$ be an element of $H_\TL(\XLV)(\!(z^{-1})\!)[\![\extEff(\XLV)]\!][\![x]\!]$ such that $\f|_{(\cQ,x)=0}=-1z$.
In this situation, $\f$ is a $\Frac(H^*_\TL(\pt))[\![\extEff(\XLV)]\!][\![x]\!]$-valued point on $\cL_{\XLV,\TL}$ if and only if the following three conditions hold $:$
\begin{itemize}
\item[\textbf{(C1)}]		For each $\alpha\in F_\sfL$, the restriction $\iota_\alpha^*\f$ belongs to 
\[
\left(H^*_{\TL}(\XLV_\alpha)\otimes_{H^*_{\TL}(\pt)}\Frac(H^*_{\TL\times\C_z^\times}(\pt))\right)[\![\extEff(\XLV)]\!][\![x]\!],
\]
and is regular as a function in $z$ except possibly for poles at
\[
\{0,\infty\} \cup \left\{ \frac{\lambda_\ab}{k} \colon \beta\in\adj(\alpha), k\in\N \right\}.
\]
\item[\textbf{(C2)}]		The principal parts of the restrictions of $\f$ satisfy the following recursion relations: for any $\alpha\in F_\sfL$, $\beta\in\adj(\alpha)$ and $k\in\N$, we have
\[
\Prin_{z=\frac{\lambda_\ab}{k}}\iota_\alpha^*\f = {\paba}_* \left[ q^{k\cdot \dab} \cdot \frac{C_\ab(k)}{-kz+c_1^\TL(L_\ab)} \cdot \pabb^*\iota_\beta^*\f\left( z=\frac{c_1^\TL(L_\ab)}{k} \right) \right]
\]
where $C_\ab(k)$ is an element of $H^*_\TL(\XLV_\acb)_{\loc}$ defined as
\begin{align*}
C_\ab(k)^{-1} =\ &\prod_{c=1}^{k-1}\prod_{ \substack{\delta\colon\mathrm{Chern\ roots} \\ \mathrm{of\ } V_{i_{\alpha,\beta}}} } \left(\delta + p_{\alpha\cup\beta,\alpha}^*\iota_\alpha^*u_{i_{\alpha,\beta}}-\frac{c}{k}c^\TL_1(L_{\alpha,\beta})\right) \\
\cdot &\prod_{i\notin\beta}\prod_{c=1}^{k\cdot u_i(\dab)}\prod_{ \substack{\delta\colon\mathrm{Chern\ roots} \\ \mathrm{of\ } V_i} } \left(\delta + p_{\alpha\cup\beta,\alpha}^*\iota_\alpha^*u_i-\frac{c}{k}c^\TL_1(L_{\alpha,\beta})\right).
\end{align*}
\item[\textbf{(C3)}]		The Laurent expansion of $\iota_\alpha^*\f$ at $z=0$ is a $\Frac(H^*_\TL(\pt))[\![\extEff(\XLV)]\!][\![x]\!]$-valued point of $\cL_{\XLV_\alpha,(N_\alpha,e_\TL^{-1})}$ for any $\alpha\in F_\sfL$.
\end{itemize}
\end{theorem}

The key tool for the proof is the virtual localization (Section \ref{subsubsec:vir_loc}).
The $\TL$-action on $\XLV$ induces a $\TL$-action on the moduli stack $\XLV_{0,n,\cD}$.
We will study the $\TL$-fixed locus $\XLV_{0,n,\cD}^\TL\subset\XLV_{0,n,\cD}$ and its virtual normal bundle in order to compute the virtual localization formula.
In the rest of this section, we will give a proof of Theorem \ref{thm:characterization}.

\begin{remark}
This theorem can be directly generalized to general points on $\cL_{\XLV,\TL}$.
However, since it is enough to characterize $\C[\![\extEff(\XLV]\!][\![x]\!]$-valued points for our purpose, we will not deal with the general case.
\end{remark}

\subsection{Decorated graphs and fixed stable maps}
\label{subsubsec:DG}
We introduce the notion of an \emph{$\sfL$-decorated graph} and describe the stack of $\TL$-fixed stable maps.
We first define a graph $\Gamma_\sfL$ associated to $\sfL$.
We let the set of vertices $V(\Gamma_\sfL)$ equal $F_\sfL$, and the edge joining $\alpha\in F_\sfL$ and $\beta\in F_\sfL$ exists if and only if they are adjacent (in the sense of Definition \ref{def:fixed_locus}).

\begin{definition}[\cite{Brown,Liu,CCIT:mirror,JTY,FL}]
\phantom{A} 
\begin{itemize}
\item[$(1)$] An ($n$-pointed) \emph{$(\sfL,\vecV)$-decorated tree} $\vecG=(\Gamma,\vecalpha,\veck,\veccD,\vecs)$ consists of the following data:
\begin{itemize}
\item a connected acyclic undirected graph $\Gamma$;
\item a graph homomorphism $\vecalpha\colon \Gamma\to\Gamma_\sfL$, called a \emph{label map}, which sends $v\in V(\Gamma)$ to $\alpha_v\in F_\sfL$;
\item an \emph{edge-degree map} $\veck\colon E(\Gamma)\to\Z_{>0}$ which sends $e\in E(\Gamma)$ to a positive integer $k_e$;
\item a \emph{vertex-degree map} $\veccD\colon V(\Gamma)\to\Eff(\XLV)$ which sends $v\in V(\Gamma)$ to an effective curve class $\cD_v\in\Eff(\XLV_{\alpha_v})$;
\item a \emph{marking map} $\vecs\colon\{1,\dots,n\}\to V(\Gamma)$ which sends $i$ to $s_i\in V(\Gamma)$ (this datum is trivial if $n=0$).
\end{itemize}

\item[$(2)$] Let $e\in E(\Gamma)$ and $v\in V(\Gamma)$.
We call the pair $(e,v)$ a \emph{flag} if $v$ incident to $e$.
The set of flags for $\Gamma$ is denoted by $F(\Gamma)$.

\item[$(3)$] For any $v\in V(\Gamma)$, we define $\adj(v)$ to be the set of adjacent vertices of $v$, and $\mrk(v)$ to be the set of markings on $v$:
\begin{align*}
\adj(v) &= \{ v'\in V(\Gamma) \colon \text{there exists an edge between $v$ and $v'$} \}, \\
\mrk(v) &= \vecs^{\,-1}(v).
\end{align*}
We write the valency of $v$ as $\val(v):=|\adj(v)|+|\mrk(v)|$.

\item[$(4)$] For $e\in E(\Gamma)$, define $d_e:=d_{\alpha_v\alpha_{v'}}\in\Eff(\XLV)$ where $v,v'\in V(\Gamma)$ denote the endpoints of $e$.
Define 
\[
\deg(\vecG) = \sum_{e\in E(\Gamma)} k_ed_e + \sum_{v\in V(\Gamma)} \cD_v,
\]
which we call a \emph{degree of $\vecG$}.

\item[$(5)$] We define $\DG_{0,n,\cD}=\DG_{0,n,\cD}(\sfL,\vecV)$ to be a set of all  $n$-pointed $(\sfL,\vecV)$-decorated trees of of degree $\cD$.
\end{itemize}
\end{definition}

We can naturally associate an $n$-pointed decorated graph $(\Gamma,\vecalpha,\veck,\veccD,\vecs)$ of degree $\cD$ to a $\TL$-fixed stable map $[f\colon(C,\bp)\to\XLV]\in\XLV_{0,n,\cD}$ as follows.
Let $E(\Gamma)$ be a set of irreducible components of $C$ sent by $f$ dominantly to a one-dimensional orbit in $\XLV$, and let $V(\Gamma)$ be a set of connected components of the set $f^{-1}(\XLV^\TL)$.
These sets with topological data of $C$ give a tree $\Gamma$.
We write for $e\in E(\Gamma)$ (resp. $v\in V(\Gamma)$) the corresponding subvariety of $C$ as $C_e$ (resp. $C_v$).
We remark that $C_v$ might be just one point.
The other data of $\vecG$ are determined by the following conditions:
\begin{itemize}
\item for $v\in V(\Gamma)$, the image of $C_v$ via $f$ is in $\XLV_{\alpha_v}$;
\item for $e\in E(\Gamma)$, we have $\deg(f|_{C_e}) = k_ed_e$;
\item for $v\in V(\Gamma)$, we have $\deg(f|_{C_v}) = \cD_v$;
\item for $1\leq i\leq n$, the $i$-th marking point is on the component $C_{s_i}$ or equal to $C_{s_i}$ if $C_{s_i}$ is just a one point.
\end{itemize}

For convenience, we introduce the following notation.
\begin{notation}
Let $X$ be a smooth projective variety.
We set $X_{0,1,0} = X$ and $X_{0,2,0} = X$.
When we regard $x\in X$ as a point on $X_{0,1,0}$ or $X_{0,2,0}$, we sometimes denote it as follows:
\begin{align*}
[f\colon(C=\pt,p_1)\to\XLV]&\in X_{0,1,0} \\
[f\colon(C=\pt,p_1,p_2)\to\XLV]&\in X_{0,2,0}
\end{align*}
where in each case the image of $f$ is $x$.
We also define evaluation maps:
\begin{align*}
\ev_1=\id_X\colon X_{0,1,0}\to X, \\
\ev_1=\ev_2=\id_X\colon X_{0,2,0}\to X.
\end{align*}
\end{notation}
We decompose the stack of fixed stable maps $\XLV_{0,n,\cD}^\TL$ with respect to $(\sfL,\vecV)$-decorated graphs.
For a decorated tree $\vecG$ we define a stack $\ovcM_\vecG$ as follows.
First we take and fix a total order $\{1,\dots,A=\#V(\Gamma)\}\to V(\Gamma)$, of which the image of $i$ is denoted by $v_i$, that satisfies the following; for any $1\leq i\leq A$, the full subgraph $\Gamma_i$ of $\Gamma$ formed by $\{ v_1,\dots,v_i \}$ is connected.
This is possible since $\Gamma$ is connected.
Then we recursively define $\ovcM_i$ for $1\leq i\leq A$.
We set $\ovcM_1 = (\XLV_{\alpha_{v_1}})_{0,\val(v_1),\cD_{v_1}}$.
Assume that we know $\ovcM_i$ for some $i<A$.
There is a unique vertex $v\in V(\Gamma_i)$ which is adjacent to $v_{i+1}$ in $\Gamma$.
We write $e\in E(\Gamma)$ for the edge joining $v$ and $v_{i+1}$.
We first take the fiber product with respect to $\ovcM_i\to\XLV_{\alpha_v}$ and $\ev_1\colon\ovcM^{\alpha_v,\alpha_{v_{i+1}}}_{k_e}\to\XLV_{\alpha_v}$.
Here the first morphism is the composition
\[
\ovcM_i	\to	(\XLV_{\alpha_v})_{0,\val(v),\cD_v}	\xrightarrow{\ev_{v_{i+1}}}	\XLV_{\alpha_v}	
\]
where the first map is a canonical projection.
We define $\ovcM_{i+1}$ as the fiber product with respect to 
\[
\ev_v\colon(\XLV_{\alpha_{v_{i+1}}})_{0,\val(v_{i+1}),\cD_{v_{i+1}}}\to\XLV_{\alpha_{v_{i+1}}}
\]
and 
\[
\ovcM_i\times_{\XLV_{\alpha_v}}\ovcM^{\alpha_v,\alpha_{v_{i+1}}}_{k_e}\to\ovcM^{\alpha_v,\alpha_{v_{i+1}}}_{k_e}\xrightarrow{\ev_2}\XLV_{\alpha_{v_{i+1}}}.
\]
Finally, we set $\ovcM_\vecG = \ovcM_A$. 
Note that $\ovcM_\vecG$ is determined independently of the choice of the total order.

There exists a natural morphism $\ovcM_\vecG\to\XLV_{0,n,\cD}^\TL$.
For each point 
\begin{align*}
&\left( \left( \left[ f_e \colon (C_e=\PP^1,\bp_e) \to \XLV \right] \right)_{e\in E(\Gamma)}, \right. \\
&\left. \left( \left[ f_v \colon \left(C_v, \bp_v = \{ p_{v,v'} \}_{v'\in\adj(v)} \amalg \{ p_i \}_{i\in\mrk(v)}\right) \to \XLV_{\alpha_v} \right] \right)_{v\in V(\Gamma)} \right)
\end{align*}
on $\ovcM_\vecG$, we assign a $\TL$-fixed stable map $[f\colon(C,\bp)\to\XLV]$ as follows.
For $e\in E(\Gamma)$, we denote its endpoints as $\{ v_e, v_e'\}$, and set $\bp_e=\{ p_{e,v_e}, p_{e,v_e'} \}$ which satisfies $f_e(p_{e,v_e})=f_{v_e}(p_{v_e,v_e'})$ and $f_e(p_{e,v_e'})=f_{v_e'}(p_{v_e',v_e})$.
We set
\[
C = \left. \left( \coprod_{v\in V(\Gamma)} C_v \amalg \coprod_{e\in E(\vecG)} C_e \right) \right/ \sim
\]
where the equivalence relation $\sim$ is generated by $p_{v_e,v_e'}\sim p_{e,v_e}$ and $p_{v_e',v_e}\sim p_{e,v_e'}$ for each $e\in E(\Gamma)$.
For each $i$ with $1\leq i\leq n$, there are points $[p_i]\in C$, which will be denoted by $p_i$.
The maps $f_v$ and $f_e$ induce a morphism $f\colon(C,\bp)\to\XLV$, which is stable and fixed by $\TL$.
This gives rise to the morphism $\ovcM_\vecG\to\XLV_{0,n,\cD}^\TL$.

In order to obtain a substack of $\XLV_{0,n,\cD}^\TL$ from $\ovcM_\vecG$, we need to take a quotient by $\Aut(\vecG)$, the automorphism group of the decorated graph $\vecG$.

\begin{proposition}
We have an isomorphism of Deligne-Mumford stacks $:$
\[
\XLV_{0,n,\cD}^\TL \cong \coprod_{\vecG\in\DG_{0,n,\cD}} \left[ \left.\ovcM_\vecG\right/\Aut(\vecG) \right].	
\]
\end{proposition}

Finally, we introduce some morphisms related to $\ovcM_\vecG$. 

\begin{definition}
Let $\vecG$ be an $n$-pointed $(\sfL,\vecV)$-decorated tree. 
\begin{itemize}
\item[$(1)$] Define $i_\vecG\colon\ovcM_\vecG\to\XLV_{0,n,\cD}$ to be a composition of the canonical morphism $\ovcM_\vecG\to\XLV_{0,n,\cD}^\TL$ and the embedding $\XLV_{0,n,\cD}^\TL\to\XLV_{0,n,\cD}$.

\item[$(2)$] For $v\in V(\Gamma)$, $\pr_v$ denotes a projection to the component corresponding to $v\in V(\Gamma)$: 
\[
\pr_v\colon\ovcM_\vecG\to(\XLV_{\alpha_v})_{0,\val(v),\cD_v}.
\]
For $e\in E(\Gamma)$, $\pr_e$ denotes a canonical projection
\[
\pr_e\colon\ovcM_\vecG\to\ovcM^{\alpha,\beta}_{k_e}
\] 
where $\alpha$ and $\beta$ are labels of endpoints of $e$.

\item[$(3)$] For $1\leq i\leq n$, define
\[
\ev_{\vecG,i}\colon\ovcM_\vecG\to\XLV_{\alpha_{s_i}}
\]
to be a composition $\ev_i\circ\pr_{s_i}$.
\end{itemize}
\end{definition}

Let $\vecG\in\DG_{0,n+1,\cD}$ be an $(\sfL,\vecV)$-decorated graph.
We write the pull-back to $\ovcM_\vecG$ of the virtual normal bundle of $\cF_\vecG$ as $N^{\vir}_\vecG$.
For any point $[f\colon(C,\bp)\to\XLV]$ on the moduli space $\XLV_{0,n,\cD}$, we have the \emph{tangent-obstruction exact sequence}:
\[
\xymatrix@R=0pt{
0\ar[r]	&	\Hom_{\cO_C}(\Omega_C(\bp),\cO_C)\ar[r]		&	H^0(C,f^*T_{\XLV})\ar[r]	&	T^1_\vecG		&	\\
\ar[r]		&	\Ext^1_{\cO_C}(\Omega_C(\bp),\cO_C)\ar[r]	&	H^1(C,f^*T_{\XLV})\ar[r]	&	T^2_\vecG\ar[r]	&	0	
}
\]
where $T^1_\vecG$ is the tangent space and $T^2_\vecG$ is the obstruction space at $[f]$.
Collecting on $\ovcM_\vecG$ the spaces appearing in the above sequence gives rise to $\TL$-equivariant sheaves $\Aut(C,\bp)_\vecG$, $\Def(f)_\vecG$, $\cT^1_\vecG$, $\Def(C,\bp)_\vecG$, $\Ob(f)_\vecG$ and $\cT^2_\vecG$ respectively, and we have the exact sequence of sheaves:
\[
0\to\Aut(C,\bp)_\vecG\to\Def(f)_\vecG\to\cT^1_\vecG\to\Def(C,\bp)_\vecG\to\Ob(f)_\vecG\to\cT^2_\vecG\to0.
\]
These sheaves are endowed with $\TL$-actions, and all arrows are $\TL$-equivariant.
By taking the moving parts we obtain the following exact sequence:
\[
\xymatrix@R=0pt{
0\ar[r]	&	\Aut(C,\bp)_\vecG^{\mov}\ar[r]	&	\Def(f)_\vecG^{\mov}\ar[r]	&	\cT^{1,\mov}_\vecG		&	\\
\ar[r]		&	\Def(C,\bp)_\vecG^{\mov}\ar[r]	&	\Ob(f)_\vecG^{\mov}\ar[r]	&	\cT^{2,\mov}_\vecG\ar[r]	&	0.	
}
\]
Since $N^{\vir}_\vecG = \cT^{1,\mov}_\vecG \ominus \cT^{2,\mov}_\vecG$, it follows that
\begin{align}
\label{eqn:vir_normal}
e_\TL(N^{\vir}_\vecG) = \frac{e_\TL(\Def(C,\bp)_\vecG^{\mov}) \cdot e_\TL(\Def(f)_\vecG^{\mov})}{e_\TL(\Aut(C,\bp)_\vecG^{\mov})\cdot e_\TL(\Ob(f)_\vecG^{\mov})}.
\end{align}
The vector bundles $\Aut(C,\bp)_\vecG^{\mov}$, $\Def(C,\bp)_\vecG^{\mov}$ and $\Def(f)_\vecG\ominus\Ob(f)_\vecG$ over $\ovcM_\vecG$ can be described as follows \cite{Liu}:
\begin{equation}
\label{eqn:fiber_aut_def_ob}
\begin{split}
\Aut(C,\bp)_\vecG^{\mov}		&=	\bigoplus_{ \substack{(e,v)\in F(\Gamma); \\ |\adj(v)|=1,|\mrk(v)|=0, \cD_v=0} } \pr_e^*\cL^\vee_{e,v},	\\
\Def(C,\bp)_\vecG^{\mov}		&=	\bigoplus_{ \substack{v\in V(\Gamma) \\ \adj(v)=\{e,e'\}, |\mrk(v)|=0, \cD_v=0} } \pr_e^*\left( \cL^\vee_{e,v}\otimes\cL^\vee_{e',v} \right) \\
&\phantom{=} \oplus \bigoplus_{ \substack{(e,v)\in F(\Gamma); \\ \val(v)\geq3 \text{ or } \cD_v\neq0} } \pr_v^*\cL^\vee_{v,e}\otimes \pr_e^*\cL^\vee_{e,v}, \\
\Def(f)_\vecG\ominus\Ob(f)_\vecG	&=	\bigoplus_{ \substack{v\in V(\Gamma); \\ \val(v)\geq3 \text{ or } \cD_v\neq0} } \pr_v^* {\R g_v}_* F_v^* \iota_{\alpha_v}^* T_{\XLV} \oplus \bigoplus_{e\in E(\Gamma)} \pr_e^* {\R g_e}_* F_e^* T_{\XLV} \\
&\phantom{=} \ominus \bigoplus_{ \substack{v\in V(\Gamma); \\ |\adj(v)|=2, |\mrk(v)|=0, \cD_v=0} } \pr_v^* \iota_{\alpha_v}^* T_{\XLV} \\
&\phantom{=} \ominus \bigoplus_{ \substack{(e,v)\in F(\Gamma); \\ \val(v)\geq3 \text{ or } \cD_v\neq0} } \pr_e^* \ev_{e,v}^* T_{\XLV}
\end{split}
\end{equation}
where $\ev_{e,v}\colon\ovcM^ab_{k_e}\to\XLV$ is the evaluation map associated to $(e,v)\in F(\Gamma)$ where $\alpha,\beta$ are the labels of the endpoints of $e$, and $g_v$ and $F_v$ (resp. $g_e$ and $F_e$) fit into the following diagram for the universal map over $(\XLV_{\alpha_v})_{0,\val(v),\cD_v}$ (resp. $\ovcM^\ab_{k_e}$):
\[
\xymatrix{
\cC_v\ar[r]^-{F_v} \ar[d]_-{g_v}	&	\XLV_\alpha	&	\cC_e \ar[r]^-{F_e} \ar[d]_-{g_e}	& \XLV \\
(\XLV_\alpha)_{0,\val(v),\cD_v}	&				&	\ovcM^{\alpha,\beta}_k			& 
}
\]

Let $\f$ be a $\Frac(H^*_\TL(\pt))[\![\extEff(\XLV)]\!][\![x]\!]$-valued point on $\cL_{\XLV,\T}$.
By definition, $\f$ is an element of $H^*_\TL(\XLV)_{\loc}(\!(z^{-1})\!)[\![\extEff(\XLV)]\!][\![x]\!]$ and can be written as
\begin{align*}
\f = -1\cdot z + \bt(z) + \sum_{ \substack{n\geq0, \cD\in\Eff(\XLV) \\ (n,\cD)\neq(0,0),(1,0)} } \frac{\cQ^\cD}{n!} \cdot {\ev_1}_* \left[ \frac{\prod_{i=2}^{n+1} \ev_i^* \bt(\psi_i)}{-z-\psi_1} \cap [\XLV_{0,n+1,\cD}]^{\vir} \right]
\end{align*}
where $\bt(z)\in H^*_\TL(\XLV)_{\loc}[z][\![\extEff(\XLV)]\!][\![x]\!]$ with $\bt(z)|_{(\cQ,x)=0}$.
Using the virtual localization formula (Theorem \ref{thm:vir_loc}), we have
\begin{equation}
\label{eqn:f_restricted}
\iota_\alpha^*\f = -1\cdot z + \iota_\alpha^* \bt(z) + \sum_{ \substack{n\geq0, \cD\in\Eff(\XLV) \\ (n,\cD)\neq(0,0),(1,0)} } \frac{\cQ^\cD}{n!} \sum_{\vecG\in\DG_{0,n+1,\cD}} \iota_\alpha^*\Cont(\vecG)
\end{equation}
where
\[
\Cont(\vecG) := \frac{1}{|\Aut(\vecG)|} \cdot {\ev_1}_* {i_\vecG}_*  \left[ i_\vecG^* \left( \frac{\prod_{i=2}^{n+1} \ev_i^* \bt(\psi_i)}{-z-\psi_1} \right) e_\TL\left(N^{\vir}_\vecG\right)^{-1} \cap \left[\ovcM_\vecG\right]^{\vir} \right].
\]
In order to compute the right-hand side, $\DG_{0,n+1,\cD}$ is divided into three parts.

\begin{definition}
Let $\alpha\in F_\sfL$.
Define
\begin{align*}
\DG^{\alpha,0}_{0,n+1,\cD}	:=& \left\{ \vecG\in\DG_{0,n+1,\cD} \colon \alpha_{s_1}\neq\alpha \right\}, \\
\DG^{\alpha,1}_{0,n+1,\cD}	:=& \left\{ \vecG\in\DG_{0,n+1,\cD} \colon \alpha_{s_1}=\alpha, \val(s_1)=2, \cD_{s_1}=0 \right\}, \\
\DG^{\alpha,2}_{0,n+1,\cD}	:=& \DG_{0,n+1,\cD} \setminus \left(\DG^{\alpha,0}_{0,n+1,\cD} \cup \DG^{\alpha,1}_{0,n+1,\cD}\right), \\
\DG^\alpha_{0,n+1,\cD}		:=& \DG^{\alpha,1}_{0,n+1,\cD} \amalg \DG^{\alpha,2}_{0,n+1,\cD}.
\end{align*}
These sets give decompositions
\begin{align*}
\DG_{0,n+1,\cD} &= \DG^{\alpha,0}_{0,n+1,\cD} \amalg \DG^{\alpha,1}_{0,n+1,\cD} \amalg \DG^{\alpha,2}_{0,n+1,\cD},	\\
\DG_{0,n+1,\cD} &= \coprod_{\alpha\in F_\sfL} \DG^\alpha_{0,n+1,\cD}.
\end{align*}
We call an element of $\DG^{\alpha,i}_{0,n+1,\cD}$ a \emph{graph of type $(\alpha,i)$} or an \emph{$(\alpha,i)$-type graph}.
\end{definition}

\begin{proposition}
\label{prop:cont0}
If $\vecG$ is an $(\alpha,0)$-type graph, its contribution $\iota_\alpha^*\Cont(\vecG)$ equals zero.
\end{proposition}

\begin{proof}
Let $\vecG\in\DG^{\alpha,0}_{0,n+1,\cD}$ and set $\beta = \alpha_{s_1}$.
Consider the following commutative diagram:
\[
\xymatrix{
\ovcM_\vecG\ar[r]^-{i_\vecG}\ar[rd]_-{\ev_{\vecG,1}}	&	\XLV_{0,n+1,\cD}\ar[r]^-{\ev_1}			&	\XLV							\\
									&	\XLV_\beta\ar@{^{(}->}[ru]_-{\iota_\beta}	&	\XLV_\alpha\ar@{^{(}->}[u]_-{\iota_\alpha}			
}
\]
From the diagram, it follows that $\iota_\alpha^*{\ev_1}_*{i_\vecG}_*=\iota_\alpha^*{\iota_\beta}_*{\ev_{\vecG,1}}_*$.
Since $\XLV_\alpha \cap \XLV_\beta = \emptyset$, we have $\iota_\alpha^*{\iota_\beta}_* = 0$. 
\end{proof}

In the following two subsections, we will compute the contributions of the graphs of type $(\alpha,1)$ and type $(\alpha,2)$ separately.

\subsection{Contribution of the $(\alpha,1)$-type graphs}
Let $\vecG=(\Gamma,\vecalpha,\veck,\veccD,\vecs)$ be a decorated graph of type $(\alpha,1)$ satisfying $|\adj(s_1)|=|\mrk(s_1)|=1$.
The graph $\vecG$ is decomposed into two trees, $\vecG_1$ and $\vecG_2 = \vecG^{\alpha,\beta}_{2,k}$.
Here $\beta$ denotes the label of the vertex adjacent to $m_1$, $k$ denotes the degree of the edge attached to $m_1$, and $\vecG^{\alpha,\beta}_{2,k}$ denotes the decorated tree determined by the following data:
\begin{itemize}
\item a $2$-marked tree consisting of two vertices $v_1$ and $v_2$, and one edge;
\item $\alpha_{v_1}=\alpha$ and $\alpha_{v_2}=\beta$;
\item a degree of the edge equals $k$;
\item $\cD_{v_i}=0$ for $i=1,2$;
\item $s_i=v_i$ for $i=1,2$.
\end{itemize}
The graph $\vecG_1\in\DG^\beta_{0,n+1,\cD-k\cdot\dab}$ is obtained by removing from $\vecG$ the vertex $m_1$ and the edge attached to $m_1$, and adding the first marking on the vertex which is linked to $m_1$ in $\vecG$.
Conversely, if we are given $\beta\in\adj(\alpha)$, $k\in\N$ and $\vecG_1\in\DG^\beta_{0,n+1.\cD}$, we can recover $\vecG\in\DG^{\alpha,1}_{0,n+1,\cD+k\cdot\dab}$ by connecting $\vecG^{\alpha,\beta}_{2,k}$ and $\vecG_1$ by clutching the second marking of $\vecG^{\alpha,\beta}_{2,k}$ with the first marking of $\vecG_1$.
Summarizing, we have the following bijections.
\begin{lemma}
There is a bijection$:$
\begin{align*}
\Phi_1\colon \coprod_{\beta\in\adj(\alpha)} \coprod_{k\in\N} \coprod_{n=0}^\infty \coprod_{\cD\in\Eff(\XLV)} \DG^\beta_{0,n+1,\cD} \times \{ \vecG^{\alpha,\beta}_{2,k} \} \stackrel{\sim}{\longrightarrow} \coprod_{ \substack{n\geq0, \cD\in\Eff(\XLV) \\ (n,\cD)\neq(0,0),(1,0)} } \DG^{\alpha,1}_{0,n+1,\cD}.
\end{align*}
\end{lemma}

We let $\vecG\in\DG^{\alpha,1}_{0,n+1,\cD}$ and write as $\Phi_1(\vecG)=(\vecG_1,\vecG_2=\vecG^{\alpha,\beta}_{2,k})$.
By definition we have $\ovcM_{\vecG_2} = \ovcM^\ab_k$ for some $\beta\in\adj(\alpha)$ and $k\in\N$.
We denote the canonical morphism $\ovcM^\ab_k\to\XLV_\acb$ by $\pi$.
There exist two natural morphisms
\[
\ovcM_\vecG\to\ovcM_{\vecG_1} \hspace{30pt}\text{and}\hspace{30pt} \ovcM_\vecG\to\ovcM_{\vecG_2}
\]
which are denoted by $\pr_1$ and $\pr_2$ respectively.
By definition, $\ovcM_\vecG$ fits into the following diagram
\begin{equation}
\label{eqn:decomp1}
\xymatrix{
\ovcM_\vecG\ar[rr]^{\pr_2}	\ar[d]^{\pr_1}&&	\ovcM^\ab_k\ar[d]^{\ev_{\vecG_2,2}}\ar[rr]^\pi	&&	\XLV_\acb\ar[dll]^{\ \ \pabb}	\\
\ovcM_{\vecG_1}\ar[rr]^{\ev_{\vecG_1,1}}	&&	\XLV_\beta	&&
}
\end{equation}
where the square in the diagram is Cartesian.

\begin{lemma}
\label{lem:eqns1}
Let $\vecG\in\DG^{\alpha,1}_{0,n+1,\cD}$ and let $(\vecG_1,\vecG_2=\vecG^\ab_{2,k})$ be as above.
\begin{itemize}
\item[$(1)$] 
$|\Aut(\vecG)| = |\Aut(\vecG_1)|$.
\item[$(2)$]
$i_\vecG^*\psi_1 = -\pr_2^*\pi^*c^\TL_1(L_{\alpha,\beta})/k$.
\item[$(3)$]
$i_\vecG^*\psi_j = \pr_1^*i_{\vecG_1}^*\psi_j$ for $2\leq j\leq n+1$.
\end{itemize}
\end{lemma}

\begin{proof}
The first equality is obvious.
For $1\leq i\leq n+1$, let $\cL_i$ (resp. $\cL'_i$, $\cL''_i$) denote the $i$-th universal cotangent line bundle over $\XLV_{0,n+1,\cD}$ (resp. $\XLV_{0,n+1,\cD-k\cdot\dab}$, $\XLV_{0,2,k\cdot\dab}$).
It can be seen that there exist the following fiber diagrams:
\[
\xymatrix@!C{
i_\vecG^*\cL_1\ar[r]\ar[d]	&	i_{\vecG_2}^*\cL''_1\ar[d]	&	\left(i_{\vecG_2}^*\cL''_1\right)^{\otimes (-k)}\ar[r]\ar[d]	&	L_\ab\ar[d]	\\
\ovcM_\vecG\ar[r]^-{\pr_2}	&	\ovcM^\ab_k				&	\ovcM^\ab_k\ar[r]^-{\pi}				&	\XLV_\acb
}
\]
These give the second formula.
The third one follows from the fiber diagram
\[
\xymatrix@!C{
i_\vecG^*\cL_i\ar[r]\ar[d]	&	i_{\vecG_1}^*\cL'_i\ar[d]	\\
\ovcM_\vecG\ar[r]^-{\pr_1}	&	\ovcM_{\vecG_1}
}
\]
for $i=2,\dots,n+1$.
\end{proof}

\begin{lemma}
\label{lem:proj_formula1}
Let $\vecG$ and $(\vecG_1,\vecG_2)$ be as in Lemma \ref{lem:eqns1}. 
For $\omega\in H^*(\ovcM_{\vecG_1})$ and $\eta\in H^*(\XLV_\acb)$, it holds that 
\begin{multline*}
\iota_\alpha^*{\ev_1}_*{i_\vecG}_* \left[ \left( \pr_1^*\omega \cdot \pr_2^*\pi^*\eta \right) \cap \left[\ovcM_\vecG\right]^{\vir} \right] \\
= \frac{1}{k} \cdot e_\TL\left(N_\alpha\right) \cdot {\paba}_* \left[ \eta \cdot \pabb^*\left( e_\TL\left(N_\beta\right)^{-1} \cdot \iota_\beta^* {\ev_1}_* {i_{\vecG_1}}_* \left( \omega \cap \left[\ovcM_{\vecG_1}\right]^{\vir} \right) \right) \right].
\end{multline*}
\end{lemma}

\begin{proof}
Since $\ev_1\circ i_\vecG = \iota_\alpha\circ\paba\circ\pi\circ\pr_2$ and $\iota_\alpha^*{\iota_\alpha}_* = e_\TL(N_\alpha)$, we have
\begin{align*}
\iota_\alpha^*{\ev_1}_*{i_\vecG}_* &= \iota_\alpha^* {\iota_\alpha}_* {\paba}_* \pi_* {\pr_2}_*	\\
&= e_\TL\left(N_\alpha\right) \cdot {\paba}_* \pi_* {\pr_2}_*. 
\end{align*}
From the diagram \eqref{eqn:decomp1} it follows that
\[
{\pr_2}_* \pr_1^* = \pi^* \pabb^* {\ev_{\vecG_1,1}}_*.
\]
Since $\ev_1\circ i_{\vecG_1} = \iota_\beta\circ\ev_{\vecG_1,1}$ and $\iota_\beta^*{\iota_\beta}_* = e_\TL(N_\beta)$, we have
\begin{align*}
{\ev_{\vecG_1,1}}_* &= e_\TL\left(N_\beta\right)^{-1} \cdot \iota_\beta^* {\iota_\beta}_* {\ev_{\vecG_1,1}}_*	\\
&= e_\TL\left(N_\beta\right)^{-1} \cdot \iota_\beta^* {\ev_1}_* {i_{\vecG_1}}_*.
\end{align*}
Finally, from the construction we have
\[
\pr_1^* \left[\ovcM_{\vecG_1}\right]^{\vir} = k \cdot \left[\ovcM_\vecG\right]^{\vir}.
\]
The desired equality follows from these equations and the projection formula.
\end{proof}

\begin{lemma}
\label{lem:vir_normal1}
Let $\vecG\in\DG^{\alpha,1}_{0,n+1,\cD}$ and let $(\vecG_1,\vecG_2)$ be as in Lemma \ref{lem:eqns1}.
Assume that $\vecG\neq\vecG^\ab_{1,k}, \vecG^\ab_{2,k}$.
Then the following equality holds$:$
\[
\frac{\pr_1^*e_\TL(N^{\vir}_{\vecG_1})}{e_\TL(N^{\vir}_\vecG)} = \frac{k}{-\pr_2^*\pi^*c_1^\TL(L_{\alpha,\beta})-k\pr_1^*i_{\vecG_1}^*\psi_1} \cdot \pr_2^* \pi^* \left[ \frac{p_{\alpha\cup\beta,\beta}^*e_\TL(N_\beta)}{p_{\alpha\cup\beta,\alpha}^*e_\TL(N_\alpha)} \cdot C_{\alpha,\beta}(k) \right].
\]
Here we use the notation $C_\ab(k)$ introduced in Theorem \ref{thm:characterization}.
\end{lemma}

\begin{proof}
Thanks to the formula \eqref{eqn:vir_normal}, we can divide the left-hand side into three parts:
\begin{equation}
\label{eqn:three_parts1}
\begin{split}
\frac{\pr_1^*e_\TL(N^{\vir}_{\vecG_1})}{e_\TL(N^{\vir}_\vecG)} =&\ \frac{e_\TL(\Aut(C,\bp)^{\mov}_\vecG)}{\pr_1^*e_\TL(\Aut(C,\bp)^{\mov}_{\vecG_1})} \cdot \frac{\pr_1^*e_\TL(\Def(C,\bp)^{\mov}_{\vecG_1})}{e_\TL(\Def(C,\bp)^{\mov}_\vecG)}	\\
\cdot&\ \frac{\pr_1^*e_\TL(\Def(f)^{\mov}_{\vecG_1} \ominus \Ob(f)^{\mov}_{\vecG_1})}{e_\TL(\Def(f)^{\mov}_\vecG \ominus \Ob(f)^{\mov}_\vecG)}.
\end{split}
\end{equation}
We will compute them separately.

From \eqref{eqn:fiber_aut_def_ob}, we can see that $\Aut(C,\bp)^{\mov}_\vecG \cong \pr_1^*\Aut(C,\bp)^{\mov}_{\vecG_1}$ and 
\[
\Def(C,\bp)^{\mov}_\vecG \ominus \pr_1^*\Def(C,\bp)^{\mov}_{\vecG_1} = \pr_2^*i_{\vecG_2}^*{\cL''_2}^\vee \otimes \pr_1^*i_{\vecG_1}^*{\cL'_1}^\vee.
\]
(We use the notation in the proof of Lemma \ref{lem:eqns1}.)
Hence the first part of \eqref{eqn:three_parts1} equals $1$, and the second one is equal to
\[
\frac{1}{-\pr_2^*i_{\vecG_2}^*\psi_1 - \pr_1^*i_{\vecG_1}^*\psi_1} = \frac{k}{-\pr_2^*\pi^*c_1^\TL(L_{\alpha,\beta})-k\pr_1^*i_{\vecG_1}^*\psi_1}.
\]
Here we use Lemma \ref{lem:eqns1} (2).

Let $v\in V(\Gamma)$ be the vertex adjacent to $m_1$, and write the unique edge joining $m_1$ and $v$ as $e\in E(\Gamma)$.
We write the universal curve over $\ovcM_{\vecG_2} = \ovcM^\ab_k$ as follows:
\[
\xymatrix{
\cC^{\alpha,\beta}_k \ar[r]^-F \ar[d]_-g & \XLV \\
\ovcM^{\alpha,\beta}_k \ar@/^24pt/[u]^{s_1,\ s_2}
}
\]
Again from \eqref{eqn:fiber_aut_def_ob}, we have
\[
\left(\Def(f)_\vecG \ominus \Ob(f)_\vecG\right) \ominus \pr_1^* \left(\Def(f)_{\vecG_1} \ominus \Ob(f)_{\vecG_1}\right) = \pr_2^* \R g_* F^* T_{\XLV} \ominus \pr_2^*s_2^* F^* T_{\XLV}
\]
where $\R g_*$ denotes the $K$-theoretic pushforward.
Since we have the exact sequences
\[
\xymatrix@R=0pt@C=18pt{
0\ar[r]	&	T_{\XLV/B}\ar[r]		&	T_{\XLV}\ar[r]						&	p^*T_B\ar[r]		&	0,	\\
0\ar[r]	&	\cO^{\oplus K}\ar[r]	&	\bigoplus_{i=1}^N (V_i\otimes L_i)\ar[r]	&	T_{\XLV/B}\ar[r]	&	0,	
}
\]
and $\TL$ acts trivially on $p^*T_B$ and $\cO$, we have
\[
\left(\R g_* F^* T_{\XLV}\right)^{\mov} = \bigoplus_{i=1}^N \left(\R g_* F^* (V_i\otimes L_i)\right)^{\mov}.
\]
We set
\begin{align*}
\cU_\alpha &= \cC^{\alpha,\beta}_k\setminus \text{Im}(s_2), 		&g_\alpha&\colon\cU_\alpha\hookrightarrow\cC^\ab_k\xrightarrow{g}\ovcM^\ab_k,			\\
\cU_\beta &= \cC^{\alpha,\beta}_k\setminus \text{Im}(s_1),		&g_\beta&\colon\cU_\beta\hookrightarrow\cC^\ab_k\xrightarrow{g}\ovcM^\ab_k,				\\
\cU_{\alpha\beta} &= \cU_\alpha\times_{\cC^\ab_k}\cU_\beta,	&g_{\alpha\beta}&\colon\cU_{\alpha\beta}\hookrightarrow\cC^\ab_k\xrightarrow{g}\ovcM^\ab_k.	
\end{align*}
Note that $\cU_\alpha \cong \cL^\vee_1|_{\ovcM^\ab_k}$ and $\cU_\beta \cong \cL^\vee_2|_{\ovcM^\ab_k}$, and $\cU=\{ \cU_\alpha, \cU_\beta \}$ is an open covering of $\cC^{\alpha,\beta}_k$.
It is easy to see that
\begin{align*}
{g_\alpha}_*(F|_{\cU_\alpha})^*(V_i\otimes L_i)	&=\pi^*p_\acb^*V_i\otimes\pi^*\paba^*\iota_\alpha^*L_i\otimes\bigoplus^\infty_{c=0}\left( \cL^\vee_1|_{\ovcM^\ab_k} \right)^{\otimes c},	\\
{g_\beta}_*(F|_{\cU_\beta})^*(V_i\otimes L_i)	&=\pi^*p_\acb^*V_i\otimes\pi^*\pabb^*\iota_\beta^*L_i\otimes\bigoplus^\infty_{c=0}\left( \cL^\vee_2|_{\ovcM^\ab_k} \right)^{\otimes c},	\\
{g_{\alpha\beta}}_*(F|_{\cU_{\alpha\beta}})^*(V_i\otimes L_i)	&=\pi^*p_\acb^*V_i\otimes\pi^*\pabb^*\iota_\beta^*L_i\otimes\bigoplus^\infty_{c=-\infty}\left( \cL^\vee_1|_{\ovcM^\ab_k} \right)^{\otimes c}.	\\
\end{align*}
Hence we have
\begin{align*}
&\ e_\TL\left(\R g_*F^*(V_i\otimes L_i)\right) \\
=&\ e_\TL\left({g_\alpha}_*(F|_{\cU_\alpha})^*(V_i\otimes L_i)\oplus{g_\beta}_*(F|_{\cU_\beta})^*(V_i\otimes L_i)\ominus{g_{\alpha\beta}}_*(F|_{\cU_{\alpha\beta}})^*(V_i\otimes L_i)\right)	\\
=&\ \pi^* \prod_{ \substack{\delta\colon\mathrm{Chern\ roots} \\ \mathrm{of\ } V_i} } \frac{\prod^{k\cdot u_i(\dab)}_{c=-\infty}}{\prod^{-1}_{c=-\infty}}\left( \delta + \paba^*\iota_\alpha^*u_i - \frac{c}{k}c_1(L_\ab) \right).
\end{align*}
The moving part can be described as 
\[
e_\TL\left( \bigoplus^N_{i=1} \left(\R g_*F^*(V_i\otimes L_i)\right)^{\mov} \right) = \pi^*\left[ \frac{\paba e_\TL(N_\alpha)}{C_\ab(k)} \right].
\]
On the other hand, we have
\begin{align*}
e_\TL\left( \left( s_2^*F^*T_{\XLV} \right)^{\mov} \right) 
&= e_\TL\left( \bigoplus^N_{i=1} s_2^*F^*(V_i\otimes L_i)^{\mov} \right)		\\
&= \pi^* \prod_{i\notin\beta} \left( \delta + \pabb^*\iota_\beta^*u_i \right)	\\
&= \pi^*\pabb^*e_\TL(N_\beta).
\end{align*}
These computations give the desired formula.
\end{proof}

By performing calculations similar to those in the previous proof, we can establish the following formulas.

\begin{lemma}
\label{lem:leading}
Let $\beta\in\adj(\alpha)$ and $k\in\N$.
We have
\begin{align*}
\iota_\alpha^*\Cont_{\vecG^{\alpha,\beta}_{1,k}}(z) &= {\paba}_* \left[ \frac{C_\ab(k)}{-kz+c_1^\TL(L_\ab)} \cdot \left( -\frac{c_1^\TL(L_\ab)}{k} \right) \right], \\
\iota_\alpha^*\Cont_{\vecG^{\alpha,\beta}_{2,k}}(z) &= {\paba}_* \left[ \frac{C_\ab(k)}{-kz+c_1^\TL(L_\ab)} \cdot \pabb^*\iota_\beta^*\bt\left( \frac{c_1^\TL(L_\ab)}{k} \right) \right].
\end{align*}
\end{lemma}

Using the above lemmas, we can compute the contributions of the graphs of type $(\alpha,1)$.

\begin{proposition}
\label{prop:cont1}
\begin{multline*}
\sum_{ \substack{n\geq0, \cD\in\Eff(\XLV) \\ (n,\cD)\neq(0,0),(1,0)} } \frac{\cQ^\cD}{n!} \sum_{\vecG\in\DG^{\alpha,1}_{0,n+1,\cD}} \iota_\alpha^*\Cont(\vecG) \\
= \sum_{\beta\in\adj(\alpha)} \sum_{k\in\N} {p_{\alpha\cup\beta,\alpha}}_* \left[ q^{k\cdot \dab} \cdot \frac{C_\ab(k)}{-kz+c_1^\TL(L_\ab)} \cdot \pabb^*\iota_\beta^*\f\left( z=\frac{c_1^\TL(L_\ab)}{k} \right) \right]
\end{multline*}
\end{proposition}

\begin{proof}
To begin with, we rewrite the left-hand side using the bijection $\Phi_1$ as follows:
\begin{align*}
&	\sum_{ \substack{n\geq0, \cD\in\Eff(\XLV) \\ (n,\cD)\neq(0,0),(1,0)} } \frac{\cQ^\cD}{n!} \sum_{\vecG\in\DG^{\alpha,1}_{0,n+1,\cD}} \iota_\alpha^*\Cont_\vecG(z)	\\
=\ &	\sum_{\beta\in\adj(\alpha)} \sum_{k\in\N} q^{k\cdot \dab} \left( \iota_\alpha^*\Cont_{\vecG^{\alpha,\beta}_{1,k}}(z) + \iota_\alpha^*\Cont_{\vecG^{\alpha,\beta}_{2,k}}(z) \right)	\\
+\ &	\sum_{\beta\in\adj(\alpha)} \sum_{k\in\N} q^{k\cdot \dab} \sum_{ \substack{n\geq0, \cD\in\Eff(\XLV) \\ (n,\cD)\neq(0,0),(1,0)} } \frac{\cQ^\cD}{n!} \sum_{\vecG_1\in\DG^\beta_{0,n+1,\cD}} \iota_\alpha^*\Cont_{\Phi_1(\vecG_1,\vecG_{2,k}^\ab)}(z).
\end{align*}
By using Lemma \ref{lem:eqns1}, Lemma \ref{lem:proj_formula1} and Lemma \ref{lem:vir_normal1}, we have
\begin{align*}
\iota_\alpha^*\Cont_\vecG(z) = {\paba}_* \left[ \frac{C_\ab(k)}{-kz+c_1^\TL(L_\ab)} \cdot \pabb^* \iota_\beta^*\Cont_{\vecG_1}\left(z=\frac{c_1^\TL(L_\ab)}{k}\right) \right]
\end{align*}
where $\vecG\in\DG^{\alpha,1}_{0,n+1,\cD}$ and $\Phi_1(\vecG_1,\vecG^\ab_{2,k})=\vecG$.
By combining the above equations with Lemma \ref{lem:leading}, we obtain the desired equality.
\end{proof}

\subsection{Contribution of the $(\alpha,2)$-type graphs}

The contribution of the $(\alpha,2)$-type graphs can be computed as follows.

\begin{proposition}[\cite{FL}]
\label{prop:cont2}
It holds that
\begin{multline*}
\sum_{ \substack{n\geq0, \cD\in\Eff(\XLV) \\ (n,\cD)\neq(0,0),(1,0)} } \frac{\cQ^\cD}{n!} \sum_{\vecG\in\DG^{\alpha,2}_{0,n+1,\cD}} \iota_\alpha^*\Cont(\vecG) \\
= \sum_{ \substack{n\geq0, \cD\in\Eff(\XLV_\alpha) \\ (n,\cD)\neq(0,0),(1,0)} } \frac{\cQ^{\ {\iota_\alpha}_*\cD}}{n!} \cdot e_\TL(N_\alpha) \cdot {\ev_1}_* \left[ \frac{\prod_{i=2}^{n+1}\ev_i^*\bt^\alpha(\psi_i)}{-z-\psi_1} \cdot e_\TL((N_\alpha)_{0,n+1,\cD})^{-1} \right. \\
\left. \cap \left[ (\XLV_\alpha)_{0,n+1,\cD} \right]^{\vir} \right]
\end{multline*}
where 
\begin{multline*}
\bt^\alpha(z) = \iota_\alpha^*\bt(z) \\
+ \sum_{\beta\in\adj(\alpha)} \sum_{k\in\N} {p_{\alpha\cup\beta,\alpha}}_* \left[ q^{k\cdot \dab} \cdot \frac{C_\ab(k)}{-kz+c_1^\TL(L_\ab)} \cdot \pabb^*\iota_\beta^*\f\left( \frac{c_1^\TL(L_\ab)}{k} \right) \right].
\end{multline*}
\end{proposition}

\begin{proof}
This follows from the argument in \cite[Section 3.2]{FL}.
We only give a sketch of the proof.

Let $\vecG=(\Gamma,\vecalpha,\veck,\veccD,\vecs)$ be a decorated graph of type $(\alpha,2)$, and set $m=\val(s_1)-1$.
Then $\vecG$ can be decomposed into the graph $\vecG_1$ and $m$ $(\alpha,1)$-type graphs $\vecG_2,\dots,\vecG_{m+1}$.
Here $\vecG_1$ is the decorated graph given by the following data:
\begin{itemize}
\item a tree $\Gamma_1$ consisting of one vertex $v$ and $m+1$ markings;
\item $\alpha_v=\alpha$ and $\cD_v$ equals the degree of the vertex $s_1\in\Gamma$.
\end{itemize}
Let $\vecG_j\in\DG^{\alpha,1}_{0,n_j+1,\cD_j}$ for $2\leq j\leq m+1$.

By construction, $\ovcM_\vecG$ fits into the following fiber diagram:
\[
\xymatrix{
\ovcM_\vecG\ar[r]\ar[d]		&	\prod_{j=2}^{m+1}\ovcM_{\vecG_j}\ar[d]	\\
\ovcM_{\vecG_1}\ar[r]		&	(\XLV_\alpha)^{m}
}
\]
where the morphism $\ovcM_{\vecG_1}\to(\XLV_\alpha)^{m}$ is given by the evaluation maps $\ev_{\vecG_1,2},\dots,\ev_{\vecG_1,m+1}$ and the morphism $\prod_{j=2}^{m+1}\ovcM_{\vecG_j}\to(\XLV_\alpha)^{m}$ is given by the evaluation maps $\ev_{\vecG_2,1},\dots,\ev_{\vecG_{m+1},1}$.
Hence the integral over $\ovcM_\vecG$ can be computed by the integral over $\ovcM_{\vecG_1}$ with inputs given by the integrals over $\ovcM_2,\dots,\ovcM_{m+1}$.
By taking the sum over all $(\alpha,1)$-type decorated graphs, each input becomes the summation of the contribution of the $(\alpha,1)$-type graphs, which equals $\bt^\alpha(\psi_j)$; see Proposition \ref{prop:cont1}.

The term $e_\T((N_\alpha)_{0,n+1,\cD})^{-1}$ comes from the comparison of the Euler classes of the virtual normal bundles:
\[
\frac{\prod^{n_1+1}_{j=2}\pr^*_j e_\TL(N^{\vir}_{\vecG_j})}{e_\TL(N^{\vir}_\vecG)} = \left[ {\prod_j}' \frac{\pr_j^*\ev_{\vecG_j,1}^*e_\TL(N_\alpha)}{-\pr_1^*i_{\vecG_1}^*\psi_{1,j}-\pr_j^*i_{\vecG_j}^*\psi_{j,1}} \right] \cdot \frac{1}{\pr_1^*e_\TL\left( \left( N_\alpha \right)_{0,n_1+1,\cD_1} \right)}
\]
where the symbol $\prod_j'$ means taking the product over all $2\leq j\leq n_1+1$ such that $(n_j,\cD_j)\neq(1,0)$, $\pr_j\colon\ovcM_\vecG\to\ovcM_{\vecG_j}$ denotes the canonical projection, and $\psi_{j,i}$, $2\leq j\leq m$, (resp. $\psi_{1,i}$) denotes the $\T$-equivariant first Chern class of the $i$-th universal cotangent line bundle for $\XLV_{0,n_j+1,\cD_j}$ (resp. $\XLV_{0,m+1,\cD_0}$).
\end{proof}

\subsection{Proof of Theorem \ref{thm:characterization}}
\label{subsec:proof_characterization}
We now prove Theorem \ref{thm:characterization}.
We write $S=\extEff(\XLV)$.
We first assume that $\f$ be a $\Frac(H^*_\TL(\pt))[\![S]\!][\![x]\!]$-valued point on $\cL^\TL_{\XLV}$ and write
\[
\f = -1\cdot z + \bt(z) + \sum_{ \substack{n\geq0, \cD\in\Eff(\XLV) \\ (n,\cD)\neq(0,0),(1,0)} } \frac{\cQ^\cD}{n!} \cdot {\ev_1}_* \left[ \frac{\prod_{i=2}^{n+1} \ev_i^* \bt(\psi_i)}{-z-\psi_1} \cap [\XLV_{0,n+1,\cD}]^{\vir} \right]
\]
where $\bt(z)\in H^*_\TL(\XLV)_{\loc}[z][\![S]\!][\![x]\!]$ with $\bt(z)|_{(\cQ,x)=0}$.
By combining the equation \eqref{eqn:f_restricted}, Proposition \ref{prop:cont0}, Proposition \ref{prop:cont1} and Proposition \ref{prop:cont2}, we can see that, via Laurent expansion at $z=0$, $\iota_\alpha^*\f$ can be interpreted as a $\Frac(H^*_\TL(\pt))[\![S]\!][\![x]\!]$-valued point of the twisted cone $\cL_{\XLV_\alpha,(N_\alpha,e_\TL^{-1})}$ whose non-negative part as a $z$-series equals
\begin{equation}
\label{eqn:non-negative}
\iota_\alpha^*\f - \Prin_{z=0} (\iota_\alpha^*\f) = -1\cdot z + \bt^\alpha(z),
\end{equation}
which implies \textbf{(C3)}.
Since the coefficients of $\Prin_{z=0} (\iota_\alpha^*\f)$ (as a formal power series in $(\cQ,x)$) are all polynomials in $z^{-1}$, the poles of $\iota_\alpha^*\f$ except for the pole at $z=0$ come from \eqref{eqn:non-negative}.
This observation with the explicit formula for $\bt^\alpha(z)$ in Proposition \ref{prop:cont2} implies \textbf{(C1)} and \textbf{(C2)}.

Conversely, we assume that $\f\in\cH_{\XLV}^\TL[\![S]\!][\![x]\!]$ satisfies $\f|_{(\cQ,x)=0}=-1\cdot z$, \textbf{(C1)}, \textbf{(C2)} and \textbf{(C3)}.
It is enough to show that, under these assumptions, $\f$ (or equivalently $\{\iota_\alpha^*\f\}_{\alpha\in F_\sfL}$) is uniquely determined by its non-negative part $z^{-1}\Prin_{z=\infty}(z\f)$.
For $\cD\in\extEff(\XLV)$, we define a degree of $\cQ^\cD$ by using the valuation introduced in Section \ref{subsec:effective}, which can be assumed to be an integer, and define $\deg(x_i)=i$ for $i\geq1$.
For any $(\cQ,x)$-series $\bg$, we write the homogeneous part of $\bg$ of degree $n$ as $(\bg)_n$.

We proceed by induction on the degree of $(\cQ,x)$.
We assume that we know $\f_{\leq n} := \sum_{m=0}^n \f_m$ for some integer $n$.
From \textbf{(C1)}, we can write
\[
(\iota_\alpha^*\f)_{n+1} = z^{-1}\Prin_{z=\infty}(z\cdot\iota_\alpha^*\f)_{n+1} + \Prin_{z=0}(\iota_\alpha^*\f)_{n+1} + \sum_{\beta\in\adj(\alpha)}\sum_{k\in\N} \Prin_{z=\frac{\lambda_\ab}{k}}(\iota_\alpha^*\f)_{n+1}.
\]
In the right-hand side, the third term can be computed from $\f_{\leq n}$ thanks to \textbf{(C2)}, while the second term is determined from 
\[
z^{-1}\Prin_{z=\infty}(z\cdot\iota_\alpha^*\f)_{\leq n+1} + \sum_{\beta\in\adj(\alpha)}\sum_{k\in\N} \Prin_{z=\frac{\lambda_\ab}{k}}(\iota_\alpha^*\f)_{\leq n+1}
\]
by \textbf{(C3)}. 
By repeating this procedure, we can completely determine $\f$ from $z^{-1}\Prin_{z=\infty}(z\f)$ by using \textbf{(C1)}, \textbf{(C2)} and \textbf{(C3)}.
This proves Theorem \ref{thm:characterization}.

\section{Mirror theorem for a product of projective bundles}
\label{sec:twist_mirror}
In this section, we construct a twisted $I$-function for a product of projective bundles each coming from a vector bundle.
The proof is based on the proof of the mirror theorem for a projective bundle \cite[Theorem 1.1]{IK:quantum}.
This section is independent of the previous section.
By combining Theorem \ref{thm:characterization} with the mirror theorem (Theorem \ref{thm:twist_mirror}), we will establish the main result in the next section.

\subsection{Statement}
\label{subsec:twist_FT}
We begin with the following data:
\begin{itemize}
\item a smooth toric data $\sfL=(\LL^\vee,D\colon\Z^K\to\LL^\vee,\omega)$ with $\rank(\LL^\vee)=K$;
\item a complex smooth projective variety $B$;
\item $N$ vector bundles $V_1,\dots V_K,W_{K+1},\dots,W_N$ over $B$ whose duals are globally generated;
\item $D_{K+1},\dots,D_N\in\LL^\vee$.
\end{itemize}
In this case $\XLV$ is a fiber product of projective bundles $\PP(V_1),\dots,\PP(V_K)$ over $B$.
Due to Proposition \ref{prop:cohomology}, its cohomology can be written as
\[
H^*_\TL(\XLV) \cong \left.H^*_\TL(B) [u_1,\dots,u_K] \right/ \left( e_\TL(V_1\otimes L_1),\dots,e_\TL(V_K\otimes L_K) \right).
\]
For $K+1\leq i\leq N$, we define
\[
u_i = -\lambda_i + \sum_{j=1}^K D_j^\vee(D_i) \cdot (u_j+\lambda_j)
\]
where $\{D_i^\vee\}_{i=1}^K\subset\LL$ denotes a dual basis of $\{D_i\}_{i=1}^K$.
We construct the $\TL$-equivariant vector bundle $\cW\to\XLV$ as follows:
\[
\cW := \left. \left( \cU_\sfL(\vecV) \times_B \bigoplus_{i=K+1}^N W_i \right) \right/ \KL
\]
where $\TL = (\C^\times)^N$ acts on $\bigoplus_{i=1}^K V_i \oplus \bigoplus_{i=K+1}^N W_i$ diagonally, and $\KL = \Hom(\LL^\vee,\C^\times)$ acts on $W_i$ via the character given by $D_i$ for $K+1\leq i\leq N$.
Using the notation in Section \ref{subsec:effective}, we can write $\cW=\bigoplus_{i=K+1}^N W_i(D_i)$.

We set $V=\bigoplus_{i=1}^K V_i$, $W=\bigoplus_{i=K+1}^N W_i$ and $\TLd$ to be the copy of $\TL$.
We consider the diagonal action of $\TLd$ on $V\oplus W$, and write $\mu_1,\dots,\mu_N$ for the equivariant parameters for $\T'$.
We take a function $I_{V\oplus W}^\mu(x,z)\in H^*_\TLd(V\oplus W)[z,z^{-1}][\![\Eff(B)]\!][\![x]\!]$ such that $(zI_{V\oplus W}^\mu)|_{z\mapsto-z}$ is a $H^*_\TLd(\pt)[\![\Eff(B)]\!][\![x]\!]$-valued point of $\cL_{V\oplus W, \TLd}$ where $x$ be a set of formal parameters.
We will prove the following.
\begin{theorem}
\label{thm:twist_mirror}
Let $\sfL$, $\vecV$, $\vecW$, $\cW$ and $I_{V\oplus W}^\mu$ be as above.
Define the function $(I_{V\oplus W}^\mu)_{\tw}\sphat(t,x,y,z)$ to be
\begin{multline*}
(I_{V\oplus W}^\mu)_{\tw}\sphat(t,x,y,z) = e^{\sum_{i=1}^N t_i u_i/z} \sum_{\ell\in\LL} \frac{\tq^\ell e^{\sum_{i=1}^N D_i(\ell)\cdot t_i}}{\prod_{i=1}^K \prod_{c=1}^{D_i(\ell)} \prod_{ \substack{\delta\colon\mathrm{Chern\ roots} \\ \mathrm{of\ } V_i} } (u_i + \delta + cz)}	\\
\cdot \frac{I_{V\oplus W}^{u + D(\ell)z}(x,z)}{\prod_{i=K+1}^N \prod_{c=1}^{D_i(\ell)} \prod_{ \substack{\delta\colon\mathrm{Chern\ roots} \\ \mathrm{of\ } W_i} } (u_i + \delta + cz)}
\end{multline*}
where $\tq$ is a formal variable for $\C[\![\LL_\eff]\!]$, and $I_{V\oplus W}^{u + D(\ell)z}$ denotes the function $I_{V\oplus W}^\mu$ with replaced $\mu_i$ with $u_i+D_i(\ell)z$ for $1\leq i\leq N$.
Then $-z(I_{V\oplus W}^\mu)_{\tw}\sphat(t,x,y,-z)$ is a $\Frac(H^*_\TL(\pt))[\![ \extEff(\XLV) ]\!] [\![ t,x,y ]\!]$-valued point of the twisted Givental cone $\cL_{\XLV,(\cW,e_\lambda^{-1})}$.
\end{theorem}

\begin{remark}
A priori the function $(I_{V\oplus W}^\mu)_{\tw}\sphat$ belongs to 
\[
H^*(\XLV)(\lambda)(\!(z)\!)[\![\Eff(B)\oplus\LL_\eff]\!][\![t,x]\!].
\]
We implicitly use the identification in \eqref{eqn:extEff} and interpret this function as an element of 
\[
H^*(\XLV)(\lambda)(\!(z)\!)[\![\extEff(\XLV)]\!][\![t,x,y]\!].
\]
\end{remark}

\begin{remark}
\label{rem:T'}
For the convenience of the proof, we wish to distinguish between the torus $\T$ acting on $\XLV$ and the torus acting on $V\oplus W$, so we denote the latter one as $\T'$.
While this is a different setup from Theorem \ref{thm:intro_twist}, the assertion remains the same.
\end{remark}

\subsection{Big $J$-function}
\label{subsec:big_J}
Before proceeding to the proof, we introduce a specific point on $\cL_{V\oplus W,\TLd}$.
We take any $\C$-basis of $H^*(B)$ and write as $\{\phi_i\}_{i\in I}$.
Since $H^*_{\TLd}(V\oplus W)\cong H^*(B)[\mu]$, we can take a $\C$-basis of $H^*_{\TLd}(V\oplus W)[z]$ as follows:
\begin{equation}
\label{eqn:basis}
\{ \phi_i z^n \mu^a \colon i\in I, n\in\Z_{\geq0}, a\in\Z_{\geq0}^N \}.
\end{equation}
We write the coordinate on $H^*_{\TLd}(V\oplus W)[z]$ associated to this basis as $\tbtau = \{ \tau^i_{n,a} \}$.
We also obtain the coordinate system $\btau = \{ \tau_n^i = \tau^i_{n,0} \}$ on $H^*(V\oplus W)[z]$. 

\begin{definition}
We define $\J_{V\oplus W}^\mu(\tbtau,z)\in H^*_{\TLd}(B)[z,z^{-1}][\![\Eff(B)]\!][\![\tbtau]\!]$ as follows: 
\begin{align*}
z\J_{V\oplus W}^\mu(\tbtau,z) = z + \tbtau(z) + \sum_{ \substack{n\geq0,d\in\Eff(B) \\ (n,d)\neq(0,0),(1,0)} } \sum_{i\in I} \frac{Q^d}{n!} \corr{\frac{\phi_i}{z-\psi},\tbtau(\psi),\dots,\tbtau(\psi)}_{0,n+1,d}^{V\oplus W,\TLd}\phi^i,	
\end{align*}
where $\{\phi^i\}$ is a dual basis of $\{\phi_i\}$ with respect to the Poincar\'{e} pairing on $V\oplus W$, and 
\[
\tbtau(z) = \sum_{i\in I} \sum_{n\geq0} \sum_{a\in\Z_{\geq0}^N} \tau^i_{n,a} \phi_iz^n\mu^a.
\]
We write $\btau(z) = \sum_{i\in I} \sum_{n\geq0} \tau^i_n \phi_iz^n$ and identify $\tau^i_n$ with $\tau^i_{n,0}$.
\end{definition}

From the definition of the Lagrangian cone, it is clear that $-z\J_{V\oplus W}^\mu(\tbtau,-z)$ is a $H^*_{\TLd}(\pt)[\![\Eff(B)]\!][\![\tbtau]\!]$-valued point of $\cL_{V\oplus W,\TLd}$.
Note that from $\J_{V\oplus W}^\mu$ we can obtain any $H^*_\TLd(\pt)[\![\Eff(B)]\!][\![x]\!]$-valued point $I_{V\oplus W}^\mu$ in $\cL_{V\oplus W,\TLd}$ by substituting the non-negative part of $zI_{V\oplus W}^\mu-z$ with respect to $z$ into $\tbtau$.
This observation and the following lemma imply that it is enough to prove Theorem \ref{thm:twist_mirror} for $\J_{V\oplus W}^\mu(\btau,z)$.
\begin{lemma}
It holds that
\[
(\J_{V\oplus W}^\mu(\tbtau,z))_{\tw}\sphat = \Delta(t,\tbtau,z\partial_t,z\partial_{\btau},z) (\J_{V\oplus W}^\mu(\btau,z))_{\tw}\sphat
\]
where $\Delta(t,\tbtau,z\partial_t,z\partial_{\btau},z)$ is the differential operator:
\[
\Delta(t,\tbtau,z\partial_t,z\partial_{\btau},z) = \exp \left( z^{-1} \sum_{i\in I} \sum_{n\geq0} z\partial_{\tau^i_{n,0}} \sum_{a\in\Z_{\geq0}^N\setminus\{0\}} \tau^i_{n,a} \prod_{j=1}^N (z\partial_{t_j})^{a_j} \right).
\]
\end{lemma}

\begin{proof}
The right-hand side is equal to
\begin{align*}
&\Delta(t,\tbtau,z\partial_t,z\partial_{\btau},z) \sum_{\ell\in\LL} \frac{\tq^\ell e^{\sum_{i=1}^N \left(D_i(\ell)\cdot t_i + \frac{t_iu_i}{z}\right)}}{\prod_{i=1}^K \prod_{c=1}^{D_i(\ell)} \prod_{ \substack{\delta\colon\mathrm{Chern\ roots} \\ \mathrm{of\ } V_i} } (u_i + \delta + cz)} \\
\cdot&\ \frac{\J_{V\oplus W}^{u + D(\ell)z}(\btau,z)}{\prod_{i=K+1}^N \prod_{c=1}^{D_i(\ell)} \prod_{ \substack{\delta\colon\mathrm{Chern\ roots} \\ \mathrm{of\ } W_i} } (u_i + \delta + cz)} \\
=&\ \sum_{\ell\in\LL} \frac{\tq^\ell e^{\sum_{i=1}^N \left(D_i(\ell)\cdot t_i + \frac{t_iu_i}{z}\right)} \Delta(t,\tbtau,u+D(\ell)z,z\partial_{\btau},z) \J_{V\oplus W}^{u + D(\ell)z}(\btau,z)}{\prod_{i=1}^K \prod_{c=1}^{D_i(\ell)} \prod_{ \substack{\delta\colon\mathrm{Chern\ roots} \\ \mathrm{of\ } V_i} } (u_i + \delta + cz) \prod_{i=K+1}^N \prod_{c=1}^{D_i(\ell)} \prod_{ \substack{\delta\colon\mathrm{Chern\ roots} \\ \mathrm{of\ } W_i} } (u_i + \delta + cz)}.
\end{align*}
Since $\Delta(t,\tbtau,\mu,z\partial_{\btau},z)$ is the operator that shifts $\tau$ to $\btau$, we have
\[
\Delta(t,\tbtau,u+D(\ell)z,z\partial_{\btau},z) \J_{V\oplus W}^{u + D(\ell)z}(\btau,z) = \J_{V\oplus W}^{u + D(\ell)z}(\tbtau,z).
\]
These computations imply the desired formula.
\end{proof}

\subsection{Quantum Riemann-Roch operator}
In this subsection, we use the notations introduced in Section \ref{subsubsec:QRR}.
Let $W$ be a vector bundle over $\XLO$.
We introduce the following functions:
\begin{align*}
z^{-1}G(\lambda)	&= \frac{\lambda\log\lambda-\lambda}{z} + \frac{1}{2}\log\lambda + \sum_{m\geq2} \frac{B_m}{m(m-1)} \left(\frac{z}{\lambda}\right)^{m-1},	\\
G_W^\lambda	&= \sum_{ \substack{\delta\colon\mathrm{Chern\ roots} \\ \mathrm{of\ } W} } G(\lambda+\delta),	\\
H_W^\lambda	&= \rank(W) \cdot \left( \lambda\log\lambda - \lambda + \frac{z}{2}\log\lambda \right) + (\log\lambda) z\partial_{c_1(W)},	\\
S_W^\lambda	&= (\log\lambda) \left(c_1(W) - z\partial_{c_1(W)}\right)
\end{align*}
where $\partial_{c_1(W)}$ denotes the unique vector field on the $(\btau,t_1,\dots,t_K)$-space such that
\[
\partial_{c_1(W)} \left(\btau + \sum_{i=1}^K t_iu_i\right) = c_1(W).
\]
Since $H_W^\lambda, S_W^\lambda \in H^*(B\times\prod_{i=1}^K\PP^{s_i})[\lambda,\log\lambda,z,z\partial_{\btau},z\partial_t]$, the exponents $e^{H_W^\lambda/z}$ and $e^{S_W^\lambda/z}$ are ill-defined as operators on $\cH_{B\times\prod_{i=1}^K\PP^{s_i}}$.
However, it follows from the divisor equation that the operator $e^{S_W(\lambda)/z}$ restricted to Lagrangian cones is well-defined, which replaces $\cQ^\cD$ by $\cQ^\cD \lambda^{-\int_\cD c_1(W)}$. 
Hence the operator $\Delta_W^\lambda$ \eqref{eqn:modified_QRR} can be written as
\[
\Delta_W^\lambda = e^{S_W^\lambda/z} \Delta_{(W,\te^{-1})}(\lambda,z).
\]

\begin{lemma}
\label{lem:H}
Let $W$ be a vector bundle over $\XLO$, $D\in\LL^\vee$, $u=c_1(\cO(D))$ and $k\in\Z$.
\begin{itemize}
\item[$(1)$]
The operator 
\[
\exp\left( \frac{ H_W^{\lambda + z \partial_u} - H_{W(D)}^\lambda }{z} \right) 
\]
is well-defined and belongs to $\C[z^{-1},z\partial_t,z\partial_\tau][\![\lambda^{-1}]\!]$.
Moreover, this operator preserves (twisted) Lagrangian cones.

\item[$(2)$]
It holds that
\begin{equation}
\label{eqn:gamma}
e^{(H_W^{\lambda + u +kz} - H_{W(D)}^\lambda)/z} \cdot \frac{\Delta_W^{\lambda+u+kz}}{\Delta_{W(D)}^\lambda} = \prod_{ \substack{\delta\colon\mathrm{Chern\ roots} \\ \mathrm{of\ } W} } \frac{\prod_{c=-\infty}^k(\lambda+u+\delta+cz)}{\prod_{c=-\infty}^0(\lambda+u+\delta+cz)}.
\end{equation}
Here the right-hand side is intepreted as an element of $\lambda^k \cdot H^*(\XLO)[\lambda^{-1}](\!(z)\!)$.
\end{itemize}
\end{lemma}

\begin{proof}
A direct computation shows that $z^{-1} (H_W^{\lambda + z\partial_u} - H_{W(D)}^\lambda)$ is equal to
\[
\frac{\rank(W)}{z} \cdot \left( \left( \lambda + z\partial_u + \frac{z}{2} \right) \log\left(1+\frac{z\partial_u}{\lambda}\right) - z\partial_u \right) + \frac{z\partial_{c_1(W)}}{z} \log\left( 1 + \frac{z\partial_u}{\lambda} \right).
\]
This belongs to $\lambda^{-1} \cdot \C[z^{-1},z\partial_{\tau},z\partial_t][\![\lambda^{-1}]\!]$, and hence its exponent is a well-defined operator which preserves Lagrangian cones; see Lemma \cite[Lemma 2.7]{IK:quantum}.

We have the following identities:
\begin{gather*}
\frac{G(\lambda+kz) - G(\lambda)}{z} = \sum_{c=-\infty}^k \log(\lambda+cz) - \sum_{c=-\infty}^0 \log(\lambda+cz),	\\
G_{W\otimes L}^\lambda = G_W^{\lambda+c_1(L)},	\qquad
z\log\left( \Delta_{(W,\te^{-1})}^\lambda\right) = G_W^\lambda - H_W^\lambda - S_W^\lambda.	
\end{gather*}
Using these identities, the logarithm of the left-hand side of \eqref{eqn:gamma} is computed as follows:
\begin{align*}
&z^{-1} \left( H_W^{\lambda+u+kz} + S_W^{\lambda+u+kz} + z\log\left(\Delta_{(W,\te^{-1})}^{\lambda+u+kz}\right) - H_{W(D)}^\lambda - S_{W(D)}^\lambda - z\log\left(\Delta_{(W(D),\te^{-1})}^\lambda\right) \right)	\\
=&\ z^{-1} \left( G_W^{\lambda+u+kz} - G_{W(D)}^\lambda \right)	\\
=&\ \sum_{ \substack{\delta\colon\mathrm{Chern\ roots} \\ \mathrm{of\ } W} } \left( \sum_{c=-\infty}^k \log(\lambda+u+\delta+cz) - \sum_{c=-\infty}^0 \log(\lambda+u+\delta+cz) \right).
\end{align*}
By taking exponentials, we obtain \eqref{eqn:gamma}.
\end{proof}

\subsection{Proof of Theorem \ref{thm:twist_mirror}}
As discussed at the end of Section \ref{subsec:big_J}, we prove Theorem \ref{thm:twist_mirror} only for $\J_{V\oplus W}^\mu(\btau,z)$.
The proof is based on that of \cite[Theorem 1.1]{IK:quantum}.

Since $V_1,\dots,V_K$ are assumed to be generated by global sections, there exist vector bundles $\cQ_1,\dots,\cQ_K$ and non-negative integers $s_1,\dots,s_K$ which fit into the following short exact sequences
\begin{align*}
0\to V_i\to \cO^{\oplus s_i}\to\cQ_i\to0	
\end{align*}
for $1\leq i\leq K$.
Without loss of generality, we can assume that $s_i\geq2$ for any $i$.
We set $\vecO = (\cO^{\oplus s_1},\dots,\cO^{\oplus s_K})$.
By definition, we have $\XLO\cong B\times\prod_{i=1}^K\PP^{s_i-1}$.

From Lemma \ref{lem:vector}, we can see that the function
\begin{equation}
\label{eqn:long}
\left. \left( \prod_{i=1}^K \Delta_{\cQ_i}^{\mu_i} \right) \left( \prod_{i=K+1}^N \left( \Delta_{W_i}^{\mu_i} \right)^{-1} \right) \J_{V\oplus W}^\mu(\btau,z) \right|_{z\to-z} \cdot (-z)
\end{equation}
is a $\C[\![\Eff(B)]\!][\![\btau,\mu^{-1}]\!]$-valued point of $\cL_B$.
We apply the following lemma to this function.

\begin{lemma}[{\cite[Lemma 2.8]{IK:quantum}}]
Let $X$ be a smooth projective variety, $\{\phi_i\}_{0\leq i\leq s}$ be a basis of $H^*(X)$ and $\tau=\sum_{i=0}^s \tau^i\phi_i\in H^*(X)$.
Let $-zI(\tau,x,-z)$ be a $\C[\![\Eff(X)]\!][\![\tau,x]\!]$-valued point of $\cL_X$ such that
\[
I(\tau,0,z)|_{Q=0} = 1 + \frac{\tau}{z} + O(z^{-2}).
\]
Then there exists a differential operator $F\in\sum_{i=0}^s\C[\![\Eff(X)]\!][\![\tau,x]\!]z\partial_{\tau^i}$ satisfying $e^{F/z}J_X(\tau,z)=I(\tau,x,z)$.
\end{lemma}
We can write \eqref{eqn:long} as $\exp(F(\btau,\mu,z\partial_\tau)/z) J_B(\tau,z)$ for some $F(\btau,\mu,z\partial_\tau)\in\sum_{i\in I}\C[\![\Eff(B)]\!][\![\btau,\mu^{-1}]\!] z\partial_{\tau^i_0}$.
We introduce the following function:
\begin{multline*}
I^\mu := \left( \prod_{i=1}^K \left( \Delta_{\cQ_i(D_i)}^{\mu_i} \right)^{-1} \cdot e^{(H_{\cQ_i}^{\mu_i+z\partial_{t_i}}-H_{\cQ_i(D_i)}^{\mu_i})/z} \right) \cdot \left( \prod_{i=K+1}^N \Delta_{W_i(D_i)}^{\mu_i+u_i^0} \cdot e^{(H_{W_i(D_i)}^{\mu_i+u_i^0}-H_{W_i}^{\mu_i+z\partial_{t_i}})/z} \right) \\
\cdot e^{F(\btau,\mu+z\partial_t,z\partial_\tau)/z} J_{\XLO}(t,\tau,z)
\end{multline*}
where $u_i^0:=-\lambda_i+\sum_{i=1}^K D_j^\vee(D_i)\lambda_j$ is the image of $u_i$ under the projection $H^2_\TL(\XLO)\to H^2_\TL(\pt)$, and $J_{\XLO}(t,\tau,z)$ is the $J$-function at the parameter $\tau+\sum_{i=1}^Nt_iu_i$:
\[
J_{\XLO}(t,\tau,z) = e^{\sum_{i=1}^Nt_iu_i/z}\sum_{\ell\in\LL} \frac{q^\ell e^{\sum_{i=1}^N D_i(\ell) \cdot t_i}}{\prod_{i=1}^K\prod_{c=1}^{D_i(\ell)}(u_i+cz)^{s_i}} J_B(\tau,z).
\]
From the mirror theorem for split toric bundles \cite{Brown}, it can be seen that the function $-zJ_{\XLO}(t,\tau,-z)$ is a point on $\cL_{\XLO}$.
For any $f\in H^*(\XLO)[z\partial_t][\![\Eff(B)]\!][\![\btau,\mu^{-1}]\!]$, we have
\[
f(z\partial_t) J_{\XLO}(t,\tau,z) = \sum_{\ell\in\LL} \frac{q^\ell e^{\sum_{i=1}^N \left(D_i(\ell)\cdot t_i + \frac{t_iu_i}{z}\right)}}{\prod_{i=1}^K\prod_{c=1}^{D_i(\ell)}(u_i+cz)^{s_i}} f(u+D(\ell)z) J_B.
\]
Therefore, $I^\mu$ can be computed as 
\begin{align*}
I^\mu 
=&\ \sum_{\ell\in\LL} \frac{q^\ell e^{\sum_{i=1}^N \left(D_i(\ell)\cdot t_i + \frac{t_iu_i}{z}\right)}}{\prod_{i=1}^K\prod_{c=1}^{D_i(\ell)}(u_i+cz)^{s_i}} \left( \prod_{i=1}^K \left( \Delta_{\cQ_i(D_i)}^{\mu_i} \right)^{-1} \cdot e^{(H_{\cQ_i}^{\mu_i+u_i+D_i(\ell)z}-H_{\cQ_i(D_i)}^{\mu_i})/z} \right) 
\\
&\cdot \left( \prod_{i=K+1}^N \Delta_{W_i(D_i)}^{\mu_i+u_i^0} \cdot e^{(H_{W_i(D_i)}^{\mu_i+u_i^0}-H_{W_i}^{\mu_i+u_i+D_i(\ell)z})/z} \right) \cdot e^{F(\btau,\mu+u+D(\ell)z,z\partial_\tau)}J_B(\tau,z) 
\\
=&\ \sum_{\ell\in\LL} \frac{q^\ell e^{\sum_{i=1}^N \left(D_i(\ell)\cdot t_i + \frac{t_iu_i}{z}\right)}}{\prod_{i=1}^K\prod_{c=1}^{D_i(\ell)}(u_i+cz)^{s_i}} \left( \prod_{i=1}^K  \frac{\Delta_{\cQ_i}^{\mu_i+u_i+D_i(\ell)z}}{\Delta_{\cQ_i(D_i)}^{\mu_i}}  \cdot e^{(H_{\cQ_i}^{\mu_i+u_i+D_i(\ell)z}-H_{\cQ_i(D_i)}^{\mu_i})/z} \right) 
\\
&\cdot \left( \prod_{i=K+1}^N \frac{\Delta_{W_i}^{\mu_i+u_i+D_i(\ell)z}}{\Delta_{W_i(D_i)}^{\mu_i+u_i^0}} \cdot e^{(H_{W_i}^{\mu_i+u_i+D_i(\ell)z}-H_{W_i(D_i)}^{\mu_i+u_i^0})/z} \right)^{-1} \cdot \J_{V\oplus W}^{\mu+u+D(\ell)z}
\\
=&\ \sum_{\ell\in\LL} q^\ell e^{\sum_{i=1}^N \left(D_i(\ell)\cdot t_i + \frac{t_iu_i}{z}\right)} \left( \prod_{i=1}^K\prod_{c=1}^{D_i(\ell)} \frac{\prod_{ \substack{\delta\colon\mathrm{Chern\ roots} \\ \mathrm{of\ } \cQ_i} }(\mu_i+u_i+\delta+cz)}{(u_i+cz)^{s_i}} \right)	\\
&\cdot \frac{1}{\prod_{i=K+1}^N\prod_{c=1}^{D_i(\ell)}\prod_{ \substack{\delta\colon\mathrm{Chern\ roots} \\ \mathrm{of\ } W_i} }(\mu_i+u_i+\delta+cz)} \cdot \J_{V\oplus W}^{\mu+u+D(\ell)z}.
\end{align*}
Here Lemma \ref{lem:H} (2) is used for the last equality.
Since $\exp((H_{\cQ_i}^{\mu_i+z\partial_{t_i}}-H_{\cQ_i(D_i)}^{\mu_i})/z)$ and $\exp((H_{W_i(D_i)}^{u_i^0}-H_{W_i}^{u_i^0+z\partial_{t_i}})/z)$ preserve Lagrangian cones (Lemma \ref{lem:H} (1)), $-z \cdot I^\mu|_{z\to-z}$ lies in the Givental cone for $(\vecE,\vecbc)$-twisted Gromov-Witten theory of $\XLO$ where $(\vecE,\vecbc)$ denotes the following collection:
\[
(E_i,\bc^i) =
\begin{cases}
(Q_i(D_i),e_{\mu_i})			&\text{for } 1\leq i\leq K,	\\
(W_i(D_i),e_{\mu_i+u_i^0}^{-1})	&\text{for } K+1\leq i\leq N.
\end{cases}
\]
Since the non-equivariant limit $\lim_{\mu\to0}I^\mu$ with respect to $\TLd$ exists, we can apply Theorem \ref{thm:subvariety} to $-zI^\mu(-z)$ and obtain a point on $\cL_{\XLV,(\cW,e_\lambda^{-1})}$, which coincides with $-z(I_{V\oplus W}^\mu)_{\tw}\sphat(-z)$.
This proves Theorem \ref{thm:twist_mirror}.

\begin{remark}
For convenience, we list the rings to which the functions introduced in this section belong.
\begin{align*}
\J_{V\oplus W}^\mu(\btau,z)
&\in	H^*(B)[\mu,z,z^{-1}][\![\Eff(B)]\!][\![\btau]\!],		\\
\frac{ \prod_{i=1}^K \Delta_{\cQ_i}^{\mu_i} }{ \prod_{i=K+1}^N (\Delta_{W_i}^{\mu_i})^{-1} }  \J_{V\oplus W}^\mu(\btau,z)
&\in	H^*(B)[\mu^{-1},z,z^{-1}][\![\Eff(B)]\!][\![\btau]\!],	\\
I^\mu(t,\btau,z)
&\in	H^*(\XLO)[\mu](\lambda)(\!(z)\!)[\![\Eff(B)\oplus\LL_\eff]\!][\![t,\btau]\!]	\\
(&\subset	H^*(\XLO)(\mu,\lambda)(\!(z)\!)[\![\Eff(B)\oplus\LL_\eff]\!][\![t,\btau]\!],)	\\
(\J_{V\oplus W}^\mu(\btau,z))_{\tw}\sphat
&\in	H^*(\XLV)(\lambda)(\!(z)\!)[\![\Eff(B)\oplus\LL_\eff]\!][\![t,\btau]\!].
\end{align*}
\end{remark}

\section{Mirror theorem for toric bundles}
\label{sec:main_result}
In this section, we will prove the mirror theorem (Theorem \ref{thm:mirror_thm}) for (possibly non-split) toric bundles.
Throughout this section, we fix the following data:
\begin{itemize}
\item a smooth toric data $\sfL=(\LL^\vee,D\colon\Z^N\to\LL^\vee,\omega)$;
\item a smooth projective variety $B$;
\item vector bundles $V_1,\dots,V_N$ over $B$ whose duals are generated by global sections.
\end{itemize}
We let $-\lambda_i\in H^2_\TL(\pt)$ be the equivariant parameter corresponding to the $i$-th projection $\TL\to\C^\times$.

\subsection{Main theorem}
We consider the diagonal $\TL$-action on $V$.
Let $I_V^\lambda(x,z)$ be  a function lying in
\[
H^*_\TL(V)[z,z^{-1}][\![\Eff(B)]\!][\![x]\!] = H^*(B)[\mu,z,z^{-1}][\![\Eff(B)]\!][\![x]\!]
\]
such that $-zI_V^\lambda(x,-z)$ is a $H^*_\TL(\pt)[\![\Eff(B)]\!][\![x]\!]$-valued point of $\cL_{V,\TL}$.
We define
\[
(I_V^\lambda)\sphat\ (t,x,y,z) = e^{\sum_{i=1}^Nt_iu_i/z} \sum_{\ell\in\LL} \frac{\tq^\ell e^{\sum_{i=1}^ND_i(\ell)\cdot t_i}}{\prod_{i=1}^N\prod_{c=1}^{D_i(\ell)}\prod_{ \substack{\delta\colon\mathrm{Chern\ roots} \\ \mathrm{of\ } V_i} } (u_i + \delta + cz) } \cdot I_V^{u+D(\ell)z}
\]
where $I_V^{u+D(\ell)z} := (I_V^\lambda)|_{\lambda_i\to u_i+D_i(\ell)z}$ and $\tq$ denotes a formal variable for $\C[\![\LL_\eff]\!]$.	
Note that $(I_V^\lambda)\sphat\ $ can be interpreted as an element of 
\[
H^*_\TL(\XLV)_{\loc}(\!(z^{-1})\!)[\![\extEff(\XLV)]\!][\![t,x,y]\!]
\]
via the identification \eqref{eqn:extEff}.
We can now state our main result.

\begin{theorem}
\label{thm:mirror_thm} 
The function $-z\cdot (I_V^\lambda)\sphat\ (-z)$ is a $\Frac(H^*_\TL(\pt))[\![\extEff(\XLV)]\!][\![x,y]\!]$-valued point on $\cL_{\XLV,\TL}$.
\end{theorem}

Thanks to Theorem \ref{thm:characterization}, it is enough to confirm that $(I_V^\lambda)\sphat\ $ satisfies the three conditions appearing there.
In the following three subsections, we will sequentially verify the conditions \textbf{(C3)}, \textbf{(C1)} and \textbf{(C2)}.

\subsection{Restrictions of $(I_V^\lambda)\sphat$}
In this subsection, we describe the restriction $\iota_\alpha^*(I_V^\lambda)$ for $\alpha\in F_\sfL$.
Recall from Section \ref{subsec:coh} and \ref{subsec:fixed_locus} that we have the isomorphisms
\begin{align*}
H^*_\T(\XLV)		&\cong H^*_\TL(B)[u_1,\dots,u_N]/(\cI+\cJ),	\\
H^*_\T(\XLV_\alpha)	&\cong H^*_\TL(B)[\{u_i\}_{i\in\alpha}]/\langle e_\TL(V_i\otimes L_i) \colon i\in\alpha \rangle,
\end{align*}
and the map 
\[
\iota_\alpha^*\colon H^*_\TL(B)[u_1,\dots,u_N]/(\cI+\cJ) \to H^*_\TL(B)[\{u_i\}_{i\in\alpha}]/\langle e_\TL(V_i\otimes L_i) \colon i\in\alpha \rangle
\]
is a $H^*_\TL(B)$-module morphism with
\[
\iota_\alpha^*u_i = 
\begin{cases}
u_i												&\text{if } i\in\alpha,	\\
-\lambda_i+\sum_{j\in\alpha}D_{\alpha,j}^\vee(D_i)\cdot(u_j+\lambda_j)	&\text{if } i\notin\alpha
\end{cases}
\]
where $\{D_{\alpha,i}^\vee\}_{i\in\alpha}\subset\LL$ is a dual basis of $\{D_i\}_{i\in\alpha}$.
Hence the function $\iota_\alpha^*(I_V^\lambda)\sphat\ (t,x,y,z)$ can be written as
\begin{equation}
\label{eqn:restriction}
\iota_\alpha^*(I_V^\lambda)\sphat\ (z) = e^{\sum_{i=1}^Nt_i\cdot\iota_\alpha^*u_i/z} \sum_{\ell\in\LLeff} \frac{\tq^\ell e^{\sum_{i=1}^ND_i(\ell)\cdot t_i}}{\prod_{i\in\alpha}\prod_{c=1}^{D_i(\ell)} R_{V_i}(u_i+cz)}  \cdot\frac{I_V^{\iota_\alpha^*u+D(\ell)z}}{\prod_{i\notin\alpha}\prod_{c=1}^{D_i(\ell)} R_{V_i}(\iota_\alpha^*u_i+cz)}
\end{equation}
where, for a vector bundle $V\to B$, the form $R_V(w)\in H^*_\TL(\XLV_\alpha)[w]$ is defined as
\[
R_V(w) := \prod_{ \substack{\delta\colon\mathrm{Chern\ roots} \\ \mathrm{of\ } V} } (w + \delta)= w^{\rank(V)} + c_1(V)\cdot w^{\rank(V)-1}+\cdots+c_{\rank(V)}(V).
\]
This coincides with $(I_V^\lambda)_{\tw}\sphat$, which is introduced in Section \ref{subsec:twist_FT}, associated to the following data:
\begin{itemize}
\item a smooth toric data $(\LL^\vee,D_\alpha\colon\Z^\alpha\to\LL^\vee,\omega)$;
\item vector bundles $\{V_i\}_{i\in\alpha}$ and $\{V_i\}_{i\notin\alpha}$;
\item $\{D_i\}_{i\notin\alpha}\subset\LL^\vee$.
\end{itemize}
Considering that $N_\alpha=\bigoplus_{i\notin\alpha}(V_i\otimes L_i)$, Theorem \ref{thm:twist_mirror} shows that $-z\iota_\alpha^*(I_V^\lambda)\sphat\ (-z)$ is a $\Frac(H^*_\TL(\pt))[\![\extEff(\XLV)]\!][\![t,x,y]\!]$-valued point of $\cL_{\XLV_\alpha,(N_\alpha,e_\TL^{-1})}$, which implies that $\f$ satisfies \textbf{(C3)}.

\subsection{Poles of $\iota_\alpha^*(I_V^\lambda)\sphat$}
In order to verify \textbf{(C1)}, we examine the set of poles of $\iota_\alpha^*(I_V^\lambda)\sphat\ (z)$.
Since the function $I_V^{\iota_\alpha^*u+D(\ell)z}$ belongs to
\[
H^*_\TL(\XLV_\alpha)[z,z^{-1}][\![\Eff(B)]\!][\![x]\!],
\]
all poles except for those at $z=0,\infty$ arise from the denominators in \eqref{eqn:restriction}:
\[
\{ \iota_\alpha^*u_i+\delta+cz \colon 1\leq i\leq N, c\geq1, \delta\colon\text{Chern roots of\ } V_i \}.
\]
Note that the image of $\iota_\alpha^*u_i+\delta$ under the projection $H^2_\TL(\XLV)\to H^2_\TL(\pt)$ equals 
\[
-\lambda_i + \sum_{j\in\alpha}D_{\alpha,j}^\vee(D_i)\cdot\lambda_j,
\]
which coincides with $\lambda_\ab$ if there exists $\beta\in F_\sfL$ such that $i_\ab=i$.
The condition \textbf{(C1)} for $\f$ follows from the fact that, if $D_i(\ell)>0$ for some $\ell\in\LLeff$ and $i\notin\alpha$, there exists $\beta\in\adj(\alpha)$ such that $i_\ab=i$.

\subsection{Recursion formula of $\iota_\alpha^*(I_V^\lambda)\sphat$}
In this subsection, we establish the recursion formula \eqref{eqn:recursion} for principal parts of $\iota_\alpha^*(I_V^\lambda)\sphat\ $, which implies that $-z\iota_\alpha^*(I_V^\lambda)\sphat\ (-z)$ satisfies \textbf{(C2)}.
Note that since we already know that $-z\iota_\alpha^*(I_V^\lambda)\sphat\ (-z)$ satisfies \textbf{(C1)} and \textbf{(C3)}, the following proposition finish the proof of Theorem \ref{thm:mirror_thm}.

\begin{proposition}
For any $\alpha\in F_\sfL$, $\beta\in\adj(\alpha)$ and $k\in\N$, it holds that
\begin{multline}
\label{eqn:recursion}
\Prin_{z=-\frac{\lambda_\ab}{k}}\iota_\alpha^*(I_V^\lambda)\sphat\ (z) \\
= {\paba}_* \left[ q^{k\cdot\dab} \cdot \frac{C_\ab(k)}{kz+c_1^\TL(L_\ab)} \cdot \pabb^*\iota_\beta^* (I_V^\lambda)\sphat\ \left(z=-\frac{c_1^\TL(L_\ab)}{k}\right) \right]
\end{multline}
where $C_\ab(k)$ is the element of $H^*_\TL(\XLV_\alpha)_{\loc}$ introduced in Theorem \ref{thm:characterization}.
\end{proposition}

\begin{proof}
Since the function $R_{V_{i_\ab}}(\iota_\alpha^*u_{i_\ab}+kz) \cdot \iota_\alpha^*(I_V^\lambda)\sphat\ (z)$ is regular at $z=-\lambda_\ab/k$, we have
\begin{equation}
\label{eqn:Prin}
\Prin_{z=-\frac{\lambda_\ab}{k}}\iota_\alpha^*(I_V^\lambda)\sphat\ (z) = \frac{q^{k\cdot\dab}}{R_{V_{i_\ab}}(\iota_\alpha^*u_{i_\ab}+kz)} \cdot \sum_{\ell\in\LLeff} \tq^\ell e^{\sum_{i=1}^ND_i(\ell)\cdot t_i} \cdot \barA_{\ab,k;\ell}(\iota_\alpha^*u_{i_\ab}+kz)
\end{equation}
where $\barA_{\ab,k;\ell}(w)\in H^*_\TL(\XLV_\alpha)_{\loc}[\![\Eff(B)]\!][\![t,x,y]\!][w]$ is the unique element satisfying
\begin{itemize}
\item the function lying in $H^*_\TL(\XLV_\alpha)_{\loc}[w][\![\Eff(B)]\!][\![t,x,y]\!]$
\begin{multline*}
A_{\ab,k;\ell}(w) := \exp\left( \frac{k\sum_{i=1}^N t_i\cdot\iota_\alpha^*u_i}{w-\iota_\alpha^*u_{i_\ab}} \right) \cdot \frac{ e^{\sum_{i=1}^ND_i(k\cdot\dab)\cdot t_i} }{\prod_{i=1}^N\prod_{c=1}^{D_i(\ell+k\cdot\dab)} R_{V_i}(\iota_\alpha^*u_i-\frac{c}{k}\iota_\alpha^*u_{i_\ab}+\frac{c}{k}w)}	\\
\cdot R_{V_{i_\ab}}(w) \cdot I_V^{\iota_\alpha^*u-\frac{D(\ell+k\cdot\dab)}{k}\iota_\alpha^*u_{i_\ab}+\frac{D(\ell+k\cdot\dab)}{k}w}
\end{multline*}
where 
\[
I_V^{\iota_\alpha^*u-\frac{D(\ell+k\cdot\dab)}{k}\iota_\alpha^*u_{i_\ab}+\frac{D(\ell+k\cdot\dab)}{k}w} 
\]
denotes the function $I_V^\lambda$ with $\lambda_i$ replaced with $\iota_\alpha^*u_i-\frac{D_i(\ell+k\cdot\dab)}{k}\iota_\alpha^*u_{i_\ab}+\frac{D_i(\ell+k\cdot\dab)}{k}w$, coincides with the function $\barA_{\ab,k;\ell}(w)$ in the quotient ring $H^*_\TL(\XLV_\alpha)_{\loc}[w][\![\Eff(B)]\!][\![t,x,y]\!]/( R_{V_{i_\ab}}(w) )$;
\item as a $(Q,t,x,y)$-series, all coefficients of $\barA_{\ab,k;\ell}(w)$ are polynomials in $w$ of degree less than $\rank(V_{i_\ab})$ with coefficients in $H^*_\TL(\XLV_\alpha)_{\loc}$.
\end{itemize}
By definition, there exist functions $\barA_{\ab,k;\ell}^n\in H^*_\TL(\XLV_\alpha)_{\loc}[\![\Eff(B)]\!][\![t,x,y]\!]$ ($0\leq n\leq\rank(V_{i_\ab})-1$) such that
\[
\barA_{\ab,k;\ell}(w) = \sum_{n=0}^{\rank(V_{i_\ab})-1} \barA_{\ab,k;\ell}^n\cdot w^n.
\]

We now proceed with the computation of the right-hand side.
From Corollary \ref{cor:coh}, it follows that 
\begin{align*}
H^*_\TL(\XLV_\acb) &\cong H^*_\TL(B)[\{u_i\}_{i\in\acb}]/\langle e(V_i\otimes L_i) \colon i\in\acb \rangle,	\\
H^*_\TL(\XLV_\beta) &\cong H^*_\TL(B)[\{u_i\}_{i\in\beta}]/\langle e(V_i\otimes L_i) \colon i\in\beta \rangle
\end{align*}
and the map $\pabb^*\colon H^*_\TL(\XLV_\beta)\to H^*_\TL(\XLV_\acb)$ is the $H^*_\TL(B)$-module morphism sending $u_i$ to $u_i$ for $i\in\beta$.
By direct calculation, we can see that
\begin{multline*}
\pabb^*\iota_\beta^* (I_V^\lambda)\sphat\ \left(z=-\frac{c_1^\TL(L_\ab)}{k}\right)  \\
= \exp\left( -\frac{k\sum_{i=1}^Nt_i\cdot\pabb^*\iota_\beta^*u_i}{c_1^\T(L_\ab)} \right) 
\cdot \sum_{\ell\in\LLeff} \frac{\tq^\ell e^{\sum_{i=1}^ND_i(\ell)\cdot t_i}I_V^{\iota_\beta^*u-\frac{D(\ell)}{k}c_1^\TL(L_\ab)}}{\prod_{i=1}^N\prod_{c=1}^{D_i(\ell)} R_{V_i}(\pabb^*\iota_\beta^*u_i-\frac{c}{k}c_1^\TL(L_\ab))}.	
\end{multline*}
From Lemma \ref{lem:dab_comparison} $(3)$, we have
\begin{align*}
\exp\left( -\frac{k\sum_{i=1}^Nt_i\cdot\pabb^*\iota_\beta^*u_i}{c_1^\T(L_\ab)} \right)		&= \exp\left( -k\sum_{i=1}^Nt_i\cdot\left(\frac{\paba^*\iota_\alpha^*u_i}{c_1^\T(L_\ab)}-D_i(\dab)\right) \right),	\\
\prod_{c=1}^{D_i(\ell)}R_{V_i}\left(\pabb^*\iota_\beta^*u_i - \frac{c}{k}c_1^\T(L_\ab)\right)	&= \prod_{c=k\cdot D_i(\dab)+1}^{k\cdot D_i(\ell+\dab)} R_{V_i}\left(\paba^*\iota_\alpha^*u_i - \frac{c}{k}c_1^\T(L_\ab)\right).
\end{align*}
Hence we have
\begin{align*}
&\ q^{k\cdot\dab} \cdot \frac{C_\ab(k)}{kz+c_1^\TL(L_\ab)} \cdot \pabb^*\iota_\beta^* (I_V^\lambda)\sphat\ \left(z=-\frac{c_1^\TL(L_\ab)}{k}\right)	\\
=&\ \frac{q^{k\cdot\dab}}{kz+c_1^\TL(L_\ab)} \cdot \exp\left( -k\sum_{i=1}^Nt_i\cdot\left(\frac{\paba^*\iota_\alpha^*u_i}{c_1^\T(L_\ab)}-D_i(\dab)\right) \right)	\\
\cdot&\ R_{V_{i_\ab}}\left( \paba^*\iota_\alpha^*u_{i_\ab} - c_1^\TL(L_\ab) \right) \cdot \sum_{\ell\in\LLeff} \frac{\tq^\ell e^{\sum_{i=1}^ND_i(\ell)\cdot t_i}I_V^{\iota_\beta^*u+\frac{D(\ell)}{k}c_1^\TL(L_\ab)}}{\prod_{i=1}^N\prod_{c=1}^{D_i(\ell+\dab)} R_{V_i}(\paba^*\iota_\alpha^*u_i-\frac{c}{k}c_1^\TL(L_\ab))}	\\
=&\ \frac{q^{k\cdot\dab}}{kz+\paba^*\iota_\alpha^*u_{i_\ab}-u_{i_\ab}} \cdot \sum_{\ell\in\LLeff} \tq^\ell e^{\sum_{i=1}^ND_i(\ell)\cdot t_i}\cdot\paba^*\barA_{\ab,k;\ell}\left( u_{i_\ab} \right).
\end{align*}
For the last equality, we use Lemma \ref{lem:dab_comparison} $(2)$ and the equality
\[
\paba^*A_{\ab,k:\ell}\left( u_{i_\ab} \right) = \paba^*\barA_{\ab,k:\ell}\left( u_{i_\ab} \right)
\]
in $H^*_\T(\XLV_\acb)_{\loc}[\![\Eff(B)]\!][\![t,x,y]\!]$.

Since $\XLV_\acb$ is isomorphic to the projectivization of the vector bundle $V_{i_\ab}\to\XLV_\alpha$, it holds that 
\[
{\paba}_*\left[u_{i_\ab}^n\right] = s_{n-\rank(V_{i_\ab})+1}(V_{i_\ab})
\] 
where $s_i$ denotes the $i$-th Segre class \cite{Fulton}. 
Here we define $s_i(V)=0$ if $i<0$.
Note that the Segre classes satisfy the following formula
\[
\sum_{i\geq0}x^{-i-\rank(V)}s_i(V) = R_V(x)^{-1}
\]
for any vector bundle $V$.
Using these formulas and the projection formula, it holds that
\begin{align*}
&\ {\paba}_*\left[ \frac{\paba^*\barA_{\ab,k;\ell}( u_{i_\ab} )}{kz+\paba^*\iota_\alpha^*u_{i_\ab}-u_{i_\ab}} \right] \\
=&\ \sum_{m\geq0} \left(\iota_\alpha^*u_{i_\ab}+kz\right)^{-m-1} \cdot {\paba}_*\left[ u_{i_\ab}^m \cdot \sum_{n=0}^{\rank(V_{i_\ab})-1}\paba^*\barA_{\ab,k;\ell}^n \cdot u_{i_\ab}^n \right]	\\
=&\ \sum_{n=0}^{\rank(V_{i_\ab})-1} \barA_{\ab,k;\ell}^n \cdot \sum_{m\geq0} \left(\iota_\alpha^*u_{i_\ab}+kz\right)^{-m-1} \cdot s_{m+n-\rank(V_{i_\ab})+1}(V_{i_\ab})	\\
=&\ \sum_{n=0}^{\rank(V_{i_\ab})-1} \barA_{\ab,k;\ell}^n \cdot \left(\iota_\alpha^*u_{i_\ab}+kz\right)^n \cdot R_{V_{i_\ab}}\left(\iota_\alpha^*u_{i_\ab}+kz\right)^{-1}	\\
=&\ \frac{\barA_{\ab,k;\ell}(\iota_\alpha^*u_{i_\ab}+kz)}{R_{V_{i_\ab}}(\iota_\alpha^*u_{i_\ab}+kz)}.
\end{align*}
From the calculations above, the right-hand side of \eqref{eqn:recursion} can be computed as follows.:
\begin{align*}
&\ {\paba}_*\left[q^{k\cdot\dab} \cdot \frac{C_\ab(k)}{kz+c_1^\TL(L_\ab)} \cdot \pabb^*\iota_\beta^* (I_V^\lambda)\sphat\ \left(z=-\frac{c_1^\TL(L_\ab)}{k}\right)\right]	\\
=&\ {\paba}_*\left[ \frac{q^{k\cdot\dab}}{kz+\paba^*\iota_\alpha^*u_{i_\ab}-u_{i_\ab}} \cdot \sum_{\ell\in\LLeff} \tq^\ell e^{\sum_{i=1}^ND_i(\ell)\cdot t_i}\cdot\paba^*\barA_{\ab,k;\ell}\left( u_{i_\ab} \right) \right]	\\
=&\ q^{k\cdot\dab} \cdot \sum_{\ell\in\LLeff} \tq^\ell e^{\sum_{i=1}^ND_i(\ell)\cdot t_i} \cdot {\paba}_*\left[ \frac{\paba^*\barA_{\ab,k;\ell}( u_{i_\ab} )}{kz+\paba^*\iota_\alpha^*u_{i_\ab}-u_{i_\ab}} \right]	\\
=&\ \frac{q^{k\cdot\dab}}{R_{V_{i_\ab}}(\iota_\alpha^*u_{i_\ab}+kz)} \cdot \sum_{\ell\in\LLeff} \tq^\ell e^{\sum_{i=1}^ND_i(\ell)\cdot t_i} \cdot \barA_{\ab,k;\ell}(\iota_\alpha^*u_{i_\ab}+kz),
\end{align*}
which coincides with \eqref{eqn:Prin}.
\end{proof}


\appendix

\section{Fourier transformation}
\label{app:Fourier}
We introduce a Fourier transform of Givental cones, which we learned from \cite{IS}, and check that the function $(I_V^\lambda)\sphat\ $ coincides with a Fourier transform of $I_V^\lambda$.
See Remark \ref{rem:Fourier} for the background.
In this appendix, we interprete $(I_V^\lambda)\sphat\ $ as a $\T/\K$-equivariant $I$-function for $\XLV$; see Remark \ref{rem:T/K}.

Let $\T=(\C^\times)^N$ be an algebraic torus, and let $X$ be a smooth semi-projective variety with a $\T$-action whose fixed point set is projective. 
For such a $\T$-variety, we can consider the \emph{extended shift operator} \cite{Iritani:blowups} $\hbS^\beta$ and $\hcS^\beta$ for each $\beta\in H_2^\T(X)$. 
Via the fundamental solution $M_X(\tau)$ for the $\T$-equivariant theory, the operators are related in the following way \cite[Proposition 2.7]{Iritani:blowups}:
\[
\hcS^\beta\circ M_X(\tau) = M_X(\tau)\circ\hbS^\beta.
\]
The extended shift operator $\hbS^\beta$ satisfies that
\[
\hbS^\beta(f(\lambda,z)\alpha) = f(\lambda-\ovbeta z,z)\hbS^\beta\alpha
\]
for any $f(\lambda,z)\in H^*_\T(\pt)[z]$ and any $\alpha\in H^*_\T(X)$, and is defined by using Gromov-Witten invariants of the Seidel space $E_{-\ovbeta}$, which is defined as the quotient
\[
E_{-\ovbeta} = X\times(\C^2\setminus0)/\C^\times
\]
where $\C^\times$ acts on $X\times(\C^2\setminus0)$ by the formula $t \cdot (x,(v_1,v_2)) = (t^{-\ovbeta}x,(tv_1,tv_2))$.
We omit the precise definition here; see \cite{Iritani:shift,Iritani:blowups} for details.

Let $F$ be a connected component of the $\T$-fixed point set $X^\T$.
Let $N_F$ be the normal bundle to $F$ in $X$, and consider its $\T$-weight decomposition $N_F = \bigoplus_{\alpha\in\Hom(\T,\C^\times)}N_{F,\alpha}$.
Let $\{ \rho_{F,\alpha,j} \}_{j=1}^{\rank N_{F,\alpha}}$ be the Chern roots of $N_{F,\alpha}$.
For $\beta\in H_2^\T(X)$, the shift operator $\hcS^\beta$ on the Givental space $\cH_{X,\T}$ satisfies that
\[
\left.\left(\hcS^\beta\f\right)\right|_F = \frac{ Q^{\beta+\sigma_F(-\ovbeta)}e^{-z\ovbeta\partial_\lambda}(\f|_F) }{ \prod_{\alpha\in\Hom(\T,\C^\times)}\prod_{j=1}^{\rank N_{F,\alpha}}\prod_{c=1}^{-\alpha\cdot\ovbeta}(\rho_{F,\alpha,j}+\alpha+cz) }
\]
for any point $\f$ on $\cH_{X,\T}$ and any connected component $F$ of $X^\T$.
Here $\sigma_F(-\ovbeta)\in H_2^\T(X,\Z)$ denotes the image of the section class of $E_{-\ovbeta}\to\PP^1$ associated with $F\subset X$ under the natural map $H_2(E_{-\ovbeta},\Z)\to H_2^\T(X,\Z)$, and $e^{-z\ovbeta\partial_\lambda}$ denotes the operator sending $f(\lambda,z)$ to $f(\lambda-(\lambda\cdot\ovbeta) z,z)$.
One can see that this formula uniquely determines $\hcS$ by using the localization formula.

Now, we assume that there is a smooth GIT quotient $Y = X/\!/\K$ for a subtorus $\K$ of $\T$, that is, $X^s=X^{ss}$ and $\K$ acts freely on $X^s$.
We denote the equivariant Kirwan map by $\kappa_Y\colon H^*_\T(X)\to H^*_{\T/\K}(Y)$.
We define the \emph{Fourier transform} of $\f\in\cH_{X,\T}^{\pol}$ to be
\[
\hat{\f} = \sum_{[\beta]\in H_2^\K(X)/H_2(X)} \kappa_Y\left( \hcS^{-\beta} e^{\sum_{i=1}^Nt_i\lambda_i/z}\f \right) \hS^\beta
\]
where $\hS^\beta$ denotes a formal variable associated to $\beta\in H_2^\K(X,\Z)$.
It is easy to see that, for any $\lambda\in H^2_\T(\pt,\Z)$ and $\beta\in H_2^\K(X)$, 
\[
(\lambda\f)\sphat\ = \left(\kappa_Y(\lambda)+z\cdot\lambda\hS\frac{\partial}{\partial\hS}\right)\hat{\f}, \quad \left(\hcS^\beta\f\right)\sphat = \hS^\beta \hat{\f}
\]
where $\lambda\hS\frac{\partial}{\partial\hS}$ is the derivation sending $\hS^\beta$ to $(\lambda\cdot\beta)\hS^\beta$.
Moreover, it is expected the following \cite{IS}: 
\begin{itemize}
\item[$(1)$] The summation $\hat{\f}$ is, as a power series in $\hS$, supported on a strictly convex cone in $H_2^\K(X)$ (described by the GIT data).
\item[$(2)$] The transform $\f\mapsto\hat{\f}$ gives a map from the $\T$-equivariant Givental cone of $X$ to the $\T/\K$-equivariant Givental cone of $Y$.
\end{itemize}
Note that this is a straightforward generalization of \cite[Conjecture 1.7]{Iritani:blowups}.

We now consider the case that $X=V=\bigoplus_{i=1}^N V_i$ with the $\T$-action described in Section \ref{sec:toric_bundle}.
Using a very ample line bundle over $B$, one can easily check that $Y = \XLV$ can be realized as a smooth GIT quotient of $V$ by $\K$.
Via the canonical splitting $H_2^\K(V) = H_2(B)\oplus H_2^\K(\pt)$, $H_2^\K(\pt) = \LL$ can be interpreted as the subgroup of $H_2^\K(V)$.
We can explicitly compute the shift operator and the equivariant Kirwan map. 
For any character $\ell\in\Hom(\C^\times,\K)\cong\LL$ and any point $\f(\lambda,z)$ on $\cH_{V,\T}$, we have
\[
\hcS^\ell \f(\lambda,z) = \frac{ \f(\lambda-D(\ell) z,z) }{ \prod_{i=1}^N \prod_{c=1}^{-D_i(\ell)} \prod_{ \substack{\delta\colon\mathrm{Chern\ roots} \\ \mathrm{of\ } V_i} } (u_i + \delta + cz) }.
\]
The equivariant Kirwan map $\kappa_{\XLV}\colon H^*_\T(V)=H^*(B)[\lambda]\to H^*_{\T/\K}(\XLV)$ sends $\phi\in H^*(B)$ to its pull-back along the projection $\XLV\to B$, and sends $\lambda_i$ to $c_1^{\T/\K}(L_i)$, where $L_i$ is the $\T/\K$-line bundle \eqref{eqn:L_i} over $\XLV$.
By a direct computation, we can see that, for any point $I_V^\lambda$ on $\cL_{V,\T}\cap\cH_{V,\T}^{\pol}$, the Fourier transform of $I_V^\lambda$ actually coincides with the function $(I_V^\lambda)\sphat\ $ introduced in Theorem \ref{thm:intro_mirror}.

\bibliographystyle{amsplain}
\bibliography{reference}
\end{document}